\documentclass[11pt,a4paper,twoside,final]{scrartcl}
\usepackage{a4wide}
\usepackage{amsfonts}
\usepackage{amsmath}
\usepackage{amssymb}
\usepackage{amsthm}
\usepackage{algorithm}
\usepackage{algorithmic}
\usepackage{graphicx}
\usepackage{color}
\usepackage{colortbl}
\usepackage{showkeys}
\usepackage{enumerate}
\usepackage{subcaption}
\usepackage{verbatim}
\usepackage{bm}
\usepackage{hyperref}
\usepackage{pgfplots}
\pgfplotsset{compat=newest}
\usepackage{tikz}
\usepackage{tikz-cd}
\pgfplotsset{colormap={whitered}{color(0cm)=(white);
	color(1cm)=(orange!75!red)}
}

\newtheorem{theorem}{Theorem}[section]
\newtheorem{lemma}[theorem]{Lemma}
\newtheorem{remark}[theorem]{Remark}
\newtheorem{generalisation}[theorem]{Generalisation}
\newtheorem{definition}[theorem]{Definition}
\newtheorem{example}[theorem]{Example}
\newtheorem{corollary}[theorem]{Corollary}
\newtheorem{proposition}[theorem]{Proposition}

\makeatletter
\def\imod#1{\allowbreak\mkern10mu({\operator@font mod}\,\,#1)}
\makeatother

\makeatletter 
 
\@addtoreset{algorithm}{chapter} 
\makeatother

\numberwithin{equation}{section}
\numberwithin{table}{section}
\numberwithin{figure}{section}

\newcommand{\bend}{\hspace*{0ex} \hfill \hbox{\vrule height
    1.5ex\vbox{\hrule width 1.4ex \vskip 1.4ex\hrule  width 1.4ex}\vrule
    height 1.5ex}}

\long\def\symbolfootnote[#1]#2{\begingroup \def\thefootnote{\fnsymbol{footnote}}\footnote[#1]{#2}\endgroup}
\makeatletter
\let\@fnsymbol\@arabic
\makeatother

\title{Efficient multivariate approximation on the cube}

\date{}\author{
	Robert Nasdala\thanks{
    	Faculty of Mathematics, Chemnitz University of Technology, D-09107 Chemnitz, Germany.\newline E-mail:                \href{mailto:robert.nasdala@math.tu-chemnitz.de}{robert.nasdala@math.tu-chemnitz.de}
	}
	\and
	Daniel Potts\thanks{
		Faculty of Mathematics, Chemnitz University of Technology, D-09107 Chemnitz, Germany.\newline E-mail:                \href{mailto:potts@math.tu-chemnitz.de}{potts@math.tu-chemnitz.de}
	}
}

\hypersetup{pdfauthor={Robert Nasdala, Daniel Potts},
		pdftitle={Efficient multivariate approximation on the cube},
		pdfsubject={},
		pdfcreator = {pdflatex},
		plainpages=false,
		pdfstartview=FitH,       
		pdfview=FitH,            
		pdfpagemode=UseOutlines, 
		bookmarksnumbered=true, 
		bookmarksopen=false,     
		bookmarksopenlevel=0,   
		colorlinks=true,       
		linkcolor=black,
		citecolor=black,
		urlcolor=black}
\begin{document}

\maketitle

\begin{abstract}
	We combine a periodization strategy for weighted $L_{2}$-integrands with efficient approximation methods in order to approximate multivariate non-periodic functions on the high-dimensional cube $\left[-\frac{1}{2},\frac{1}{2}\right]^{d}$.
		Our concept allows to determine conditions on the $d$-variate torus-to-cube transformations ${\psi:\left[-\frac{1}{2},\frac{1}{2}\right]^{d}\to\left[-\frac{1}{2},\frac{1}{2}\right]^{d}}$ 
		such that a non-periodic function is transformed into a smooth function in the Sobolev space $\mathcal H^{m}(\mathbb{T}^{d})$ when applying $\psi$. 
		We adapt some $L_{\infty}(\mathbb{T}^{d})$- and $L_{2}(\mathbb{T}^{d})$-approximation error estimates for single rank-$1$ lattice approximation methods and adjust algorithms for the fast evaluation and fast reconstruction of multivariate trigonometric polynomials on the torus in order to apply these methods to the non-periodic setting. 
		We illustrate the theoretical findings by means of numerical tests in up to $d=5$ dimensions.
	
\medskip


\medskip

\end{abstract}

\medskip

\section{Introduction}
In this paper we discuss a general framework for the approximation of non-periodic multivariate functions $h$ on the $d$-dimensional cube $\left[-\frac{1}{2},\frac{1}{2}\right]^d$.
We combine a periodization strategy for weighted $L_{2}$-integrands with approximation methods based on \mbox{rank-$1$} lattices.
For dimension $d=1$, we consider increasing and $k$-times continuously differentiable transformations $\psi:\left[-\frac{1}{2},\frac{1}{2}\right]\to\left[-\frac{1}{2},\frac{1}{2}\right]$ whose derivatives vanish at the boundary points $\left\{-\frac{1}{2}, \frac{1}{2}\right\}$.
Applying such a change of variables $y = \psi(x)$ to any function $h\in L_{2}\left(\left[-\frac{1}{2},\frac{1}{2}\right],\omega\right)$ yields
\begin{align}\label{eq:change_of_variables_first}
\int_{-\frac{1}{2}}^{\frac{1}{2}} |h(y)|^2 \, \omega(y) \,\mathrm{d}y
	= \int_{-\frac{1}{2}}^{\frac{1}{2}} \left| h(\psi(x)) \right|^2 \, \omega(\psi(x)) \, \psi'(x) \, \mathrm{d}x
	= \int_{-\frac{1}{2}}^{\frac{1}{2}} \left| f(x) \right|^2 \, \mathrm{d}x
\end{align}
with $f(x) = h(\psi(x))\sqrt{\omega(\psi(x)) \, \psi'(x)}$. 
Considering dimensions $d\geq 2$ a multivariate generalization yields similar $d$-variate functions $f$.
We prove that the transformed function $f$ is continuously extendable on the torus $\mathbb{T}^d$ and has some guaranteed minimal degree of Sobolev-smoothness, provided that the non-periodic function $h$ has certain smoothness properties and, both the weight function $\omega$ and the transformation $\psi$ fulfill certain boundary conditions.
As a consequence of this analysis, we rewrite the involved objects, algorithms and approximation error bounds in terms of the inverse transformation $\psi^{-1}$.
This strategy converts the approximation of a function ${f\in L_2(\mathbb{T}^d)}$ by the Fourier partial sum ${S_{I}f := \sum_{\mathbf k\in I} \hat f_{\mathbf k}\,\mathrm{e}^{2\pi\mathrm i \mathbf k\cdot \circ}}$ translates into the approximation of a function $h\in L_2\left(\left[-\frac{1}{2},\frac{1}{2}\right]^d,\omega\right)$ by the transformed Fourier partial sum  $\sum_{\mathbf k\in I} \hat h_{\mathbf k}\, \varphi_{\mathbf k}$ based on the orthonormal system $\left\{ \varphi_{\mathbf k} :=  \sqrt{\frac{(\psi^{-1})'(\circ)}{\omega(\circ)}} \, \mathrm{e}^{2\pi\mathrm i \mathbf k\cdot\psi^{-1}(\circ)} \right\}_{\mathbf k\in\mathbb{Z}^d}$ and the Fourier coefficients of $h$ are defined as ${\hat h_{\mathbf k} := (h,\varphi_{\mathbf k})_{L_{2}\left(\left[-\frac{1}{2},\frac{1}{2}\right]^d,\omega\right)}}$.
Furthermore, the outlined periodization strategy allows to apply existing approximation methods for smooth functions defined on the torus $\mathbb T^d$.
More precisely, we adapt results coming from approximation theory in the Wiener Algebra $\mathcal{A}(\mathbb{T}^d)$, which contains all $L_{1}(\mathbb{T}^d)$-functions with absolutely summable Fourier coefficients $\hat f_{\mathbf k}$ and $\mathbf k = (k_1,\ldots,k_d)^{\top}\in\mathbb{Z}^d$, see \cite{Tem93,DuTeUl2018}.
We define the weight function 
\begin{align} \label{def:hyperbolic_cross_weight}
	\omega_{\mathrm{hc}}(\mathbf k) := \prod_{j=1}^{d} \max(1, |k_j|)
\end{align}
and, for $\beta\ge0$, we denote suitable subspaces of the Wiener Algebra $\mathcal{A}(\mathbb{T}^d)$ by
\begin{align}
	\label{def:Aalphaspace}
	\mathcal{A}^{\beta}(\mathbb T^d) 
	&:= \left\{ f \in L_{1}(\mathbb{T}^d) : \|f\|_{\mathcal{A}^{\beta}(\mathbb{T}^d)} := \sum_{\mathbf k\in\mathbb{Z}^d} \omega_{\mathrm{hc}}(\mathbf k)^{\beta} |\hat f_{\mathbf k}| < \infty \right\}.
\end{align}
Moreover, we consider the Hilbert spaces
\begin{align}
	\label{def:HbetaRaum}
	\mathcal{H}^{\beta}(\mathbb T^d) 
	&:= \left\{ f \in L_2(\mathbb{T}^d) : \|f\|_{\mathcal{H}^{\beta}(\mathbb{T}^d)} := \left( \sum_{\mathbf k\in\mathbb{Z}^d} \omega_{\mathrm{hc}}(\mathbf k)^{2\beta} |\hat f_{\mathbf k}|^2 \right)^{\frac{1}{2}} < \infty \right\}.
\end{align}
The norms defined above characterize the decay of the Fourier coefficients $\hat f_{\mathbf k}$ with respect to the weight function ${\omega_{\mathrm{hc}}}$ and the smoothness rate $\beta$.
Furthermore, we define hyperbolic cross index sets ${I_{N}^{d} := \{ \mathbf k\in\mathbb{Z}^d: \omega_{\mathrm{hc}}(\mathbf k)\leq N \} \subset \mathbb{Z}^d}$ with $N\in\mathbb{N}$, which contain in some sense the most important frequencies of functions in $\mathcal{A}^{\beta}(\mathbb T^d)$ and $\mathcal{H}^{\beta}(\mathbb T^d)$.
We apply existing error bounds on the approximation $S_{I_{N}^{d}}^{\Lambda}f := \sum_{\mathbf k\in I} \hat f_{\mathbf k}^{\Lambda}\,\mathrm{e}^{2\pi\mathrm i \mathbf k\cdot \circ}$ of periodic functions $f\in\mathcal{A}^{\beta}(\mathbb T^d)$ and $f\in\mathcal{H}^{\beta}(\mathbb T^d)$.
The used coefficients $\hat f_{\mathbf k}^{\Lambda}$ are just approximated Fourier coefficients $\hat f_{\mathbf k}^{\Lambda} \approx\hat f_{\mathbf k}$ that are computed based on samples along a suitable rank-1 lattice.
For continuous functions $f\in\mathcal{A}^{\beta}(\mathbb{T}^d)$, it was shown in \cite[Theorem~3.3]{KaPoVo13} that the $L_{\infty}$-approximation error $\left\|f - S_{I_{N}^{d}}^{\Lambda}f\right\|_{L_{\infty}(\mathbb{T}^d)}$ is bounded above by $N^{-\beta} \|f\|_{\mathcal{A}^{\beta}(\mathbb{T}^d)}$.
For continuous functions $f\in\mathcal{H}^{\beta}(\mathbb{T}^d)$, the same approximant causes the $L_2$-errors $\left\|f - S_{I_{N}^{d}}^{\Lambda}f\right\|_{L_{2}(\mathbb{T}^d)}$ to be bounded above by $C_{d,\beta} N^{-\beta} (\log N)^{(d-1)/2} \|f\|_{\mathcal{H}^{\beta}(\mathbb{T}^d)}$ with some constant $C_{d,\beta} = C(d,\beta) > 0$ when measured in the $L_{2}(\mathbb{T}^d)$-norm, as shown in \cite[Theorem~2.30]{volkmerdiss}.
The approximation of functions in Hilbert space $\mathcal{H}^{\beta}(\mathbb{T}^d)$ was investigated in \cite{Tem86} and later, more generally, in \cite{KaPoVo13}.

In general, it is hard to calculate the Fourier coefficients $\hat f_{\mathbf k}$ in order to determine if they are absolutely or square summable.
Instead we use certain norm equivalences to extract the information about the decay rate of the Fourier coefficients $\hat f_{\mathbf k}$. 
Given a multi-index $\bm{\alpha} = (\alpha_1, \ldots, \alpha_d)^{\top}\in\mathbb{N}_{0}^d$ with $\|\bm\alpha\|_{\ell_{\infty}} := \max(|\alpha_1|,\ldots,|\alpha_d|)$ and the domain ${\Omega \in \{ \mathbb{T}^d, \left[-\frac{1}{2},\frac{1}{2}\right]^d \}}$ we define the norm
\begin{align} \label{def:Hmix_norm}
	\|f\|_{H_{\mathrm{mix}}^{m}(\Omega)} 
	:= \left( \sum_{\|\bm\alpha\|_{\ell_{\infty}}\leq m} \| D^{\bm\alpha} [f] \|_{L_{2}(\Omega)}^{2} \right)^{1/2}
\end{align}
of the Sobolev space $H_{\mathrm{mix}}^{m}(\Omega)$ of functions $f\in L_{2}(\Omega)$ with mixed smoothness ${m\in\mathbb{N}_{0}}$, that has been studied in \cite{TriSchmei87,UllTDiss,VybiralDiss}.
Similarly, we define the norm
\begin{align} \label{def:Cmix_norm}
	\|f\|_{\mathcal{C}_{\mathrm{mix}}^{m}(\Omega)} := \sum_{\|\bm\alpha\|_{\ell_{\infty}}\leq m} \| D^{\bm\alpha} [f] \|_{L_{\infty}(\Omega)}
\end{align}
of the space $\mathcal{C}_{\mathrm{mix}}^{m}(\Omega)$ of functions with mixed continuous differentiability order $m\in\mathbb{N}_{0}$ as defined in \cite{TriSchmei87}.
The norms $\|\cdot\|_{H_{\mathrm{mix}}^{m}(\mathbb{T}^d)}$
and $\|\cdot\|_{\mathcal{H}^{\beta}(\mathbb{T}^d)}$ are equivalent for ${\beta = m \in\mathbb{N}}$ as shown in \cite[Lemma~2.3]{KuSiUl15}.
Furthermore, for all $\beta \geq 0$ and all $\lambda>\frac{1}{2}$, we have the continuous embedding ${\mathcal{H}^{\beta+\lambda}(\mathbb{T}^d) \hookrightarrow \mathcal{A}^{\beta}(\mathbb{T}^d)}$ as shown in \cite[Lemma~2.2]{KaPoVo13}.

The goal is to guarantee that a function $f$ obtained by the change of variables~\eqref{eq:change_of_variables_first} is in $\mathcal{A}^{m}(\mathbb{T}^d)$ or $\mathcal H^{m}(\mathbb{T}^d)$, for which we provide a set of sufficient $L_{\infty}$-conditions.
These are easier to check than estimating either one of the equivalent norms $\|\cdot\|_{H_{\mathrm{mix}}^{m}(\mathbb{T}^d)} \sim \|\cdot\|_{\mathcal{H}^{m}(\mathbb{T}^d)}$.
At first we prove these conditions for all possible transformations $\psi$ and weight functions $\omega$. 
Later on we use families of parameterized transformations ${\psi(\circ) = \psi(\circ, \bm\eta)}$ and families of weight functions ${\omega(\circ) = \omega(\circ, \bm\mu)}$ with $\bm\eta,\bm\mu\in\mathbb{R}_{+}^d$.
Consequently, we end up with parameterized transformed functions ${f(\circ) = f(\circ,\bm\eta,\bm\mu) \in L_{2}(\mathbb{T}^d)}$ and both parameter vectors may affect the smoothness of these functions.
Based on these sufficient $L_{\infty}$-smoothness conditions, we calculate sufficient lower bounds on $\bm\eta$ and $\bm\mu$ such that the smoothness degree $m$ of a function ${h\in L_{2}\left(\left[-\frac{1}{2},\frac{1}{2}\right]^d,\omega(\circ,\bm\mu)\right) \cap \mathcal{C}_{\mathrm{mix}}^{m}\left(\left[-\frac{1}{2},\frac{1}{2}\right]^d\right)}$ 
remains the same under the composition with a family of transformations $\psi(\circ,\bm\eta)$ so that we end up with ${f\in \mathcal H^{m}(\mathbb{T}^d)}$.

This preservation of smoothness allows for transferring well-known results from the torus $\mathbb{T}^d$ to the non-periodic setting on the cube $\left[-\frac{1}{2},\frac{1}{2}\right]^d$ by means of the inverse transformation $\psi^{-1}$, which includes:
\begin{itemize}
	\item reasonable $L_{2}$- and $L_{\infty}$-approximation results from \cite{KaPoVo13,ByKaUlVo16},
	\item highly efficient algorithms for high-dimensional approximation based on rank-$1$ lattice sampling \cite{KaPoVo13,kaemmererdiss}.
\end{itemize}
In general, the change of variables is a versatile and powerful tool in numerical analysis.
We recommend the excellent overview in \cite[Chapter 16 and 17]{boyd00}, where many practical aspects of the described methods are discussed. 
In recent years these periodization approaches were repeatedly used for the numerical integration and approximation of non-periodic functions in Chebyshev spaces \cite{PoVo15}, as well as in half-periodic cosine spaces and in Korobov spaces by means of tent-transformed lattice rules \cite{DiNuPi14,CoKuNuSu16,GoSzYo19,KuoMiNoNu19}.
Specific strategies to periodize integrands have been discussed for numerical integration in \cite{KuoSlWo07}.
Besides single and multiple rank-$1$ lattice rules \cite{KaPoVo13,kaemmererdiss}, there are several other sampling strategies for periodic signals such as sparse grids \cite{GriHa13,ByDuSiUl14,GriHa19}, randomized least square sampling approaches \cite{KaUlVo19,KrUl19} and also interlaced scrambled polynomial lattice rules \cite{GoDi15,DiGoSuYo17}.
However, we focus on sinlge rank-$1$ lattice based methods in this paper. 
An introduction to numerical integration based on lattice rules can be found in \cite{Nie78,SlJo94,JoKuSl13}. 
These cubature rules can also be used for the approximation of functions on the torus \cite{Tem93}. 
Efficient algorithms based on component-by-component methods \cite{NuCo04,CoKuNu10} were presented to compute high-dimensional integrals.
For the approximation of high-dimensional functions there are efficient algorithms using sampling schemes based on {rank-$1$} lattices \cite{KaPoVo13,kaemmererdiss}.
Moreover, there exist strategies that allow for the efficient approximation of high-dim functions based on rank-$1$ lattice sampling.
These approaches provide reasonable approximation properties \cite{ByKaUlVo16,KuoMiNoNu19}. 
We adjust the algorithms for the non-periodic setting and incorporate the outlined transformations. 
In addition to the theoretical considerations, we adapt these algorithms for the non-periodic setting and apply the results in order to present numerical examples.

The outline of the paper is as follows:
In Section~2 we provide the basic notions from classical Fourier approximation theory on the torus $\mathbb T^d$, the corresponding function spaces and important convergence properties. 
We introduce the Sobolev spaces $H_{\mathrm{mix}}^{m}(\mathbb T^d)$ of mixed natural smoothness order $m\in\mathbb{N}_0$ and the Wiener Algebra $\mathcal{A}(\mathbb T^d)$ of functions with absolutely summable Fourier coefficients.
We discuss certain properties of the subspaces $\mathcal{A}^{\beta}(\mathbb T^d)$ and $\mathcal{H}^{\beta}(\mathbb T^d)$ of the Wiener Algebra
and we highlight the norm equivalence of $\|\cdot\|_{\mathcal{H}^{m}(\mathbb{T}^d)}$ and $\|\cdot\|_{H_{\mathrm{mix}}^{m}(\mathbb{T}^d)}$ for all $m\in\mathbb{N}$, see \cite{KuSiUl15}.
Subsequently, we define {rank-$1$} lattices as introduced in \cite{Ko59},
discuss their importance in the context of Fourier approximation 
and recall two important approximation error bounds on the torus in Theorems~\ref{thm:L_infty_approx_error_torus} and \ref{eq:L_2_approximation_error_bound}.
In Section~3 we define the notion of a torus-to-cube transformation $\psi:\left[-\frac{1}{2},\frac{1}{2}\right]^d \to \left[-\frac{1}{2},\frac{1}{2}\right]^d$ and compare some examples.
Then, we introduce weight functions ${\omega:\left[-\frac{1}{2},\frac{1}{2}\right]^d\to[0,\infty)}$ and
specify the structure of the weighted Hilbert spaces $L_2\left(\left[-\frac{1}{2},\frac{1}{2}\right]^d,\omega\right)$, the corresponding weighted scalar product $(\cdot, \cdot)_{L_2\left(\left[-\frac{1}{2},\frac{1}{2}\right]^d,\omega\right)}$ and the Fourier coefficients $\hat h_{\mathbf k}$.
Afterwards, we prove sufficient $L_{\infty}$-conditions on the transformation $\psi$ and weight function $\omega$ for a function 
$h\in L_2\left(\left[-\frac{1}{2},\frac{1}{2}\right]^d,\omega\right)\cap \mathcal{C}_{\mathrm{mix}}^{m}\left(\left[-\frac{1}{2},\frac{1}{2}\right]^d\right)$ to be transformed into a smooth function $f \in \mathcal H^{m}(\mathbb{T}^d)$ by the composition with $\psi$.
In Theorems~\ref{thm:L_infty_approx_error_multivar} and \ref{thm:Hm_approx_error_decay_multivar} we prove upper bounds on the approximation error $\left\|h - S_{I_{N}^{d}}^{\Lambda}h\right\|$ measured in weighted $L_{2}$- and $L_{\infty}$-norms on $\left[-\frac{1}{2},\frac{1}{2}\right]^d$ that are based on the error bounds that were given in Section~2.
In Section~4 we incorporate the transformations $\psi$ into \cite[Algorithm~3.1 and 3.2]{kaemmererdiss}, which yields for the efficient realization of the evaluation and the reconstruction of multivariate non-periodic functions, as outlined in the Algorithms~\ref{alg:LFFT_eval} and \ref{alg:LFFT_recon}.
In Section~5 we illustrate the theoretical results by several numerical tests that use the logarithmic transformation \eqref{eq:logarithmic_trafo_comb} and the sine transformation \eqref{eq:sine_trafo}.
For univariate and multivariate test functions $h$ we use a constant weight function ${\omega\equiv 1}$.
Based on the sufficient $L_{\infty}$-conditions presented in Section~3, we calculate explicit bounds for $\bm\eta\in\mathbb{R}_{+}^d$ that preserve the degree of smoothness $m\in\mathbb{N}$ when composing the function $h$ with a family of transformations ${\psi(\circ,\bm\eta)}$.
Afterwards we apply algorithms of the previous section in order to approximate specific high-dimensional functions in up to $d=5$ dimensions and we discuss the results.

\section{Fourier approximation}
We introduce weighted $L_{2}$-function spaces and Sobolev spaces of mixed smoothness,
recall some definitions of classical Fourier approximation theory
and define a space of functions that have absolute square-summable Fourier coefficients.
We also reflect the ideas of {rank-$1$} lattices \cite{SLKA87,CoKuNu10,kaemmererdiss}, the corresponding Fourier approximation methods and approximation error bounds \cite{Tem86,KaPoVo13,ByKaUlVo16}.

\subsection{Preliminaries}
Let $\Omega \in \left\{ \mathbb{T}^d, \left[-\frac{1}{2},\frac{1}{2}\right]^d \right\}$ with $\mathbb T^d \simeq [-\frac{1}{2},\frac{1}{2})^d$ being the $d$-dimensional torus. 
The space ${\left(\mathcal{C}(\Omega), \|\cdot\|_{L_{\infty}(\Omega)}\right)}$ denotes the collection of all continuous multivariate functions ${f:\Omega \to \mathbb{C}}$, 
and $\left(\mathcal{C}_{0}\left(\left[-\frac{1}{2},\frac{1}{2}\right]^d\right), \|\cdot\|_{L_{\infty}\left(\left[-\frac{1}{2},\frac{1}{2}\right]^d\right)}\right)$ denotes the space of all continuous functions that vanish at the boundary points $[-\frac{1}{2},\frac{1}{2}]^d\setminus (-\frac{1}{2},\frac{1}{2})^d$.
For the multi-indices $\bm\alpha \in\mathbb{N}_{0}^d$ we define the differential operator 
\begin{align*} 
	D^{\bm \alpha}[f](\mathbf x) = D^{(\alpha_1,\ldots,\alpha_d)}[f](x_1,\ldots,x_d) := \frac{\partial^{\alpha_1}}{\partial x_{1}^{\alpha_1}}\ldots\frac{\partial^{\alpha_d}}{\partial x_{d}^{\alpha_d}} [f](x_1,\ldots,x_d).
\end{align*}
We define the \emph{function space of mixed continuous differentiability} of order $m\in\mathbb{N}$, see \cite[page 132]{TriSchmei87}, as
\begin{align*}
	\mathcal{C}_{\mathrm{mix}}^{m}(\Omega)
	:= \left\{ f\in\mathcal{C}(\Omega) : \|f\|_{\mathcal{C}_{\mathrm{mix}}^{m}(\Omega)} < \infty \right\}
\end{align*}
with $\|\cdot\|_{\mathcal{C}_{\mathrm{mix}}^{m}(\Omega)}$ given in \eqref{def:Cmix_norm}.
The corresponding univariate function spaces are denoted by $\mathcal{C}^{m}(\Omega)$.

The weighted function spaces $L_2(\Omega, \omega)$ with an integrable weight function ${\omega:\Omega\to[0,\infty)}$ are defined as
\begin{align}
	\label{def:weighted_L2_space}
	L_2(\Omega, \omega) := \left\{ h\in L_2(\Omega) : \|h\|_{L_2(\Omega, \omega)} := \left( \int_{\Omega} |h(\mathbf x)|^2 \, \omega(\mathbf x) \,\mathrm{d}\mathbf x \right)^{\frac{1}{2}} < \infty \right\}.
\end{align}
For the constant weight function $\omega(\mathbf x) \equiv 1$ we have $L_2(\Omega, \omega) = L_{2}(\Omega)$.
For functions $f$ and $g$ in the Hilbert space $L_2(\mathbb T^d)$ we have the scalar product
\begin{align*}
	(f,g)_{L_2(\mathbb T^d)} := \int_{\mathbb T^d} f(\mathbf x)\, \overline{g(\mathbf x)}\,\mathrm d\mathbf x .
\end{align*}
For any frequency set $I \subset \mathbb{Z}^d$ of finite cardinality $|I|<\infty$ we denote the space of all multivariate trigonometric polynomials supported on $I$ by
\begin{align*} 
	\Pi_{I} := \mathrm{span}\{ \mathrm{e}^{2\pi\mathrm i \mathbf k \cdot \circ} : \mathbf k\in I\}.
\end{align*}
The functions $\mathrm{e}^{2\pi\mathrm i \mathbf k \cdot\mathbf x} = \prod_{j=1}^{d}\mathrm{e}^{2\pi\mathrm i k_j x_j}$ with ${\mathbf{k}\in \mathbb{Z}^d}$ and ${\mathbf x \in\mathbb T^d}$ 
are orthogonal with respect to the $L_2(\mathbb T^d)$-scalar product.
For all $\mathbf k\in \mathbb{Z}^d$ we denote the \textsl{Fourier coefficients} $\hat f_{\mathbf k}$ by
\begin{align*}
	\hat f_{\mathbf k} 
	= (f, \mathrm{e}^{2\pi\mathrm i \mathbf k \cdot \circ})_{L_2(\mathbb T^d)} 
	= \int_{\mathbb T^d} f(\mathbf x)\,\mathrm{e}^{-2\pi\mathrm i \mathbf k \cdot \mathbf x} \,\mathrm d\mathbf x,
\end{align*}
and the corresponding \textsl{Fourier partial sum} by $S_{I}f(\mathbf x) = \sum_{\mathbf k\in I} \hat f_{\mathbf k}\,\mathrm{e}^{2\pi\mathrm i \mathbf k\mathbf \cdot \mathbf x}$.
For all ${f \in L_2(\mathbb T^d)}$ we have
\begin{align*}
	\|f - S_{I}f\|_{L_{2}(\mathbb T^d)} \to 0 \quad \text{for} \quad |I|\to\infty,
\end{align*} 
where $|I|\to\infty$ means $\min(|k_1|, \ldots, |k_d|)\to\infty$ for $\mathbf k = (k_1,\ldots,k_d)^{\top}\in I$, see \cite[Theorem~4.1]{We12}.
Finally, we define the \emph{Sobolev spaces of mixed natural smoothness} of $L_2(\Omega)$-functions with smoothness order $m\in\mathbb{N}_{0}$, see \cite{TriSchmei87,UllTDiss,VybiralDiss}, as
\begin{align*} 
	H_{\mathrm{mix}}^{m}(\Omega)
	:= \left\{ f\in L_2(\Omega) : \|f\|_{H_{\mathrm{mix}}^{m}(\Omega)} < \infty \right\}
\end{align*}
with $\|\cdot\|_{H_{\mathrm{mix}}^{m}(\Omega)}$ being given in \eqref{def:Hmix_norm}.
The corresponding univariate spaces are denoted by $H^{m}(\Omega)$.
Based on the weight function $\omega_{\mathrm{hc}}(\mathbf k)$ given in \eqref{def:hyperbolic_cross_weight} we define the \textsl{hyperbolic crosses} $I_{N}^{d}$ as
\begin{align} \label{def:hyperbolic_cross}
	I_{N}^{d}
	:= \left\{ \mathbf k \in\mathbb{Z}^d : \omega_{\mathrm{hc}}(\mathbf k) \leq N \right\},
\end{align}
illustrated for $N=16$ in two dimensions in Figure~\ref{fig:HyperbolicCrosses}. 
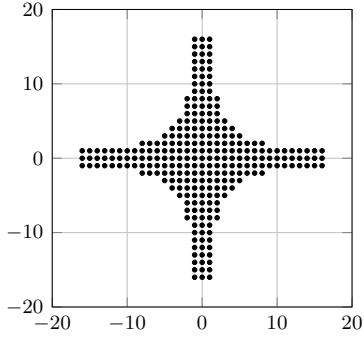
\begin{figure}[t]
	\centering
\begin{tikzpicture}[scale=0.85]
	\begin{axis}[scatter/classes = { a = {mark=o, draw=black} },
	font=\footnotesize,
	grid = both,
	xmax = 20, xmin = -20, ymax = 20, ymin = -20,
	height=0.4\textwidth, width=0.4\textwidth,
unit vector ratio*=1 1 1
	]
	\addplot[scatter ,only marks, mark size=1, scatter src = explicit symbolic]coordinates{
		(-16,-1)(-16,0)(-16,1)(-15,-1)(-15,0)(-15,1)(-14,-1)(-14,0)(-14,1)(-13,-1)(-13,0)(-13,1)(-12,-1)(-12,0)(-12,1)(-11,-1)(-11,0)(-11,1)(-10,-1)(-10,0)(-10,1)(-9,-1)(-9,0)(-9,1)(-8,-2)(-8,-1)(-8,0)(-8,1)(-8,2)(-7,-2)(-7,-1)(-7,0)(-7,1)(-7,2)(-6,-2)(-6,-1)(-6,0)(-6,1)(-6,2)(-5,-3)(-5,-2)(-5,-1)(-5,0)(-5,1)(-5,2)(-5,3)(-4,-4)(-4,-3)(-4,-2)(-4,-1)(-4,0)(-4,1)(-4,2)(-4,3)(-4,4)(-3,-5)(-3,-4)(-3,-3)(-3,-2)(-3,-1)(-3,0)(-3,1)(-3,2)(-3,3)(-3,4)(-3,5)(-2,-8)(-2,-7)(-2,-6)(-2,-5)(-2,-4)(-2,-3)(-2,-2)(-2,-1)(-2,0)(-2,1)(-2,2)(-2,3)(-2,4)(-2,5)(-2,6)(-2,7)(-2,8)(-1,-16)(-1,-15)(-1,-14)(-1,-13)(-1,-12)(-1,-11)(-1,-10)(-1,-9)(-1,-8)(-1,-7)(-1,-6)(-1,-5)(-1,-4)(-1,-3)(-1,-2)(-1,-1)(-1,0)(-1,1)(-1,2)(-1,3)(-1,4)(-1,5)(-1,6)(-1,7)(-1,8)(-1,9)(-1,10)(-1,11)(-1,12)(-1,13)(-1,14)(-1,15)(-1,16)(0,-16)(0,-15)(0,-14)(0,-13)(0,-12)(0,-11)(0,-10)(0,-9)(0,-8)(0,-7)(0,-6)(0,-5)(0,-4)(0,-3)(0,-2)(0,-1)(0,0)(0,1)(0,2)(0,3)(0,4)(0,5)(0,6)(0,7)(0,8)(0,9)(0,10)(0,11)(0,12)(0,13)(0,14)(0,15)(0,16)(1,-16)(1,-15)(1,-14)(1,-13)(1,-12)(1,-11)(1,-10)(1,-9)(1,-8)(1,-7)(1,-6)(1,-5)(1,-4)(1,-3)(1,-2)(1,-1)(1,0)(1,1)(1,2)(1,3)(1,4)(1,5)(1,6)(1,7)(1,8)(1,9)(1,10)(1,11)(1,12)(1,13)(1,14)(1,15)(1,16)(2,-8)(2,-7)(2,-6)(2,-5)(2,-4)(2,-3)(2,-2)(2,-1)(2,0)(2,1)(2,2)(2,3)(2,4)(2,5)(2,6)(2,7)(2,8)(3,-5)(3,-4)(3,-3)(3,-2)(3,-1)(3,0)(3,1)(3,2)(3,3)(3,4)(3,5)(4,-4)(4,-3)(4,-2)(4,-1)(4,0)(4,1)(4,2)(4,3)(4,4)(5,-3)(5,-2)(5,-1)(5,0)(5,1)(5,2)(5,3)(6,-2)(6,-1)(6,0)(6,1)(6,2)(7,-2)(7,-1)(7,0)(7,1)(7,2)(8,-2)(8,-1)(8,0)(8,1)(8,2)(9,-1)(9,0)(9,1)(10,-1)(10,0)(10,1)(11,-1)(11,0)(11,1)(12,-1)(12,0)(12,1)(13,-1)(13,0)(13,1)(14,-1)(14,0)(14,1)(15,-1)(15,0)(15,1)(16,-1)(16,0)(16,1) 
	};
	\end{axis}
	\end{tikzpicture}
	\caption{The hyperbolic cross $I_{N}^{d}$ for $N=16$ and $d=2$.}
	\label{fig:HyperbolicCrosses}
\end{figure}
Furthermore, for $\beta \geq 0$ there are the Hilbert space $\mathcal{H}^{\beta}(\mathbb{T}^d)$ consisting of functions $f \in L_2(\mathbb{T}^d)$ with absolutely square-summable weighted Fourier coefficients $\omega_{\mathrm{hc}}(\mathbf k)\hat f_{\mathbf k}$, as defined in \eqref{def:HbetaRaum}.
For all $m\in\mathbb{N}$ it was shown in \cite{KuSiUl15} that
\begin{align} \label{eq:Hs_norm_equivalence}
	\|\cdot\|_{\mathcal{H}^{m}(\mathbb{T}^d)} 
	\sim \|\cdot\|_{H_{\mathrm{mix}}^{m}(\mathbb{T}^d)}.
\end{align}
Closely related are the function spaces $\mathcal{A}^{\beta}(\mathbb T^d), \beta \geq 0$ of $L_1(\mathbb T^d)$-functions with absolutely summable Fourier coefficients, as defined in \eqref{def:Aalphaspace}.
For $\beta = 0$ and the constant weight function $\omega_{\mathrm{hc}}(\mathbf{k}) \equiv 1$ we call the space $\mathcal{A}(\mathbb T^d) := \mathcal{A}^{0}(\mathbb T^d)$ the \emph{Wiener Algebra}.
As shown in \cite[Lemma~2.2]{KaPoVo13}, for $\beta\geq 0, \lambda>\frac{1}{2}$ and fixed $d\in\mathbb{N}$ there are the continuous embeddings
\begin{align}
	\label{eq:Wiener_algebra_inclusion}
	\mathcal{H}^{\beta+\lambda}(\mathbb{T}^d) \hookrightarrow
	\mathcal{A}^{\beta}(\mathbb T^d) \hookrightarrow
	\mathcal{A}(\mathbb T^d)
\end{align}
and for $f\in\mathcal{A}^{\beta}(\mathbb T^d)$ we have 
\begin{align}
	\label{eq:Wiener_algebra_inclusion2}
	\|f\|_{\mathcal{A}^{\beta}(\mathbb T^d)} \leq C_{d,\lambda} \|f\|_{\mathcal{H}^{\beta+\lambda}(\mathbb{T}^d)}
\end{align}
with a constant $C_{d,\lambda} := C(d,\lambda) > 1$.
Additionally, for each function in $\mathcal{A}(\mathbb T^d)$ there exists a continuous representative, as proven in \cite[Lemma~2.1]{kaemmererdiss}.
Later on, when we sample functions $f\in\mathcal{H}^{\beta+\lambda}(\mathbb T^d)$ we identify them with their continuous representatives given by their Fourier series $\sum_{\mathbf k\in \mathbb{Z}^d} \hat f_{\mathbf k}\,\mathrm{e}^{2\pi\mathrm i \mathbf k\mathbf \cdot \circ}$
and this identification will be denoted by ${f\in\mathcal{H}^{\beta+\lambda}(\mathbb{T}^d)\cap\mathcal{C}(\mathbb{T}^d)}$.

\subsection{{Rank-1} lattices and reconstructing {rank-1} lattices}
We recollect some objects and observations from \cite{SLKA87,CoKuNu10,kaemmererdiss} to discuss the approximation of functions ${f\in\mathcal{H}^{\beta}(\mathbb T^d)\cap\mathcal{C}(\mathbb{T}^d)}$.
For each frequency set $I \subset \mathbb{Z}^d$ there is the \textsl{difference set}
\begin{align*}
	\mathcal{D}(I) &:= \{\mathbf{k} \in \mathbb{Z}^d : \mathbf{k}=\mathbf{k}_1-\mathbf{k}_2 \text{ with } \mathbf{k}_1,\mathbf{k}_2 \in I \}.
\end{align*}
The set
\begin{align}
	\label{def:rank_one_lattice}
	\Lambda(\mathbf{z}, M) &:= \left\{ \mathbf{x}_j := \left(\frac{j}{M}\,\mathbf z \bmod \mathbf 1\right) \in \mathbb{T}^d: j = 0,1,\ldots M-1 \right\}
\end{align}
is called \textsl{{rank-$1$} lattice} with the \textsl{generating vector} $\mathbf{z}\in\mathbb{Z}^d$ and the \textsl{lattice size} $M \in \mathbb{N}$, where $\mathbf 1 := \left(1,\ldots,1\right)^{\top}\in\mathbb{Z}^d$.
A \textsl{reconstructing {rank-$1$} lattice} $\Lambda(\mathbf{z}, M, I)$ is a {rank-$1$} lattice $\Lambda(\mathbf{z}, M)$ for which the condition
\begin{align*}
\mathbf{t} \cdot \mathbf{z} \not\equiv 0 \,(\bmod{M}) \quad \text{for all } \mathbf{t}\in\mathcal{D}(I)\setminus\{\mathbf 0\} 
\end{align*}
holds.
Given a reconstructing {rank-$1$} lattice $\Lambda(\mathbf z,M,I)$ we have exact integration for all multivariate trigonometric polynomials $g \in \Pi_{\mathcal{D}(I)}$, see \cite{SLKA87}, so that
\begin{align*}
	\int_{\mathbb{T}^d} g(\mathbf x) \,\mathrm{d}\mathbf x 
	= \frac{1}{M} \sum_{j =0}^{M-1} g(\mathbf x_j), 
	\quad \mathbf x_j\in\Lambda(\mathbf z,M,I).
\end{align*}
In particular, for $f \in \Pi_{I}$ and $\mathbf k\in I$ we have $f(\circ)\,\mathrm{e}^{-2\pi\mathrm{i}\mathbf k\cdot\circ} \in \Pi_{\mathcal{D}(I)}$ and
\begin{align} 
	\label{eq:exact_integration_formula}
	\hat f_{\mathbf k}
	= \int_{\mathbb{T}^d} f(\mathbf x) \,\mathrm{e}^{-2\pi\mathrm{i}\mathbf k\cdot\mathbf x} \,\mathrm{d}\mathbf x 
	= \frac{1}{M} \sum_{j =0}^{M-1} f(\mathbf x_j) \,\mathrm{e}^{-2\pi\mathrm{i}\mathbf k\cdot\mathbf x_j}, 
	\quad \mathbf x_j\in\Lambda(\mathbf z,M,I).
\end{align}
For an arbitrary function ${f \in \mathcal{H}^{\beta}(\mathbb T^d)\cap\mathcal{C}(\mathbb{T}^d)}$ we lose the former mentioned exactness and define the \textsl{approximated Fourier coefficients} $\hat f_{\mathbf k}^\Lambda$ of the form
\begin{align*} \hat f_{\mathbf k}
	&\approx \hat f_{\mathbf k}^\Lambda
	:= \frac{1}{M} \sum_{j =0}^{M-1} f(\mathbf x_j)\,\mathrm{e}^{-2\pi\mathrm i \mathbf k\cdot\mathbf x_j},
	\quad\mathbf x_j \in \Lambda(\mathbf z,M,I),
\end{align*}
leading to the \textsl{approximated Fourier partial sum} $S_{I}^{\Lambda} f$ given by
\begin{align*}
	S_{I} f(\mathbf x)
	&\approx S_{I}^{\Lambda} f(\mathbf x)
	:= \sum_{\mathbf k\in I} \hat f_{\mathbf k}^\Lambda \,\mathrm{e}^{2\pi\mathrm i \mathbf k\cdot\mathbf x}. \end{align*}

\subsection{Lattice based approximation on the torus}
We reflect upper bounds for certain approximation errors $\left\|f - S_{I_N^d}^{\Lambda} f\right\|$ of functions $f$ in ${\mathcal{A}^{\beta}(\mathbb{T}^d)\cap\mathcal{C}(\mathbb{T}^d)}$ and ${\mathcal{H}^{\beta}(\mathbb{T}^d)\cap\mathcal{C}(\mathbb{T}^d)}$. 
For this matter, the existence of reconstructing {rank-$1$} lattices is secured by the arguments provided in \cite[Corollary~1]{Kae2013} and \cite[Theorem~2.1]{KaPoVo13}.
\begin{remark}
	We note that techniques using multiple {rank-$1$} lattices were recently suggested in \cite{Kae16,Kae17} and methods for the dimension incremental construction of unknown frequency set $I\subset\mathbb{Z}^d$ are presented in \cite{PoVo14}. 
	Furthermore, there is a dimensional incremental support identification technique based on randomly chosen sampling points that was recently developed in \cite{ChIwKr18}.
	Even though the transformation method is easily incorporated into both the multiple {rank-$1$} lattice methods as well as the component-by-component construction method, they won't be discussed any further in this work.
\end{remark}
Now, it's possible to prove an upper error bound for the $L_{\infty}$-approximation of functions in the subspace $\mathcal{A}^{\beta}(\mathbb{T}^d)$ of the Wiener Algebra, as seen in \cite[Theorem~3.3]{KaPoVo13}:
\begin{theorem} \label{thm:L_infty_approx_error_torus}
	Let $f \in \mathcal{A}^{\beta}(\mathbb{T}^d)\cap\mathcal{C}(\mathbb{T}^d)$ with $\beta\geq 0$ and $d\in\mathbb{N}$, 
	a hyperbolic cross $I_N^d$ with ${|I_N^d|<\infty}$ and $N\in\mathbb{N}$,
	and a reconstructing {rank-$1$} lattice ${\Lambda(\mathbf{z}, M, I_N^d)}$ be given.
	The approximation of $f$ by the approximated Fourier partial sum $S_{I_N^d}^{\Lambda} f$
	leads to an approximation error that is estimated by
	\begin{align} \label{eq:torus_infty_approx_bound}
		\left\| f - S_{I_N^d}^{\Lambda} f \right\|_{L_{\infty}(\mathbb{T}^d)}
		\leq 2 N^{-\beta} \|f\|_{\mathcal{A}^{\beta}(\mathbb{T}^d)}.
	\end{align}
\end{theorem}
The approximation of functions in the Hilbert spaces $\mathcal{H}^{\beta}(\mathbb{T}^d)$ was investigated in \cite{Tem86} and later, more generally, in \cite{KaPoVo13}. 
It was shown that for all $\beta > 1$ there exists a reconstructing {rank-$1$} lattice generated by a vector of Korobov form $\mathbf z := (1,z,z^2,\ldots,z^{d-1})^{\top}\in\mathbb{Z}^d$ such that the $L_2$-truncation error is bounded above by
\begin{align*}
	\left\| f - S_{I_{N}^{d}}^{\Lambda}f \right\|_{L_{2}(\mathbb{T}^d)} 
	\leq N^{-\beta} (\log N)^{(d-1)/2} \|f\|_{\mathcal{H}^{\beta}(\mathbb{T}^d)}.
\end{align*}
A generalization of this estimate as well as an upper bound for the corresponding aliasing error can be found in \cite[Theorem~2]{ByKaUlVo16}, where dyadic hyperbolic cross frequency sets are used.
Furthermore, a component-by-component approach was applied to construct the generating vector $\mathbf z\in\mathbb{Z}^d$ which generally isn't of Korobov form anymore.
However, every dyadic hyperbolic cross is embedded in a non-dyadic one, see \cite[Lemma~2.29]{volkmerdiss}. 
Thus, the error estimates are easily translated in terms of non-dyadic hyperbolic crosses $I_{N}^{d}$, see \cite[Theorem~2.30]{volkmerdiss}.
We are interested in the following special case:
\begin{theorem}\label{eq:L_2_approximation_error_bound}
	Let $\beta > \frac{1}{2}$,
	$d\in\mathbb{N}$,
	$f\in\mathcal{H}^{\beta}(\mathbb{T}^d)\cap\mathcal{C}(\mathbb{T}^d)$,
	a hyperbolic cross $I_{N}^{d}$ with $N\geq 2^{d+1}$,
	and a reconstructing {rank-$1$} lattice $\Lambda(\mathbf z, M, I_{N}^{d})$ be given.
	Then we have
	\begin{align} \label{eq:H_beta_error_bound}
		\left\| f - S_{I_{N}^{d}}^{\Lambda}f \right\|_{L_{2}(\mathbb{T}^d)} 
		\leq C_{d,\beta} N^{-\beta} (\log N)^{(d-1)/2} \|f\|_{\mathcal{H}^{\beta}(\mathbb{T}^d)}
	\end{align}
	with some constant $C_{d,\beta} := C(d,\beta)>0$.
\end{theorem}

\section{Torus-to-cube transformation mappings}
Changes of variables were discussed for example in \cite{boyd00,ShTaWa11} and were used for high-dimensional integration \cite{KuWaWa06,KPPW18}.
We define torus-to-cube transformations ${\psi:\left[-\frac{1}{2},\frac{1}{2}\right]^d\to\left[-\frac{1}{2},\frac{1}{2}\right]^d}$ and compare some examples.
We introduce parameterized families of such torus-to-interval transformations, some of which are induced by invertible transformations ${\tilde{\psi}:(-\frac{1}{2},\frac{1}{2})^d\to\mathbb{R}^d}$ that were discussed in \cite{boyd00,ShTaWa11,NaPo18}.
We provide important examples that will reappear later in this paper. 
Afterwards, we introduce the weighted Hilbert spaces $L_{2}\left(\left[-\frac{1}{2},\frac{1}{2}\right]^d,\omega\right)$ with the integrable weight function $\omega:\left[-\frac{1}{2},\frac{1}{2}\right]^d\to[0,\infty)$ and investigate their structure.
Subsequently, we prove sufficient $L_{\infty}$-conditions on $\psi$ and $\omega$ which guarantee that the composition of a test function 
$h\in L_2\left(\left[-\frac{1}{2},\frac{1}{2}\right]^d,\omega\right)\cap \mathcal{C}_{\mathrm{mix}}^{m}\left(\left[-\frac{1}{2},\frac{1}{2}\right]^d\right)$ with a transformation $\psi$ 
yields a smooth function $f \in \mathcal H^{m}(\mathbb{T}^d)$. 
We transfer the already established error bounds with respect to the $L_{\infty}(\mathbb{T}^d)$- and the $L_{2}(\mathbb{T}^d)$-norms recalled in Theorems~\ref{thm:L_infty_approx_error_torus} and \ref{eq:L_2_approximation_error_bound} by means of the transformation $\psi$ and obtain specific approximation error bounds for function $h$ defined on the cube $\left[-\frac{1}{2},\frac{1}{2}\right]^d$.

\subsection{Torus-to-cube transformations}
We call a mapping
\begin{align}\label{def:Trafo_cube}
	\psi:\left[-\frac{1}{2},\frac{1}{2}\right] \to \left[-\frac{1}{2},\frac{1}{2}\right]
	\quad \text{with} \quad 
	\lim\limits_{x \to \pm\frac{1}{2}}\psi(x) = \pm\frac{1}{2}
\end{align}
a \textsl{torus-to-cube transformation} if it is continuously differentiable, increasing and has the first derivative $\psi'(x) := \frac{\mathrm{d}}{\mathrm{d}x}[\psi](x) \in\mathcal{C}(\mathbb{T})$.
Its inverse transformation is also continuously differentiable, increasing and is denoted by $\psi^{-1}:[-\frac{1}{2},\frac{1}{2}]\to [-\frac{1}{2},\frac{1}{2}]$ in the sense of $y=\psi(x) \Leftrightarrow x=\psi^{-1}(y)$ with $\psi^{-1}(y) \to \pm\frac{1}{2}$ as $y\to \pm\frac{1}{2}$.
We call the derivative of the inverse transformation the \emph{density function $\varrho$ of $\psi$}, which is a non-negative $L_1$-function on the interval $[-\frac{1}{2},\frac{1}{2}]$ and given by
\begin{align}\label{def:Varrho_cube}
	\varrho(y) := (\psi^{-1})'(y) = \frac{1}{\psi'(\psi^{-1}(y))}.
\end{align}
For multivariate transformations we put 
\begin{align}\label{def:Trafo_cube_mult}
	\psi(\mathbf x) := ( \psi_1(x_1),\ldots,\psi_d(x_d) )^{\top}
\end{align}
with ${\mathbf x = (x_1,\ldots,x_d)^{\top}\in[-\frac{1}{2},\frac{1}{2}]^d}$
and we may use different univariate torus-to-cube transformations $\psi_j$ in each coordinate.
The multivariate inverse transformation is denoted by ${\psi^{-1}(\mathbf y) := ( \psi_1^{-1}(y_1),\ldots,\psi_d^{-1}(y_d) )^{\top}}$ and 
\begin{align} \label{def:varrho_mult_cube}
	{\varrho(\mathbf y) := \prod_{j=1}^d \varrho_j(y_j)}
\end{align}
for all $\mathbf y = (y_1,\ldots,y_d)^{\top}\in[-\frac{1}{2},\frac{1}{2}]^d$.

We introduce a particular family of parameterized torus-to-cube transformations as defined in $\eqref{def:Trafo_cube}$ that are based on transformations $\tilde{\psi}$ to $\mathbb{R}$ whose definition is recalled from \cite{NaPo18}.
We call a continuously differentiable, increasing and odd mapping $\tilde{\psi}:(-\frac{1}{2},\frac{1}{2}) \to \mathbb{R}$ with $\tilde{\psi}(x) \to \pm\infty$ for $x \to \pm\frac{1}{2}$ a \textsl{transformation to $\mathbb{R}$}.
We obtain parameterized torus-to-cube transformations ${\psi(\cdot,\eta):[-\frac{1}{2},\frac{1}{2}] \to [-\frac{1}{2},\frac{1}{2}]}$ with $\eta\in\mathbb{R}_{+} := (0,\infty)$ by putting
\begin{align} \label{def:Trafo_comb}
	\psi(x, \eta) := 
	\begin{cases}
		\tilde{\psi}^{-1}(\eta \,\tilde{\psi}(x)) & \text{for}\quad x\in\left(-\frac{1}{2},\frac{1}{2}\right), \\
		\pm \frac{1}{2} & \text{for}\quad x=\pm\frac{1}{2}.
	\end{cases}
\end{align}
These transformations form a subset of all torus-to-cube transformations and are in a natural way continuously differentiable and increasing.
The respective first derivative and inverse torus-to-cube transformation are given by 
\begin{align*}
	\psi'(x,\eta) := \frac{\partial}{\partial x}[\psi](x, \eta) 
\quad \text{and} \quad
	\psi^{-1}(y, \eta) := \tilde{\psi}^{-1}\left(\frac{1}{\eta} \,\tilde{\psi}(y)\right).
\end{align*}
The corresponding density functions $\varrho(\circ,\eta)$ and $\varrho(\circ,\bm\eta)$ as well as the multivariate torus-to-cube transformation $\psi(\circ,\bm\eta)$ and its inverse $\psi^{-1}(\circ,\bm\eta)$ with $\bm\eta\in\mathbb{R}_{+}^{d}$ are simply parameterized versions of \eqref{def:Trafo_cube}, \eqref{def:Varrho_cube} and \eqref{def:Trafo_cube_mult} and share the same properties.

\subsection{Exemplary transformations}
In \cite[Section~17.6]{boyd00}, \cite[Section~7.5]{ShTaWa11} and \cite{NaPo18} we find various suggestions for transformations to $\mathbb{R}$.
We are particularly interested in the transformation
\begin{align} \label{eq:logarithmic_trafo}
		\tilde{\psi}(x) 
		&= \frac{1}{2}\log\left(\frac{1+2x}{1-2x}\right) 
		= \tanh^{-1}(2x)
\end{align}
based on the $\log$-function and the transformation
\begin{align} \label{eq:error_function_trafo}
	\tilde{\psi}(x) &= \mathrm{erf}^{-1}(2x)
\end{align}
based on the inverse of the error function 
\begin{align}\label{def:erf_fct}
	\mathrm{erf}(y) = \frac{1}{\sqrt{\pi}} \int_{-y}^{y} \mathrm{e}^{-t^2} \,\mathrm{d}t, \quad y\in\mathbb{R}.
\end{align}
Both transformations \eqref{eq:logarithmic_trafo} and \eqref{eq:error_function_trafo} induce a parameterized torus-to-cube transformation $\psi\left(y,\eta\right)$ with ${\eta > 0}$ as in \eqref{def:Trafo_comb}. 
It holds
${\psi^{-1}(y,\eta) = \psi\left(y,\frac{1}{\eta}\right)}$ and 
${\varrho(y,\eta) = \psi'\left(y,\frac{1}{\eta}\right)}$. 
For $x,y\in[-\frac{1}{2},\frac{1}{2}]$ we have the following torus-to-cube transformations:
\begin{itemize}
	\item
	\textsl{logarithmic transformation}:
	\begin{align} \label{eq:logarithmic_trafo_comb}
		\psi(x,\eta) = \frac{1}{2}\,\frac{(1+2x)^\eta - (1-2x)^\eta}{(1+2x)^\eta + (1-2x)^\eta},
		\quad \psi'(x,\eta) = \frac{4\eta(1-4x^2)^{\eta-1}}{\left((1+2x)^\eta+(1-2x)^\eta\right)^2},\end{align}
	and we observe that $\lim_{x \to \pm\frac{1}{2}} \psi'(x,\eta) = 0$ for $\eta>1$.
\item 
	\textsl{error function transformation}:
	\begin{align} \label{eq:error_function_trafo_comb}
		\psi(x,\eta) = \frac{1}{2}\,\mathrm{erf}(\eta\,\mathrm{erf}^{-1}(2x)) ,
		\quad \psi'(x,\eta) = \eta\,\mathrm{e}^{(1-\eta^2)(\mathrm{erf}^{-1}(2x))^2} \end{align}
	with the error function $\mathrm{erf}(\circ)$ as given in \eqref{def:erf_fct}, and $\mathrm{erf}^{-1}$ denoting the inverse error function.
	Again, we observe that $\lim_{x \to \pm\frac{1}{2}} \psi'(x,\eta) = 0$ for $\eta>1$.
\end{itemize}
We list an example for a torus-to-cube transformation $\psi:[-\frac{1}{2},\frac{1}{2}] \to [-\frac{1}{2},\frac{1}{2}]$ as defined in \eqref{def:Trafo_cube} that isn't induced by a transformation to $\mathbb{R}$:
\begin{itemize}
	\item 
	\textsl{sine transformation}:
	\begin{align} \label{eq:sine_trafo}
		\psi(x) &= \frac{1}{2}\,\sin(\pi x) = \frac{1}{2}\,\cos\left(\pi\left(x-\frac{1}{2}\right)\right) ,
		\quad \psi'(x) = \frac{\pi}{2} \cos(\pi x).\end{align}
\end{itemize}
Later on, we compare the limited smoothening effect of this particular transformation on any given test function $h\in L_{2}\left(\left[-\frac{1}{2},\frac{1}{2}\right]^d,\omega\right) \cap \mathcal{C}_{\mathrm{mix}}^{m}\left(\left[-\frac{1}{2},\frac{1}{2}\right]^d\right)$ with the logarithmic transformation \eqref{eq:logarithmic_trafo_comb}, for which we can achieve much more smoothness if the parameter $\eta\in\mathbb{R}_{+}$ is large enough.
In Figure~\ref{fig:combined_trafos_and_sin_compared} we compare the transformation mapping, its inverse and their derivatives of the logarithmic transformation~\eqref{eq:logarithmic_trafo_comb} for $\eta\in\{2,4\}$ and the sine transformation \eqref{eq:sine_trafo}.

\begin{figure}[]
	\begin{minipage}[b]{.245\linewidth}
		\centering
		\begin{tikzpicture}[scale=0.4]
		\pgfmathsetmacro\M{2}
		\pgfmathsetmacro\MM{4}
		\begin{axis}[
			samples=500,  
			xmin=-0.55, xmax=0.55, 
			ymin=-0.55, ymax=0.55,
			xtick = {-0.5,-0.25,0,0.25,0.5}, ytick = {-0.5,-0.25,0,0.25,0.5},
			title = {$\psi$},
			axis x line=center, axis y line=center,
			every axis plot/.append style={thick},
]
		\addplot[black, domain=-0.5:0.5] { (1/2)*sin(pi*deg(x))  };
\addplot[blue, dashdotted, domain=-0.5:0.5] { (1/2)*((1+2*x)^\M - (1-2*x)^\M)/((1+2*x)^\M + (1-2*x)^\M) };
\addplot[black, dotted, domain=-0.5:0.5] { (1/2)*((1+2*x)^\MM - (1-2*x)^\MM)/((1+2*x)^\MM + (1-2*x)^\MM) };
\end{axis}
		\end{tikzpicture}
	\end{minipage}
	\begin{minipage}[b]{.245\linewidth}
		\centering
		\begin{tikzpicture}[scale=0.4]
		\pgfmathsetmacro\M{2}
		\pgfmathsetmacro\MM{4}
		\begin{axis}[
			samples=1000,  
			xmin=-0.55, xmax=0.55, ymin=-0.55, ymax=0.55,
			xtick = {-0.5,-0.25,0,0.25,0.5}, ytick = {-0.5,-0.25,0,0.25,0.5},
			title = {$\psi^{-1}$},
			axis x line=center, axis y line=center,
			every axis plot/.append style={thick},
			legend style={at={(1,1.25)}, anchor=south,legend columns=1,legend cell align=left, font=\small}
		]
		\addplot[black, domain=-0.5:0.5] { (1/pi)*rad(asin(2*x))  };
		\addlegendentry{\eqref{eq:sine_trafo} sine transformation};
		\addplot[blue, dashdotted, domain=-0.5:0.5] { (1/2)*((1+2*x)^(1/\M) - (1-2*x)^(1/\M))/((1+2*x)^(1/\M) + (1-2*x)^(1/\M)) };
		\addlegendentry{\eqref{eq:logarithmic_trafo_comb} logarithmic transformation with $\eta=2$};
		\addplot[black, dotted, domain=-0.5:0.5] { (1/2)*((1+2*x)^(1/\MM) - (1-2*x)^(1/\MM))/((1+2*x)^(1/\MM) + (1-2*x)^(1/\MM)) };
		\addlegendentry{\eqref{eq:logarithmic_trafo_comb} logarithmic transformation with $\eta=4$};
		\end{axis}
		\end{tikzpicture}
	\end{minipage}
	\begin{minipage}[b]{.245\linewidth}
		\centering
		\begin{tikzpicture}[scale=0.4]
		\pgfmathsetmacro\M{2}
		\pgfmathsetmacro\MM{4}
		\begin{axis}[
			samples=500,  
			xmin=-0.55, xmax=0.55, ymin=-0.5, ymax=4.25,
			xtick = {-0.5,-0.25,0,0.25,0.5},
			title = {$\psi'$},
			axis x line=center, axis y line=center,
			every axis plot/.append style={thick},
]  
\addplot[black, domain=-0.5:0.5] { (pi/2)*cos(pi*deg(x))  };
\addplot[blue, dashdotted, domain=-0.5:0.5] { 4*\M*((1-4*x^2)^(\M-1))/(((1+2*x)^\M+(1-2*x)^\M)^2) };
\addplot[black, dotted, domain=-0.5:0.5] { 4*\MM*((1-4*x^2)^(\MM-1))/(((1+2*x)^\MM+(1-2*x)^\MM)^2) };
\end{axis}
		\end{tikzpicture}
	\end{minipage}
	\begin{minipage}[b]{.245\linewidth}
		\centering
		\begin{tikzpicture}[scale=0.4]
		\pgfmathsetmacro\M{2}
		\pgfmathsetmacro\MM{4}
		\begin{axis}[
			samples=500,  
			xmin=-0.55, xmax=0.55, ymin=-0.5, ymax=4.25,
			xtick = {-0.5,-0.25,0,0.25,0.5},
			title = {$\varrho = (\psi^{-1})'$},
			axis x line=center, axis y line=center,
			every axis plot/.append style={thick},
]
\addplot[black, domain=-0.5:0.5] { (2/pi)*1/(sqrt(1-4*x^2))  };
\addplot[blue, dashdotted, domain=-0.495:0.495] { 4*(1/\M)*((1-4*x^2)^((1/\M)-1))/(((1+2*x)^(1/\M)+(1-2*x)^(1/\M))^2) };
\addplot[black, dotted, domain=-0.495:0.495] { 4*(1/\MM)*((1-4*x^2)^((1/\MM)-1))/(((1+2*x)^(1/\MM)+(1-2*x)^(1/\MM))^2) };
\end{axis}
		\end{tikzpicture}
	\end{minipage}
	\caption{Comparison of the logarithmic transformation \eqref{eq:logarithmic_trafo_comb} with $\eta\in\{2,4\}$ and the sine transformation~\eqref{eq:sine_trafo}.
	}
	\label{fig:combined_trafos_and_sin_compared}
\end{figure}

\subsection{Weighted Hilbert spaces on the cube}
We describe the structure of the weighted function spaces $L_2\left([-\frac{1}{2},\frac{1}{2}],\omega\right)$ as defined in $\eqref{def:weighted_L2_space}$ for $d=1$.
The weight function $\omega:[-\frac{1}{2},\frac{1}{2}]\to[0,\infty)$ remains unspecified in this section.
Later on we may consider families of parameterized integrable weight functions $\omega(\circ,\mu)$ with ${\mu\in\mathbb{R}_{+}}$ to control the smoothness of functions in ${L_2\left([-\frac{1}{2},\frac{1}{2}],\omega(\circ,\mu)\right)\cap \mathcal{C}^{m}\left([-\frac{1}{2},\frac{1}{2}]\right)}$ and of the corresponding transformed functions as in \eqref{eq:change_of_variables_first} on the torus $\mathbb{T}$.
Families of multivariate parameterized weight functions are defined as 
\begin{align}
	\label{eq:omega_weighted_mult}
	\omega(\mathbf y, \bm \mu) := \prod_{j=1}^{d}\omega_j(y_j,\mu_j), \quad \mathbf y\in \left[-\frac{1}{2},\frac{1}{2}\right]^d, \bm\mu\in\mathbb{R}_{+}^d,
\end{align}
with univariate weight functions $\omega_j(\circ,\mu_j):[-\frac{1}{2},\frac{1}{2}]\to[0,\infty)$.

For now, we simplify the notation of the transformation, the weight function, and all related functions by omitting any parameter and just writing $\psi(\circ), \omega(\circ),$ etc. 
We remain in the univariate setting.
The system $\left\{\varphi_{k}\right\}_{k\in\mathbb{Z}}$ of weighted exponential functions
\begin{align}\label{eq:transformed_basis_functions}
	\varphi_{k}(y)
	:= \sqrt{\frac{\varrho(y)}{\omega(y)}} \, \mathrm{e}^{2\pi\mathrm i k\psi^{-1}(y)}, \quad y\in\left[-\frac{1}{2},\frac{1}{2}\right]
\end{align}
forms an orthonormal system with respect to the scalar product
\begin{align} \label{def:weighted_scalar_product}
	(h_1, h_2)_{L_2\left(\left[-\frac{1}{2},\frac{1}{2}\right], \omega \right)}
	:= \int_{-\frac{1}{2}}^{\frac{1}{2}} h_1(y) \, \overline{h_2(y)} \, \omega(y) \, \mathrm dy
\end{align}
and for all $k_1,k_2\in\mathbb{Z}$ we have
\begin{align*}
	(\varphi_{k_1}, \varphi_{k_2})_{L_2\left(\left[-\frac{1}{2},\frac{1}{2}\right], \omega\right)}
	= \delta_{k_1, k_2}.
\end{align*}
The weighted scalar product \eqref{def:weighted_scalar_product} induces the norm
\begin{align*}
	\|h\|_{L_2\left(\left[-\frac{1}{2},\frac{1}{2}\right], \omega \right)}
	:= \sqrt{ (h, h)_{L_2\left(\left[-\frac{1}{2},\frac{1}{2}\right], \omega \right)} }.
\end{align*}
In a natural way we have Fourier coefficients of the form
\begin{align} \label{def:FouCoeff_of_h}
	\hat h_{k}
	:= \left(h, \varphi_{k} \right)_{L_2\left(\left[-\frac{1}{2},\frac{1}{2}\right], \omega \right)}
	= \int_{-\frac{1}{2}}^{\frac{1}{2}} h(y) \, \sqrt{\varrho(y)\,\omega(y)} \, \mathrm{e}^{-2\pi\mathrm i k\psi^{-1}(y)} \, \mathrm dy,
\end{align}
as well as the respective Fourier partial sum for $I\subset\mathbb{Z}$ given by
\begin{align} \label{def:Fourier_part_sum_of_h}
	S_{I}h(y) 
	:= \sum_{k\in I} \hat h_{k} \, \varphi_{k}(y).
\end{align}

\subsection{Smoothness properties of transformed functions in $\mathcal{H}^{m}(\mathbb{T}^d)$}
In this section we characterize the smoothness properties of functions $h$ defined on $\left[-\frac{1}{2},\frac{1}{2}\right]^d$ and of their corresponding transformed versions $f$ as in \eqref{eq:f_is_transformed_h_mult} 
on $\left[-\frac{1}{2},\frac{1}{2}\right]^d$ after the application of a torus-to-cube transformation $\psi$ given in \eqref{def:Trafo_comb}.
We investigate the possibility to continuously extend these transformed functions $f$ to the torus $\mathbb{T}^d$.
We propose specific sufficient conditions for $\psi$ and $\omega$ such that the eventual transformed functions $f$ are in $\mathcal H^{m}\left(\mathbb{T}^d\right)$ with ${m\in\mathbb{N}_{0}}$.
These conditions are stated for both univariate and multivariate functions.
Afterwards, we utilize the embedding ${\mathcal{H}^{\beta+\lambda}(\mathbb{T}^d) \hookrightarrow \mathcal{A}^{\beta}(\mathbb T^d)}$ in \eqref{eq:Wiener_algebra_inclusion} for all ${\lambda > \frac{1}{2}}$ to discuss high-dimensional approximation problems, in which we apply {rank-$1$} lattice based fast Fourier approximation methods. 
Throughout this section we still omit the parameters $\bm\eta,\bm\mu\in\mathbb{R}_{+}^d$ in the notation of the torus-to-cube transformations $\psi$ and of the weight functions $\omega$.

For now, we consider univariate transformed functions $f\in L_{2}\left(\left[-\frac{1}{2},\frac{1}{2}\right]\right)$ of the form 
\begin{align}\label{eq:f_is_transformed_h}
	f(x) 
	:= h(\psi(x)) \, \sqrt{ \omega(\psi(x)) \, \psi'(x) },
	\quad x\in\left[-\frac{1}{2},\frac{1}{2}\right],
\end{align}
that are the result of applying a torus-to-cube transformation $y = \psi(x)$ as defined in \eqref{def:Trafo_cube} to the $L_2\left(\left[-\frac{1}{2},\frac{1}{2}\right],\omega\right)$-norm of the given function $h$ so that we have the identity
\begin{align*} 
	\|h\|_{L_{2}\left(\left[-\frac{1}{2},\frac{1}{2}\right],\omega\right)}^2
	&= \int_{-\frac{1}{2}}^{\frac{1}{2}} |h(y)|^2 \, \omega(y) \,\mathrm{d}y \\
	&= \int_{-\frac{1}{2}}^{\frac{1}{2}} \left| h(\psi(x)) \right|^2 \, \omega(\psi(x)) \, \psi'(x) \, \mathrm{d}x
	= \|f\|_{L_{2}\left(\left[-\frac{1}{2},\frac{1}{2}\right]\right)}^2.
\end{align*}
This is illustrated schematically in Figure~\ref{fig:DiagrammTrafo2}.
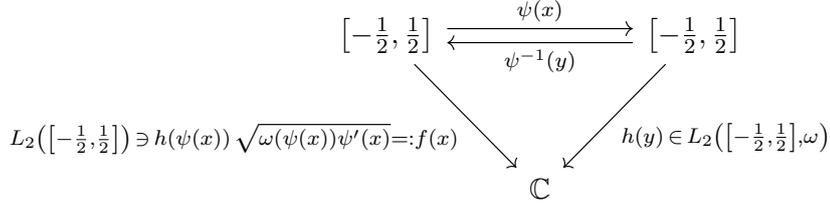
\begin{figure}[]
	\begin{center}
\begin{tikzpicture}[baseline= (a).base]
		\node[scale=1] (a) at (0,0)
		{
			\begin{tikzcd}[row sep=huge]
			 \left[-\frac{1}{2},\frac{1}{2}\right] \arrow[shift left]{rr}{\psi(x)} \arrow{dr}[swap]{ L_{2}\left(\left[-\frac{1}{2},\frac{1}{2}\right]\right) \,\ni\, h(\psi(x))\,\sqrt{\omega(\psi(x)) \psi'(x)} =: f(x) }
			& & \left[-\frac{1}{2},\frac{1}{2}\right] \arrow{dl}{h(y) \,\in\, L_{2}\left(\left[-\frac{1}{2},\frac{1}{2}\right], \omega \right)} \arrow[shift left]{ll}{\psi^{-1}(y)} & 
			\phantom{\mathbb{T} \simeq [-\frac{1}{2},\frac{1}{2}]  \arrow[l, phantom]} \\
			& \mathbb{C} & &
			\end{tikzcd}
		};
		\end{tikzpicture}
	\end{center}
	\caption{Scheme of the relation between $f$ and $h$ caused by a transformation $\psi$.}
	\label{fig:DiagrammTrafo2}
\end{figure}

Generally, it is rather difficult to check if such transformed functions $f$ are elements of $H^{m}\left(\left[-\frac{1}{2},\frac{1}{2}\right]\right)$ for some fixed $m\in\mathbb{N}_0$ by estimating the individual $L_{2}\left(\left[-\frac{1}{2},\frac{1}{2}\right]\right)$-norms within the Sobolev norm $\|f\|_{H^{m}\left(\left[-\frac{1}{2},\frac{1}{2}\right]\right)}$.
We propose a set of sufficient conditions such that ${f\in H^{m}\left(\left[-\frac{1}{2},\frac{1}{2}\right]\right)}$ with $m\in\mathbb{N}_0$
that eliminate the necessity to evaluate the $L_{2}$-integrals of various derivatives of $f$ and utilize the product structure of the functions $f$ in \eqref{eq:f_is_transformed_h}.
When we use parameterized families of torus-to-cube transformations $\psi(\circ,\eta)$ and families of weight functions $\omega(\circ,\mu)$,
we will calculate how large the parameters $\eta,\mu\in\mathbb{R}_{+}$ have to be in order to preserve the fixed degree of smoothness $m$ when transforming ${h\in L_{2}\left(\left[-\frac{1}{2},\frac{1}{2}\right],\omega\right) \cap \mathcal{C}^{m}\left(\left[-\frac{1}{2},\frac{1}{2}\right]\right)}$ into $f\in H^{m}\left(\left[-\frac{1}{2},\frac{1}{2}\right]\right)$ via $\psi(\circ,\eta)$.
By additionally assuming a certain vanishing behavior of the derivatives of the transformed weight function $\sqrt{ (\omega(\psi(\circ)) \, \psi'(\circ) }$ the transformed functions $f$ are continuously extendable to the torus $\mathbb{T}$ and we finally have smooth transformed functions $f\in \mathcal H^{m}(\mathbb{T})$ due to the norm equivalence \eqref{eq:Hs_norm_equivalence}.

\begin{remark}
	In the univariate setting we need a continuous function $f$ with ${f(-\frac{1}{2}) = f(\frac{1}{2})}$, which reads as
	\begin{align*}
		0 = \textstyle h(-\frac{1}{2})\sqrt{\omega(-\frac{1}{2}) \psi'(-\frac{1}{2})} - h(\frac{1}{2})\sqrt{\omega(\frac{1}{2}) \psi'(\frac{1}{2})},
	\end{align*}
	after recalling the we have $\psi\left(\pm\frac{1}{2}\right) = \pm\frac{1}{2}$.
	One approach to achieve this equality is to choose transformations $\psi$ whose first derivative $\psi'$ converges to $0$ at $x = \pm\frac{1}{2}$ fast enough that it isn't counteracted by the function $h$ or the weight function $\omega$. 
	Hence, we assume that $\sqrt{\omega(\psi(\circ)) \psi'(\circ)}\in\mathcal{C}_{0}\left(\left[-\frac{1}{2},\frac{1}{2}\right]\right)$.
	We focus on this approach, even though there are obviously more ways to achieve the above equality.
	In higher dimensions we assume for all $\mathbf m\in\mathbb{N}_{0}^{d}$ with $\|\mathbf m\|_{\ell_{\infty}} \leq m$ that $D^{\mathbf m}\left[\prod_{j=1}^{d}\sqrt{ (\omega_j\circ\psi_j) \,\psi_j' }\right] \in \mathcal{C}_{0}\left(\left[-\frac{1}{2},\frac{1}{2}\right]^d\right)$.

	Later on, we will choose a constant weight function $\omega \equiv 1$ and make use of the logarithmic transformation~\eqref{eq:logarithmic_trafo_comb} or the error transformation~\eqref{eq:error_function_trafo_comb} for the purpose of achieving this behavior of the transformed functions $f$ at the boundary points.
	While their first derivatives $\psi'(\circ,\eta)$ are always $0$ at the boundary points, the parameter $\eta$ has to be sufficiently large to achieve the same property for higher derivatives.
\end{remark}

Now, we propose a set of sufficient univariate conditions such that we obtain smooth transformed functions $f\in \mathcal 
H^{m}(\mathbb{T})$.
We denote the $k$-th derivative of a function $f(x)$ with respect to $x$ by one of the equivalent expressions ${f^{(k)}(x) = \frac{\mathrm{d}^{k}}{\mathrm{d}x^{k}}[f](x)}$, and for $k=1$ we continue to use the notation $f'(x)$.
\begin{theorem} \label{thm:Cm_composition_criteria}
	Let $m\in\mathbb{N}_0$,  
	$\psi$ as in \eqref{def:Trafo_cube}, 
	${h\in L_2\left(\left[-\frac{1}{2},\frac{1}{2}\right],\omega\right)\cap \mathcal{C}^{m}\left(\left[-\frac{1}{2},\frac{1}{2}\right]\right)}$ with an integrable weight function $\omega:\left[-\frac{1}{2},\frac{1}{2}\right]\to[0,\infty)$ 
	and the corresponding transformed functions $f$ of the form \eqref{eq:f_is_transformed_h} be given.

	We have $f \in \mathcal H^{m}\left(\mathbb{T}\right)$ if for all $n=0,1,\ldots,m$ we have 
	\begin{align} \label{eq:Cm_composition_criteria}
{\psi\in\mathcal{C}^{m}\left(\left[-\frac{1}{2},\frac{1}{2}\right]\right)}
		\quad \text{and} \quad 
		\left( \sqrt{ (\omega\circ\psi) \, \psi' } \right)^{(n)}(\circ)\in \mathcal{C}_{0}\left(\left[-\frac{1}{2},\frac{1}{2}\right]\right).
	\end{align}
\end{theorem}
\begin{proof}
	For $h\in L_{2}\left(\left[-\frac{1}{2},\frac{1}{2}\right],\omega\right) \cap \mathcal{C}^{m}\left(\left[-\frac{1}{2},\frac{1}{2}\right]\right)$ with $m\in\mathbb{N}_{0}$ and a torus-to-cube transformation $\psi$ as defined in $\eqref{def:Trafo_cube}$
	we consider the function $f$ as given in \eqref{eq:f_is_transformed_h}.
	At first we check if $f\in H^{m}\left(\left[-\frac{1}{2},\frac{1}{2}\right]\right)$ and have to show that $\left\| f^{(n)}(\circ) \right\|_{L_2\left(\left[-\frac{1}{2},\frac{1}{2}\right]\right)} < \infty$ for all $n=0,1,\ldots,m$. 
	
	We apply the \emph{generalized Leibniz rule} for the $n$-th derivative of a product of functions
	\begin{align*} (f\cdot g)^{(n)}(x)
		= \sum_{k=0}^{n} \binom{n}{k} f^{(k)}(x) \,g^{(n-k)}(x)
	\end{align*}
	to the Sobolev norm of $f$, which leads to
	\begin{align} \label{eq:Leibniz_estimate}
		&\|f\|_{H^{m}\left(\left[-\frac{1}{2},\frac{1}{2}\right]\right)} ^2
		=  
		\sum_{n=0}^{m} \| f^{(n)}(\circ) \|_{L_{2}\left(\left[-\frac{1}{2},\frac{1}{2}\right]\right)}^{2} 
		\nonumber \\
		&\leq 
		\sum_{n=0}^{m} \left( 
			\sum_{k=0}^{n} \binom{n}{k} \left\| (h\circ\psi)^{(k)}(\circ) 
			\left( \sqrt{ (\omega\circ\psi) \, \psi' } \right)^{(n-k)}(\circ) \right\|_{L_{2}\left(\left[-\frac{1}{2},\frac{1}{2}\right]\right)}
		\right)^{2}.
	\end{align}
	
	We leave $h\circ\psi$ in the term corresponding to $k=0$ untouched for now. 
	For $k = 1,\ldots,m$ we use the \emph{Fa\'{a} di Bruno} formula to write the $k$-th derivative of the composition of functions $h$ and $\psi$ as
	\begin{align} \label{eq:FdiB_formula}
		( h\circ\psi )^{(k)}(x)
		= \sum_{\ell = 1}^{k} h^{(\ell)}(\psi(x))\, B_{k,\ell}(\psi'(x), \psi^{(2)}(x),\ldots, \psi^{(k-\ell+1)}(x))
	\end{align}
	with $(h\circ\psi)^{(0)}(x) = h(\psi(x))$
	and the well-known Bell polynomials $B_{k,\ell}$ for $k,\ell\in\mathbb{N}_{0}$ are given by
	\begin{align*}
B_{k,\ell}(\mathbf z) := \sum_{ \substack{j_1+j_2+\ldots+j_{k-\ell+1}=\ell,\\ j_1+2j_2+\ldots+(k-\ell+1)j_{k-\ell+1} = k} }
		\frac{\ell!}{j_1! \cdot \ldots \cdot j_{k-\ell+1}!}
		\prod_{m=1}^{k-\ell+1} \left(\frac{z_{m}}{m!}\right)^{j_{m}}
	\end{align*}
	with $\mathbf z = (z_1,\ldots,z_{k-\ell+1})^{\top}$. 
	By assumption all derivatives of $\psi$ are bounded on the interval $[-\frac{1}{2},\frac{1}{2}]$. 
	Hence, each Bell polynomial $B_{k,\ell}$ in \eqref{eq:FdiB_formula} is bounded, too.
	To simplify the notation we write $B_{k,\ell}(\psi(x)) := B_{k,\ell}(\psi'(x),\ldots, \psi^{(k-\ell+1)}(x))$. We insert \eqref{eq:FdiB_formula} into \eqref{eq:Leibniz_estimate} and estimate
	\begin{align*}&\|f\|_{H^{m}\left(\left[-\frac{1}{2},\frac{1}{2}\right]\right)}^2 \nonumber\\
		&\lesssim \sum_{n=0}^{m} \left( 
			\sum_{k=0}^{n} 
			\left\| 
				\sum_{\ell = 1}^{k} h^{(\ell)}(\psi(\circ)) B_{k,\ell}(\psi(\circ)) \left( \sqrt{ (\omega\circ\psi) \, \psi' } \right)^{(n-k)}(\circ)
			\right\|_{L_{2}\left(\left[-\frac{1}{2},\frac{1}{2}\right]\right)}
		\right)^{2}.
	\end{align*}
	The appearing $L_2$-norms are estimated by their respective $L_{\infty}$-norms, so that
	\begin{align*}
		&\left\| \sum_{\ell = 1}^{k} h^{(\ell)}(\psi(\circ)) B_{k,\ell}(\psi(\circ)) \left( \sqrt{ (\omega\circ\psi) \, \psi' } \right)^{(n-k)}(\circ) \right\|_{L_{2}\left(\left[-\frac{1}{2},\frac{1}{2}\right]\right)}  \\
		&\lesssim 
		\sum_{\ell = 1}^{k} 
		\left\| \left(\sqrt{ (\omega\circ\psi) \, \psi' } \right)^{(n-k)}(\circ) \right\|_{L_{\infty}\left(\left[-\frac{1}{2},\frac{1}{2}\right]\right)},
	\end{align*}
	because of the boundedness of all appearing Bell polynomials $B_{k,\ell}$ and the assumption that $h$ is $m$-times continuously differentiable.
	Thus, the norm $\|f\|_{H^{m}\left(\left[-\frac{1}{2},\frac{1}{2}\right]\right)}$ is finite, if the first $m$ derivatives of $\sqrt{ (\omega(\psi(\circ)) \, \psi'(\circ) }$ have a fine $L_{\infty}$-norm.
	Finally, we assumed that the first $m$ derivatives of $\sqrt{ (\omega(\psi(\circ)) \, \psi'(\circ) }$ also vanish at the boundary points, which implies that the first $m$ derivatives of the transformed function $f$ vanish at the boundary points, too.
	Hence, $f$ is in $\mathcal H^{m}\left(\mathbb{T}\right)$ due to the norm equivalence \eqref{eq:Hs_norm_equivalence}.
\end{proof}

Next, we prove the multivariate version of Theorem~\ref{thm:Cm_composition_criteria}.
Similarly to \eqref{eq:f_is_transformed_h}, we consider multivariate transformed functions $f\in L_{2}\left(\left[-\frac{1}{2},\frac{1}{2}\right]^d\right)$ of the form 
\begin{align}\label{eq:f_is_transformed_h_mult}
	f(\mathbf x) 
	= h(\psi_{1}(x_1),\ldots,\psi_{d}(x_d)) \prod_{k=1}^{d}\sqrt{\omega_k(\psi_k(x_k)) \psi'_{k}(x_k)},
\end{align}
with $\mathbf x = (x_1,\ldots,x_d)^{\top}\in\left[-\frac{1}{2},\frac{1}{2}\right]^d$
that are the result of applying the multivariate transformation 
\begin{align*}
	\mathbf y = (y_1,\ldots,y_d)^{\top} = (\psi_{1}(x_{1}),\ldots,\psi_{d}(x_{d}))^{\top} = \psi(\mathbf x)
\end{align*}
as defined in \eqref{def:Trafo_cube_mult} to a function $h\in L_2\left(\left[-\frac{1}{2},\frac{1}{2}\right]^d,\omega\right)$ with a product weight $\omega$ as in \eqref{eq:omega_weighted_mult}.
For these we have the identity
\begin{align*} 
	\|h\|_{L_2\left(\left[-\frac{1}{2},\frac{1}{2}\right]^d,\omega\right)}^2
	&= \int_{\left[-\frac{1}{2},\frac{1}{2}\right]^d} |h(\mathbf y)|^2 \omega(\mathbf y) \,\mathrm{d}\mathbf y \\
	&= \int_{\mathbb{T}^d} \!\!|(h\circ\psi)(\mathbf x)|^2 (\omega\circ\psi)(\mathbf x)\! \prod_{j=1}^{d}\psi'_{j}(x_j) \, \mathrm{d}\mathbf x
	= \|f\|_{L_{2}\left(\left[-\frac{1}{2},\frac{1}{2}\right]^d\right)}^2. \nonumber
\end{align*}
Again, we derive a set of sufficient $L_{\infty}$-conditions on the transformation $\psi$ and the product weight $\omega$ 
for an $h\in L_{2}\left(\left[-\frac{1}{2},\frac{1}{2}\right]^d, \omega\right) \cap \mathcal{C}_{\mathrm{mix}}^{m}\left(\left[-\frac{1}{2},\frac{1}{2}\right]^d\right)$ to be transformed by $\psi$ into an $f \in \mathcal H^{m}\left(\mathbb{T}^d\right)$ of form \eqref{eq:f_is_transformed_h_mult}. 
\begin{theorem} \label{thm:Cm_composition_criteria_mult}
	Let $d\in\mathbb{N}$, 
	$m\in\mathbb{N}_{0}$, 
	a $d$-variate $\psi$, 
	a product weight function $\omega$
	as in \eqref{eq:omega_weighted_mult},
	${h\in L_2\left(\left[-\frac{1}{2},\frac{1}{2}\right]^d,\omega\right)\cap \mathcal{C}_{\mathrm{mix}}^{m}\left(\left[-\frac{1}{2},\frac{1}{2}\right]^d\right)}$
	and the corresponding transformed functions $f$ of the form \eqref{eq:f_is_transformed_h_mult} be given.
	
	We have ${f \in \mathcal H^{m}\left(\mathbb{T}^d\right)}$ 
	if for all multi-indices ${\mathbf m \in\mathbb{N}_{0}^{d}, \|\mathbf m\|_{\ell_{\infty}} \leq m}$, 
we have 
	\begin{align} \label{eq:Cm_composition_criteria_mult}
{\psi\in\mathcal{C}_{\mathrm{mix}}^{m}\left(\left[-\frac{1}{2},\frac{1}{2}\right]^d\right)} 
		\quad \text{and} \quad 
D^{\textbf{m}}\left[\prod_{k=1}^{d} \sqrt{(\omega_k\circ\psi_k) \psi'_{k}}\right]\in \mathcal{C}_{0}\left(\left[-\frac{1}{2},\frac{1}{2}\right]^d\right).
	\end{align}
\end{theorem}
\begin{proof}
	For $h\in L_2\left(\left[-\frac{1}{2},\frac{1}{2}\right]^d,\omega\right)\cap \mathcal{C}_{\mathrm{mix}}^{m}\left(\left[-\frac{1}{2},\frac{1}{2}\right]^d\right)$ with $m\in\mathbb{N}_{0}$ and a multivariate torus-to-cube transformation $\psi$ as defined in $\eqref{def:Trafo_cube_mult}$
	we consider the transformed function $f$ as given in \eqref{eq:f_is_transformed_h_mult}.
	At first we check if $f\in H_{\mathrm{mix}}^{m}\left(\left[-\frac{1}{2},\frac{1}{2}\right]^d\right)$ and have to show that for all multi-indices ${\mathbf m \in\mathbb{N}_{0}^{d}}$ with $\|\mathbf m\|_{\ell_{\infty}} \leq m$ we have $\left\| D^{\mathbf m}[f] \right\|_{L_2\left(\left[-\frac{1}{2},\frac{1}{2}\right]^d\right)} < \infty$. 
	
	Let $\mathbf m= (m_1,\ldots,m_d)^{\top}\in\mathbb{N}_{0}^{d}$ be any multi-index with $\|\mathbf m\|_{\ell_{\infty}} \leq m$. 
	For a multivariate transformed function $f$ of the form \eqref{eq:f_is_transformed_h_mult} we have
	\begin{align}\label{eq:mult_proof_L2T}
		\left\| D^{\mathbf m}[f](\mathbf x) \right\|_{L_2\left(\left[-\frac{1}{2},\frac{1}{2}\right]^d\right)}^2 
		= 
		\int_{\left[-\frac{1}{2},\frac{1}{2}\right]^d} \left| D^{\mathbf m}\left[(h\circ\psi)\prod_{k=1}^{d}\sqrt{(\omega_k\circ\psi_k) \psi'_{k}}\right](\mathbf x) 
		\right|^2 \mathrm{d}\mathbf x.
	\end{align}
	Due to the product weight function in the transformed function $f$ in \eqref{eq:f_is_transformed_h_mult}
	the componentwise application of the Leibniz formula as in \eqref{eq:Leibniz_estimate} we estimate
\begin{align} \label{eq:mult_proof_leibniz_applied}
		&D^{\mathbf m}\left[ (h\circ\psi)\prod_{k=1}^{d}\sqrt{(\omega_k\circ\psi_k)\psi'_{k}}\right](\mathbf x) \\
		&\leq \sum_{j_1=0}^{m_1} 
		\ldots \sum_{j_d=0}^{m_d} 
		D^{(j_1,\ldots,j_d)}[h\circ \psi](\mathbf x) \,  D^{(m_1-j_1,\ldots,m_d-j_d)} \left[\prod_{k=1}^{d}\sqrt{(\omega_k\circ\psi_k)\psi'_{k}}\right](\mathbf x). \nonumber
	\end{align}
	Next, we apply the {Fa\'{a} di Bruno} formula \eqref{eq:FdiB_formula} to each univariate $j_{\ell}$-th derivative occurring in the term $D^{(j_1,\ldots,j_d)}[h\circ \psi](\mathbf x)$ in \eqref{eq:mult_proof_leibniz_applied}.
	For all ${\ell}=1,\ldots,d$ we put $B_{j_{\ell},i_{\ell}}(\psi_{\ell}(x_{\ell})) := B_{j_{\ell},i_{\ell}}(\psi'_{\ell}(x_{\ell}),\ldots, \psi^{(j_{\ell}-i_{\ell}+1)}_{\ell}(x_{\ell}))$ and we have
	\begin{align}
	 &D^{(0,\ldots,0,j_\ell,0,\ldots,0)}[h\circ \psi](\mathbf x) \nonumber \\
	 &=
	 \begin{cases} \label{eq:mult_proof_FdB}
	 	h(\psi(\mathbf x)) &\quad\text{for } j_{\ell}=0, \\
	 	\displaystyle \sum_{i_{\ell}=1}^{j_{\ell}} D^{(0,\ldots,0,i_{\ell},0,\ldots,0)}[h](\psi(\mathbf x)) B_{j_{\ell},i_{\ell}}(\psi_{\ell}(x_{\ell})) &\quad\text{for } j_{\ell}\in\mathbb{N}.
	 \end{cases}
	\end{align}
	We combine the norm $\left\| D^{\mathbf m}[f] \right\|_{L_2\left(\left[-\frac{1}{2},\frac{1}{2}\right]^d\right)}$ in $\eqref{eq:mult_proof_L2T}$ 
	with the expression resulting from applying the Leibniz formula to $D^{\mathbf m}[f]$ in $\eqref{eq:mult_proof_leibniz_applied}$ and the subsequent application of the {Fa\'{a} di Bruno} formula in \eqref{eq:mult_proof_FdB}.
	Then we estimate the occurring summands by their $L_2$-norm and by their $L_\infty$-norm 
	and utilize the boundedness of the Bell polynomials $B_{j_{\ell},i_{\ell}}$ as well as the assumption that $h$ is a $\mathcal{C}_{\mathrm{mix}}^m$-function,
	so that we end up with
	\begin{align*} 
		&\left\| D^{\mathbf m}[f](\mathbf x) \right\|_{L_2\left(\left[-\frac{1}{2},\frac{1}{2}\right]^d\right)} \\
		&\lesssim \sum_{j_1=0,\ldots,j_d=0}^{m_1,\ldots,m_d} \sum_{i_1=1,\ldots,i_d=1}^{j_1,\ldots,j_d}
		\left( \int_{\left[-\frac{1}{2},\frac{1}{2}\right]^d}
			| D^{(i_1,\ldots,i_d)}[h](\psi(\mathbf x))|^2 \right.\times \\
		&\quad\times	
			\left. 
		 \prod_{{\ell}=1}^{d} 
		 	|B_{j_{\ell},i_{\ell}}(\psi_{\ell}(x_{\ell}))|^2 
			\left| D^{(m_1-j_1,\ldots,m_d-j_d)}\left[\prod_{k=1}^{d}\sqrt{(\omega_k\circ\psi_k)\psi'_{k}}\right](\mathbf x) \right|^2
		\,\mathrm{d}\mathbf x
		\right)^{\frac{1}{2}}\nonumber \\
		&\lesssim \sum_{j_1=0,\ldots,j_d=0}^{m_1,\ldots,m_d} \left\| D^{(m_1-j_1,\ldots,m_d-j_d)}\left[\prod_{k=1}^{d}\sqrt{(\omega_k\circ\psi_k)\psi'_{k}}\right](\circ) \right\|_{L_{\infty}\left(\left[-\frac{1}{2},\frac{1}{2}\right]^d\right)}.\nonumber
	\end{align*}
	By assumption, the derivatives 
	${D^{\mathbf m}\left[\prod_{k=1}^{d}\sqrt{ (\omega_k\circ\psi_k) \psi_k'}\right]}$ vanish at the boundary points
	for all $\mathbf m\in\mathbb{N}_{0}^{d}, \|\mathbf m\|_{\ell_{\infty}} \leq m$.
	Thus, the derivatives $D^{\mathbf m}[f]$ of the transformed function $f$ at the boundary points, too, and $f$ is in $\mathcal H^{m}\left(\mathbb{T}^d\right)$ due to the equivalence \eqref{eq:Hs_norm_equivalence}.
\end{proof}

\subsection{Approximation of transformed functions}
We establish two specific approximation error bounds for functions defined on $\left[-\frac{1}{2},\frac{1}{2}\right]^d$ based on the approximation error bounds on the torus $\mathbb{T}^d$ that we recalled in Theorems~\ref{thm:L_infty_approx_error_torus} and \ref{eq:L_2_approximation_error_bound}.
The corresponding proofs rely heavily on the previously introduced sufficient conditions in Theorem~\ref{thm:Cm_composition_criteria_mult} which guarantee that functions ${h\in L_2\left(\left[-\frac{1}{2},\frac{1}{2}\right]^d,\omega\right)\cap \mathcal{C}_{\mathrm{mix}}^{m}\left(\left[-\frac{1}{2},\frac{1}{2}\right]^d\right)}$ are transformed into Sobolev functions of dominated mixed smoothness on $\mathbb{T}^d$ of the form \eqref{eq:f_is_transformed_h_mult} by multivariate transformations $\psi:\left[-\frac{1}{2},\frac{1}{2}\right]^d\to\left[-\frac{1}{2},\frac{1}{2}\right]^d$ as given in \eqref{def:Trafo_cube_mult}.

Based on the definition of a {rank-$1$} lattice $\Lambda(\mathbf z,M)$ in \eqref{def:rank_one_lattice}, we define a \textsl{transformed {rank-$1$} lattice} as 
\begin{align}\label{eq:Def_trafo_lattice}
	\Lambda_{\psi}(\mathbf z,M) := \left\{ \mathbf y_j := \psi(\mathbf x_j) : \mathbf x_j\in\Lambda(\mathbf z,M), j = 0,\ldots,M-1 \right\}.
\end{align}
A \emph{transformed reconstructing {rank-$1$} lattice} is denoted by $\Lambda_{\psi}(\mathbf z,M,I)$.
Based on the functions $\varphi_{k}$ given in \eqref{eq:transformed_basis_functions} we put 
\begin{align*}
	{\varphi_{\mathbf k}(\mathbf y) := \prod_{j=1}^{d}\varphi_{k_j}(y_j)}.
\end{align*}
Similarly to \eqref{def:weighted_scalar_product} and \eqref{def:FouCoeff_of_h}, the multivariate weighted $L_{2}\left(\left[-\frac{1}{2},\frac{1}{2}\right]^d,\omega\right)$-scalar product reads as
\begin{align*} 
	(h_1, h_2)_{L_2\left(\left[-\frac{1}{2},\frac{1}{2}\right]^d, \omega \right)}
	:= \int_{\left[-\frac{1}{2},\frac{1}{2}\right]^d} h_1(\mathbf y) \, \overline{h_2(\mathbf y)} \, \prod_{j=1}^{d}\omega_j(y_j) \, \mathrm d\mathbf y
\end{align*}
and the multivariate Fourier coefficients $\hat h_{\mathbf k}$ are naturally given with respect to this scalar product reads as
\begin{align} \label{def:FC_trafo_multivar}
	\hat h_{\mathbf k} 
	= (h, \varphi_{\mathbf k})_{L_2\left(\left[-\frac{1}{2},\frac{1}{2}\right]^d, \omega \right)}.
\end{align}
As in \eqref{def:Fourier_part_sum_of_h}, we define the multivariate Fourier partial sum for any $I\subset\mathbb{Z}^d$ as
\begin{align*}
	S_{I}h(\mathbf y)
	:= \sum_{\mathbf k\in I} \hat h_{\mathbf k} \, \varphi_{\mathbf k}(\mathbf y).
\end{align*}
Suppose $f\in L_{2}\left(\mathbb{T}^d\right)$. 
For each $I\subset\mathbb{Z}^d$ the system $\left\{\varphi_{\mathbf k}\right\}_{\mathbf k\in I}$ spans the space of transformed trigonometric polynomials
\begin{align}
	\label{def:trig_poly_trafo_mult}
	\Pi_{I,\psi} := \mathrm{span}\left\{ \sqrt{\frac{\varrho(\circ)}{\omega(\circ)}} \, \mathrm{e}^{2\pi\mathrm i \mathbf k\cdot\psi^{-1}(\circ)} : \mathbf k \in I \right\}.
\end{align}
As in \eqref{eq:exact_integration_formula}, for any transformed trigonometric polynomial $h \in \Pi_{I,\psi}$ the transformed lattice nodes 
${\mathbf y_j\in\Lambda_{\psi}(\mathbf z,M,I)}$ 
and all $\mathbf k\in I$ we have the exact integration property of the form 
\begin{align} 
	\label{eq:exact_integration_prop_2}
	\hat h_{\mathbf k}
	&= \int_{\left[-\frac{1}{2},\frac{1}{2}\right]^d} h(\mathbf y) \, \sqrt{\varrho(\mathbf y)\,\omega(\mathbf y)} \, \mathrm{e}^{-2\pi\mathrm i \mathbf k\cdot\psi^{-1}(\mathbf y)} \, \mathrm d\mathbf y 
	= \int_{\mathbb{T}^d} f(\mathbf x) \,\mathrm{e}^{-2\pi\mathrm{i}\mathbf k\cdot\mathbf x} \,\mathrm{d}\mathbf x \nonumber \\
	&= \frac{1}{M} \sum_{j =0}^{M-1} f(\mathbf x_j) \,\mathrm{e}^{-2\pi\mathrm{i}\mathbf k\cdot\mathbf x_j} 
	= \frac{1}{M} \sum_{j=0}^{M-1} h(\mathbf y_j)\, \sqrt{\frac{\omega(\mathbf y_j)}{\varrho(\mathbf y_j)}} \, \mathrm{e}^{-2\pi\mathrm i \mathbf k\cdot\psi^{-1}(\mathbf y_j)}
	=\hat h_{\mathbf k}^{\Lambda}.
\end{align}
Generally, the multivariate approximated Fourier coefficients of the form
\begin{align*}
	\hat h_{\mathbf k}^{\Lambda} 
	:= \frac{1}{M} \sum_{j=0}^{M-1} h(\mathbf y_j)\, \sqrt{\frac{\omega(\mathbf y_j)}{\varrho(\mathbf y_j)}} \, \mathrm{e}^{-2\pi\mathrm i \mathbf k\cdot\psi^{-1}(\mathbf y_j)} 
	= \frac{1}{M} \sum_{j=0}^{M-1} \frac{\omega(\mathbf y_j)}{\varrho(\mathbf y_j)} \, h(\mathbf y_j)\, \overline{\varphi_{\mathbf k}(\mathbf y_j)}
\end{align*}
only approximate the multivariate Fourier coefficients $\hat h_{\mathbf k}$.
Finally, the multivariate version of the approximated Fourier partial sum is given by
\begin{align}
	\label{def:mult_approximated_FPS_FC}
	S_{I}^{\Lambda}h(\mathbf y) := \sum_{\mathbf k\in I} \hat h_{\mathbf k}^{\Lambda} \, \varphi_{\mathbf k}(\mathbf y).
\end{align}
Similarly to the Hilbert space $\mathcal{H}^{\beta}(\mathbb T^d)$ given in \eqref{def:HbetaRaum}, we define such a space of $L_{2}$-functions on the cube $\left[-\frac{1}{2},\frac{1}{2}\right]^d$ with square summable Fourier coefficients $\hat h_{\mathbf k}$ as in \eqref{def:FC_trafo_multivar} by
\begin{align*}
	\textstyle
\mathcal{H}^{\beta}\left(\left[-\frac{1}{2},\frac{1}{2}\right]^d\right) := \left\{ h \in L_2\left(\left[-\frac{1}{2},\frac{1}{2}\right]^d, \omega\right) : \|h\|_{\mathcal{H}^{\beta}\left(\left[-\frac{1}{2},\frac{1}{2}\right]^d\right)} < \infty \right\}
\end{align*}
with the norm
\begin{align*}
	\|h\|_{\mathcal{H}^{\beta}\left(\left[-\frac{1}{2},\frac{1}{2}\right]^d\right)} 
	:= \left( \sum_{\mathbf k\in\mathbb{Z}^d} \omega_{\mathrm{hc}}(\mathbf k)^{2\beta} |\hat h_{\mathbf k}|^2 \right)^{\frac{1}{2}}.
\end{align*}
The existence of the Fourier coefficients $\hat h_{\mathbf k}$ becomes apparent after applying the well-known Cauchy-Schwarz-Inequality so that
\begin{align*}
	|\hat h_{\mathbf k}|
	&= \left|\int_{\left[-\frac{1}{2},\frac{1}{2}\right]^d} h(\mathbf y) \sqrt{\omega(\mathbf y) \varrho(\mathbf y)} \,\mathrm{e}^{-2\pi\mathrm i \mathbf k\cdot\psi^{-1}(\mathbf y)} \,\mathrm{d}\mathbf y \right| \\
	&\leq 
		\left(\int_{\left[-\frac{1}{2},\frac{1}{2}\right]^d} \left| h(\mathbf y) \right|^2 \omega(\mathbf y) \,\mathrm{d}\mathbf y \right)^{\frac{1}{2}}
		\left(\int_{\left[-\frac{1}{2},\frac{1}{2}\right]^d} \left| \varrho(\mathbf y) \right| \,\mathrm{d}\mathbf y \right)^{\frac{1}{2}}
	< \infty,
\end{align*}
which is due to the fact that $h \in L_2\left(\left[-\frac{1}{2},\frac{1}{2}\right]^d, \omega\right)$ and each univariate factor appearing in the multivariate density $\varrho(\mathbf y)$ as defined in \eqref{def:varrho_mult_cube} is an $L_{1}$-function.
Next, we adjust the approximation error bounds in Theorems~\ref{thm:L_infty_approx_error_torus} and \ref{eq:L_2_approximation_error_bound} for functions defined on the cube $\left[-\frac{1}{2},\frac{1}{2}\right]^d$.

\subsubsection{$L_{\infty}$-approximation error}
Based on the $L_{\infty}(\mathbb{T}^d)$-approximation error bound \eqref{eq:torus_infty_approx_bound} and the conditions proposed in Theorem~\ref{thm:Cm_composition_criteria_mult} we prove a similar upper bound for the approximation error $\left\| h - S_{I_N^d}^{\Lambda} h \right\|$ in terms of a weighted $L_{\infty}$-norm on $\left[-\frac{1}{2},\frac{1}{2}\right]^d$.
\begin{theorem} \label{thm:L_infty_approx_error_multivar}
	Let $d\in\mathbb{N}$, $m\in\mathbb{N}_{0}$, 
	a hyperbolic cross $I_{N}^{d}$ with $N\geq 2^{d+1}$ and a reconstructing {rank-$1$} lattice $\Lambda(\mathbf z, M, I_{N}^{d})$ be given.	
	Let $\omega$ be a weight function as in \eqref{eq:omega_weighted_mult},
	${h\in L_2\left(\left[-\frac{1}{2},\frac{1}{2}\right]^d,\omega\right)\cap \mathcal{C}_{\mathrm{mix}}^{m}\left(\left[-\frac{1}{2},\frac{1}{2}\right]^d\right)}$,
	let $\psi$ be a multivariate transformation as in \eqref{def:Trafo_cube_mult},  	
	and let $\lambda > \frac{1}{2}$.
	For all multi-indices $\mathbf m = (m_1,\ldots,m_d)^{\top} \in\mathbb{N}_{0}^{d}$ with $\|\mathbf m\|_{\ell_{\infty}} \leq m$
we assume that
	\begin{align*}
		{\psi\in\mathcal{C}_{\mathrm{mix}}^{m}\left(\left[-\frac{1}{2},\frac{1}{2}\right]^d\right)} 
		\quad \text{and} \quad 
		D^{\mathbf m}\left[\prod_{\ell=1}^{d}\sqrt{(\omega_{\ell}\circ\psi_{\ell})\psi'_{\ell}}\right]\in \mathcal{C}_{0}\left(\left[-\frac{1}{2},\frac{1}{2}\right]^d\right).
	\end{align*}
	Then there is an approximation error estimate of the form
	\begin{align*}
		\left\|h - S_{I_N^d}^{\Lambda} h \right\|_{L_{\infty}\left(\left[-\frac{1}{2},\frac{1}{2}\right]^d, \sqrt{\frac{\omega}{\varrho}}\right)} 
		\lesssim 2 N^{-m+\lambda} \|h\|_{\mathcal{H}^{m}\left(\left[-\frac{1}{2},\frac{1}{2}\right]^d\right)}.
	\end{align*}
\end{theorem}
\begin{proof}
	Let $d\in\mathbb{N}, m\in\mathbb{N}_{0}$ and $h\in{L_2\left(\left[-\frac{1}{2},\frac{1}{2}\right]^d,\omega\right)\cap \mathcal{C}_{\mathrm{mix}}^{m}\left(\left[-\frac{1}{2},\frac{1}{2}\right]^d\right)}$.
	By assumption the criteria of Theorem~\ref{thm:Cm_composition_criteria_mult} are fulfilled 
	and the transformed function $f$ of the form \eqref{eq:f_is_transformed_h_mult} is continuously extendable to the torus $\mathbb{T}^d$.
	Thus, we have $f \in \mathcal H^{m}(\mathbb{T}^d)$ and a continuous representative, because of the inclusion ${\mathcal{H}^{m}(\mathbb{T}^d) \hookrightarrow \mathcal{A}^{m-\lambda}(\mathbb T^d) \hookrightarrow \mathcal{C}(\mathbb{T}^d)}$ with $\lambda>\frac{1}{2}$ as in \eqref{eq:Wiener_algebra_inclusion}.
	Hence, for $f\in\mathcal{A}^{m-\lambda}(\mathbb{T}^d)\cap\mathcal{C}(\mathbb{T}^d)$ we have the approximation error bound 
	\begin{align}
		\label{eq:Linf_proof_1}
		\left\|f - S_{I_N^d}^{\Lambda} f \right\|_{L_{\infty}(\mathbb{T}^d)}
		\leq 2 N^{-m+\lambda} \|f\|_{\mathcal{A}^{m-\lambda}(\mathbb{T}^d)}
	\end{align}
	as stated in Theorem~\ref{thm:L_infty_approx_error_torus}.
	
	With the inverse transformation $\mathbf x=\psi^{-1}(\mathbf y)$ we have
	\begin{align*}
		\hat h_{\mathbf k}
		= \left(h, \varphi_{\mathbf k} \right)_{L_2\left(\left[-\frac{1}{2},\frac{1}{2}\right]^d, \omega \right)}
		= (f,\mathrm{e}^{2\pi\mathrm i \mathbf k\cdot\circ})_{L_2(\mathbb{T}^d)}
		= \hat f_{\mathbf k}
	\end{align*}
	and 
	\begin{align}
		\label{eq:Linf_proof_2}
		\|h\|_{\mathcal{H}^{m}\left(\left[-\frac{1}{2},\frac{1}{2}\right]^d\right)}^2
		= \sum_{\mathbf k\in\mathbb{Z}^d} \omega_{\mathrm{hc}}(\mathbf k)^{2m} |\hat h_{\mathbf k}|^2
		= \sum_{\mathbf k\in\mathbb{Z}^d} \omega_{\mathrm{hc}}(\mathbf k)^{2m} |\hat f_{\mathbf k}|^2		
		= \|f\|_{\mathcal{H}^{m}(\mathbb{T}^d)}^2,
	\end{align}
	as well as
	\begin{align*}
		&\left\| h - S_{I_{N}^{d}}h \right\|_{L_{\infty}\left(\left[-\frac{1}{2},\frac{1}{2}\right]^d, \sqrt{\frac{\omega}{\varrho}}\right)} \\
		&= \mathrm{esssup}_{\mathbf y\in\left[-\frac{1}{2},\frac{1}{2}\right]^d} \left| \sqrt{\frac{\omega(\mathbf y)}{\varrho(\mathbf y)}} \left( h(\mathbf y) - \sum_{\mathbf k\in I_{N}^{d}} \hat h_{\mathbf k} \, \varphi_{\mathbf k}(\mathbf y)\right) \right| \\
		&= \mathrm{esssup}_{\mathbf y\in\left[-\frac{1}{2},\frac{1}{2}\right]^d} \left| h(\mathbf y)\,\sqrt{\frac{\omega(\mathbf y)}{\varrho(\mathbf y)}} - \sum_{\mathbf k\in I_{N}^{d}} \hat h_{\mathbf k} \, \mathrm{e}^{2\pi\mathrm i \mathbf k\cdot\psi^{-1}(\mathbf{y})} \right| \\
		&= \mathrm{esssup}_{\mathbf x\in\mathbb{T}^{d}} \left| h(\psi(\mathbf x))\,\sqrt{\omega(\psi(\mathbf x))\prod_{j=1}^{d}\psi'_j(x_j)} - \sum_{\mathbf k\in I_{N}^{d}} \hat h_{\mathbf k} \, \mathrm{e}^{2\pi\mathrm i \mathbf k\cdot\mathbf{x}} \right| \\
		&= \left\|f - S_{I_{N}^{d}}f \right\|_{L_{\infty}(\mathbb{T}^d)}
	\end{align*}
	and 
	\begin{align}
		\label{eq:Linf_proof_3}
		\left\| h - S_{I_{N}^{d}}^{\Lambda}h \right\|_{L_{\infty}\left(\left[-\frac{1}{2},\frac{1}{2}\right]^d, \sqrt{\frac{\omega}{\varrho}}\right)} 
		= \left\|f - S_{I_{N}^{d}}^{\Lambda}f \right\|_{L_{\infty}(\mathbb{T}^d)}.
	\end{align} 
	In total, by combining \eqref{eq:Linf_proof_3}, \eqref{eq:Linf_proof_1}, \eqref{eq:Wiener_algebra_inclusion2} and \eqref{eq:Linf_proof_2} we estimate for the function $f\in\mathcal{H}^{m}(\mathbb{T}^d)\cap\mathcal{C}(\mathbb{T}^d)$ that
	\begin{align*}
		\left\| h - S_{I_N^d}^{\Lambda} h \right\|_{L_{\infty}\left(\left[-\frac{1}{2},\frac{1}{2}\right]^d, \sqrt{\frac{\omega}{\varrho}}\right)} 
		&= \left\| f - S_{I_N^d}^{\Lambda} f \right\|_{L_{\infty}(\mathbb{T}^d)}\\
		&\leq 2 N^{-m+\lambda} \|f\|_{\mathcal{A}^{m-\lambda}(\mathbb{T}^d)} \\
		&\leq 2\, C_{d,\lambda} N^{-m+\lambda} \|f\|_{\mathcal{H}^{m}(\mathbb{T}^d)}\\
		&= 2\,C_{d,\lambda} N^{-m+\lambda} \|h\|_{\mathcal{H}^{m}\left(\left[-\frac{1}{2},\frac{1}{2}\right]^d\right)}
		< \infty
	\end{align*}
	with $\lambda > \frac{1}{2}$ and some constant $C_{d,\lambda} > 1$.
\end{proof}

\subsubsection{$L_2$-approximation error}
Based on the $L_{2}(\mathbb{T}^d)$-approximation error bound \eqref{eq:H_beta_error_bound} and the conditions proposed in Theorem~\ref{thm:Cm_composition_criteria_mult} we prove an upper bound for the approximation error $\left\| h - S_{I_N^d}^{\Lambda} h \right\|$ in terms of a weighted $L_{2}$-norm on $\left[-\frac{1}{2},\frac{1}{2}\right]^d$.
\begin{theorem}\label{thm:Hm_approx_error_decay_multivar}
	Let $d\in\mathbb{N}$, $m\in\mathbb{N}_{0}$, a hyperbolic cross $I_{N}^{d}$ with $N\geq 2^{d+1}$ and a reconstructing {rank-$1$} lattice $\Lambda(\mathbf z, M, I_{N}^{d})$ be given.
	Let $\omega$ be a weight function as in \eqref{eq:omega_weighted_mult},
	${h\in L_2\left(\left[-\frac{1}{2},\frac{1}{2}\right]^d,\omega\right)\cap \mathcal{C}_{\mathrm{mix}}^{m}\left(\left[-\frac{1}{2},\frac{1}{2}\right]^d\right)}$,
	and let $\psi$ be a multivariate transformation as in \eqref{def:Trafo_cube_mult}.
	For all multi-indices $\mathbf m = (m_1,\ldots,m_d)^{\top} \in\mathbb{N}_{0}^{d}$ with $\|\mathbf m\|_{\ell_{\infty}} \leq m$
	we assume that
	\begin{align*}
		{\psi\in\mathcal{C}_{\mathrm{mix}}^{m}\left(\left[-\frac{1}{2},\frac{1}{2}\right]^d\right)} 
		\quad \text{and} \quad 
		D^{\mathbf m}\left[\prod_{\ell=1}^{d}\sqrt{(\omega_{\ell}\circ\psi_{\ell})\psi'_{\ell}}\right]\in \mathcal{C}_{0}\left(\left[-\frac{1}{2},\frac{1}{2}\right]^d\right).
	\end{align*}
	
	Then there is an approximation error estimate of the form
	\begin{align*}
		\left\| h - S_{I_{N}^{d}}^{\Lambda}h \right\|_{L_{2}\left(\left[-\frac{1}{2},\frac{1}{2}\right]^d, \omega \right)} 
		\lesssim N^{-m} (\log N)^{(d-1)/2} \|h\|_{\mathcal{H}^{m}\left(\left[-\frac{1}{2},\frac{1}{2}\right]^d\right)}.
	\end{align*}
\end{theorem}
\begin{proof}
	Let $m\in\mathbb{N}_{0}, d\in\mathbb{N}$ and $h\in{L_2\left(\left[-\frac{1}{2},\frac{1}{2}\right]^d,\omega\right)\cap \mathcal{C}_{\mathrm{mix}}^{m}\left(\left[-\frac{1}{2},\frac{1}{2}\right]^d\right)}$.
By assumption the criteria of Theorem~\ref{thm:Cm_composition_criteria_mult} are fulfilled 
	and the transformed function $f$ of the form \eqref{eq:f_is_transformed_h_mult} is continuously extendable to the torus $\mathbb{T}^d$.
	Thus, we have $f \in \mathcal{H}^{m}(\mathbb{T}^d)$ and a continuous representative, because of the inclusion $\mathcal{H}^{m}(\mathbb{T}^d) \hookrightarrow \mathcal{C}(\mathbb{T}^d)$ as in \eqref{eq:Wiener_algebra_inclusion}.
	For $f\in\mathcal{H}^{m}(\mathbb{T}^d) \cap \mathcal{C}(\mathbb{T}^d)$ Theorem~\ref{eq:L_2_approximation_error_bound} yields the approximation error bound of the form
	\begin{align}\label{eq:L2_proof_1}
		\left\| f - S_{I_{N}^{d}}^{\Lambda}f \right\|_{L_{2}(\mathbb{T}^d)} 
		\leq C_{d,\beta} N^{-\beta} (\log N)^{(d-1)/2} \|f\|_{\mathcal{H}^{\beta}(\mathbb{T}^d)}
	\end{align}
	with some constant $C_{d,\beta} := C(d,\beta)>0$.
	With the inverse transformation $\bm x=\psi^{-1}(\bm y)$ we have
	\begin{align*} 
		\hat h_{\mathbf k}
		= \left(h, \varphi_{\mathbf k} \right)_{L_2\left(\left[-\frac{1}{2},\frac{1}{2}\right]^d, \omega\right)}
		= (f,\mathrm{e}^{2\pi\mathrm i \mathbf k\cdot\circ})_{L_2(\mathbb{T}^d)}
		= \hat f_{\mathbf k},
	\end{align*}
	and
	\begin{align*}
		\|h\|_{\mathcal{H}^{m}\left(\left[-\frac{1}{2},\frac{1}{2}\right]^d\right)}^2
		= \sum_{\mathbf k\in\mathbb{Z}^d} \omega_{\mathrm{hc}}(\mathbf k)^{2m} |\hat h_{\mathbf k}|^2
		= \sum_{\mathbf k\in\mathbb{Z}^d} \omega_{\mathrm{hc}}(\mathbf k)^{2m} |\hat f_{\mathbf k}|^2		
		= \|f\|_{\mathcal{H}^{m}(\mathbb{T}^d)}^2
	\end{align*}
	as in \eqref{eq:Linf_proof_2}, as well as 
	\begin{align} \label{eq:L2_proof_3}
		\left\| h - S_{I_{N}^{d}}h \right\|_{L_{2}\left(\left[-\frac{1}{2},\frac{1}{2}\right]^d, \omega\right)}^2
		&= \int_{\left[-\frac{1}{2},\frac{1}{2}\right]^d} \left| h(\mathbf y) - \sum_{\mathbf k\in I_{N}^{d}} \hat h_{\mathbf k} \, \varphi_{\mathbf k}(\mathbf y) \right|^2 \omega(\mathbf y) \,\mathrm{d}\mathbf y \nonumber\\
		&= \left\| f - S_{I_{N}^{d}}f \right\|_{L_{2}(\mathbb{T}^d)}^2
	\end{align}	
	and	$\left\| h - S_{I_{N}^{d}}^{\Lambda}h \right\|_{L_{2}\left(\left[-\frac{1}{2},\frac{1}{2}\right]^d, \omega \right)}
		= \left\|f - S_{I_{N}^{d}}^{\Lambda}f \right\|_{L_{2}(\mathbb{T}^d)}.$
In total, by combining \eqref{eq:L2_proof_3}, \eqref{eq:L2_proof_1}, and \eqref{eq:Linf_proof_2} we estimate for $f\in\mathcal{H}^{m}(\mathbb{T}^d)\cap\mathcal{C}(\mathbb{T}^d)$ that
	\begin{align*}
		\left\| h - S_{I_N^d}^{\Lambda} h \right\|_{L_{2}\left(\left[-\frac{1}{2},\frac{1}{2}\right]^d, \omega\right)}
		&= \left\| f - S_{I_N^d}^{\Lambda} f \right\|_{L_{2}(\mathbb{T}^d)} \\
		&\lesssim C_{d,\beta} N^{-\beta} (\log N)^{(d-1)/2} \|f\|_{\mathcal{H}^{\beta}(\mathbb{T}^d)} \\
		&= C_{d,\beta} N^{-\beta} (\log N)^{(d-1)/2} \|h\|_{\mathcal{H}^{m}\left(\left[-\frac{1}{2},\frac{1}{2}\right]^d\right)}
		< \infty
	\end{align*}
	with some constant $C_{d,\beta} > 0$.
\end{proof}

\section{Algorithms}
In this chapter we start denoting the parameters $\bm\eta,\bm\mu\in\mathbb{R}_{+}^d$.
Families of multivariate parameterized weight functions are denoted by $\omega(\circ, \bm \mu)$ as in \eqref{eq:omega_weighted_mult} 
and for families of multivariate torus-to-cube transformations we use the notation $\psi(\circ,\bm\eta)$ to represent all possible torus-to-cube transformations in the sense of definition~\eqref{def:Trafo_cube_mult}. 
Furthermore, all related functions and objects are now written with a parameter argument, too. 

We adjust the algorithms described in \cite[Algorithm~3.1 and 3.2]{kaemmererdiss} that are based on one-dimensional fast Fourier transforms (FFTs).
They are used for the fast reconstruction of the approximated Fourier coefficients $\hat h_{\mathbf{k}}^{\Lambda}$ and the evaluation of the transformed multivariate trigonometric polynomials, in particular the approximated Fourier series $S_{I}^{\Lambda} h$, both given in \eqref{def:mult_approximated_FPS_FC}.
Both procedures are denoted as matrix-vector-products of the form
\begin{align*}\mathbf h = \mathbf A\mathbf{\hat h}
	\quad \text{and} \quad 
	\mathbf{\hat{h}} = M^{-1}\mathbf A^*\mathbf{h}
\end{align*}
with $\bm\eta,\bm\mu\in\mathbb{R}_{+}^d, \mathbf{h} := \left(h(\mathbf y_j)\, \sqrt{\frac{\omega(\mathbf y_j, \bm\mu)}{\varrho(\mathbf y_j, \bm\eta)}}\right)_{j=0,\ldots,M-1}$ for $\mathbf y_j \in \Lambda_{\psi(\circ,\bm\eta)}(\mathbf z,M)$, $\hat{\mathbf{h}} := (\hat h_{\mathbf k})_{\mathbf k \in I}$, the transformed Fourier matrices $\mathbf A\in \mathbb{C}^{M \times |I|}$ and its adjoint $\mathbf A^{*}\in \mathbb{C}^{|I| \times M}$ given by
\begin{align*}
	\mathbf A 
	&
	:= \left( \sqrt{\frac{\omega(\mathbf y_j,\bm\mu)}{\varrho(\mathbf y_j,\bm\eta)}} \, \varphi_{\mathbf k}(\mathbf y_j) \right)_{\mathbf y_j \in \Lambda_{\psi(\circ,\bm\eta)}(\mathbf z,M),\, \mathbf k\in I}, \nonumber\\
	\mathbf A^{*} 
	&
	:= \left( \sqrt{\frac{\omega(\mathbf y_j,\bm\mu)}{\varrho(\mathbf y_j,\bm\eta)}} \, \overline{\varphi_{\mathbf k}(\mathbf y_j)} \right)_{\mathbf k\in I, \mathbf y_j \in \Lambda_{\psi(\circ,\bm\eta)}(\mathbf z,M)}.
\end{align*}
We incorporate the previously described idea of transforming a function 
$h\in L_2\left(\left[-\frac{1}{2},\frac{1}{2}\right]^d,\omega\right)\cap \mathcal{C}_{\mathrm{mix}}^{m}\left(\left[-\frac{1}{2},\frac{1}{2}\right]^d\right)$ 
by a transformation $\mathbf y_j = \psi(\mathbf x_j,\bm\eta)$ with $\mathbf x_j = (x_1^{j},\ldots, x_d^{j})$
into a periodic function $f$ on the torus $\mathbb T^d$ that is of the form \eqref{eq:f_is_transformed_h_mult}.
The resulting samples are of the form
\begin{align} \label{eq:Input_sample_points}
	h(\mathbf y_j)\, \sqrt{\frac{\omega(\mathbf y_j, \bm\mu)}{\varrho(\mathbf y_j, \bm\eta)}} 
	= h(\psi(\mathbf x_j,\bm\eta))  \sqrt{ \omega(\psi(\mathbf x_j,\bm\eta), \bm\mu)  \prod_{k=1}^{d}\psi_{k}'(x_k^{j}, \eta_k) }
	= f(\mathbf x_j,\bm\eta,\bm\mu)
\end{align}
with the parameters $\bm\eta,\bm\mu\in\mathbb{R}_{+}^d$.
\begin{remark}
	We identify $\mathbb T^d$ with different cubes.
	We consider $\mathbb T^d \simeq[0,1)^d$ when defining {rank-$1$} lattices $\Lambda(\mathbf z, M)$ in \eqref{def:rank_one_lattice}.
	However, we consider $\mathbb T^d \simeq [-\frac{1}{2},\frac{1}{2})^d$ when applying a transformation $\psi$ to a rank-$1$ lattice. 
	In this process, we reassign all lattice points $\mathbf x_j\in\Lambda(\mathbf z, M)$ via
	\begin{align*}
		{\mathbf{x}}_j \mapsto \left( \left( \mathbf{x}_j + \frac{\mathbf 1}{2} \right) \bmod \mathbf 1 \right) - \frac{\mathbf 1}{2}
	\end{align*}
	for all $j=0,\ldots,M-1$.
\end{remark}	
In Figure~\ref{fig:Lattice_and_transformed_lattice} we showcase different two-dimensional transformed {rank-$1$} lattices $\Lambda_{\psi(\circ,\bm\eta)}(\mathbf z, M)$ as defined in \eqref{eq:Def_trafo_lattice}, generated by $\mathbf z = (1,7)^{\top}$ and $M = 150$.
We compare the lattices transformed by the sine transformation~\eqref{eq:sine_trafo} with our previously introduced torus-to-cube transformations \eqref{eq:logarithmic_trafo_comb} with the parameter vector $\bm\eta = (\eta_1, \eta_2)^{\top} = (3, 3)^{\top}$.
\begin{figure}[tb]
	\centering
	\begin{tikzpicture}[scale=0.55]
	\begin{axis}[scatter/classes = { a = {mark=o, draw=black} },
	font=\footnotesize,	grid = both,
	xmax = 0.75,	xmin = -0.75,	ymax = 0.75,	ymin = -0.75,
	xtick={-0.5,-0.25,0,0.25,0.5}, ytick={-0.5,-0.25,0,0.25,0.5},
	title = {$\Lambda(\mathbf z, M)$},
	unit vector ratio*=1 1 1
	]
	\addplot[scatter ,only marks, mark size=1.0, scatter src = explicit symbolic]coordinates{
		(0, 0)  (0.006667, 0.04667)  (0.01333, 0.09333)  (0.02, 0.14)  (0.02667, 0.1867)  (0.03333, 0.2333)  (0.04, 0.28)  (0.04667, 0.3267)  (0.05333, 0.3733)  (0.06, 0.42)  (0.06667, 0.4667)  (0.07333, -0.4867)  (0.08, -0.44)  (0.08667, -0.3933)  (0.09333, -0.3467)  (0.1, -0.3)  (0.1067, -0.2533)  (0.1133, -0.2067)  (0.12, -0.16)  (0.1267, -0.1133)  (0.1333, -0.06667)  (0.14, -0.02)  (0.1467, 0.02667)  (0.1533, 0.07333)  (0.16, 0.12)  (0.1667, 0.1667)  (0.1733, 0.2133)  (0.18, 0.26)  (0.1867, 0.3067)  (0.1933, 0.3533)  (0.2, 0.4)  (0.2067, 0.4467)  (0.2133, 0.4933)  (0.22, -0.46)  (0.2267, -0.4133)  (0.2333, -0.3667)  (0.24, -0.32)  (0.2467, -0.2733)  (0.2533, -0.2267)  (0.26, -0.18)  (0.2667, -0.1333)  (0.2733, -0.08667)  (0.28, -0.04)  (0.2867, 0.006667)  (0.2933, 0.05333)  (0.3, 0.1)  (0.3067, 0.1467)  (0.3133, 0.1933)  (0.32, 0.24)  (0.3267, 0.2867)  (0.3333, 0.3333)  (0.34, 0.38)  (0.3467, 0.4267)  (0.3533, 0.4733)  (0.36, -0.48)  (0.3667, -0.4333)  (0.3733, -0.3867)  (0.38, -0.34)  (0.3867, -0.2933)  (0.3933, -0.2467)  (0.4, -0.2)  (0.4067, -0.1533)  (0.4133, -0.1067)  (0.42, -0.06)  (0.4267, -0.01333)  (0.4333, 0.03333)  (0.44, 0.08)  (0.4467, 0.1267)  (0.4533, 0.1733)  (0.46, 0.22)  (0.4667, 0.2667)  (0.4733, 0.3133)  (0.48, 0.36)  (0.4867, 0.4067)  (0.4933, 0.4533)  (-0.5, -0.5)  (-0.4933, -0.4533)  (-0.4867, -0.4067)  (-0.48, -0.36)  (-0.4733, -0.3133)  (-0.4667, -0.2667)  (-0.46, -0.22)  (-0.4533, -0.1733)  (-0.4467, -0.1267)  (-0.44, -0.08)  (-0.4333, -0.03333)  (-0.4267, 0.01333)  (-0.42, 0.06)  (-0.4133, 0.1067)  (-0.4067, 0.1533)  (-0.4, 0.2)  (-0.3933, 0.2467)  (-0.3867, 0.2933)  (-0.38, 0.34)  (-0.3733, 0.3867)  (-0.3667, 0.4333)  (-0.36, 0.48)  (-0.3533, -0.4733)  (-0.3467, -0.4267)  (-0.34, -0.38)  (-0.3333, -0.3333)  (-0.3267, -0.2867)  (-0.32, -0.24)  (-0.3133, -0.1933)  (-0.3067, -0.1467)  (-0.3, -0.1)  (-0.2933, -0.05333)  (-0.2867, -0.006667)  (-0.28, 0.04)  (-0.2733, 0.08667)  (-0.2667, 0.1333)  (-0.26, 0.18)  (-0.2533, 0.2267)  (-0.2467, 0.2733)  (-0.24, 0.32)  (-0.2333, 0.3667)  (-0.2267, 0.4133)  (-0.22, 0.46)  (-0.2133, -0.4933)  (-0.2067, -0.4467)  (-0.2, -0.4)  (-0.1933, -0.3533)  (-0.1867, -0.3067)  (-0.18, -0.26)  (-0.1733, -0.2133)  (-0.1667, -0.1667)  (-0.16, -0.12)  (-0.1533, -0.07333)  (-0.1467, -0.02667)  (-0.14, 0.02)  (-0.1333, 0.06667)  (-0.1267, 0.1133)  (-0.12, 0.16)  (-0.1133, 0.2067)  (-0.1067, 0.2533)  (-0.1, 0.3)  (-0.09333, 0.3467)  (-0.08667, 0.3933)  (-0.08, 0.44)  (-0.07333, 0.4867)  (-0.06667, -0.4667)  (-0.06, -0.42)  (-0.05333, -0.3733)  (-0.04667, -0.3267)  (-0.04, -0.28)  (-0.03333, -0.2333)  (-0.02667, -0.1867)  (-0.02, -0.14)  (-0.01333, -0.09333)  (-0.006667, -0.04667)
	};
	\end{axis}
	\end{tikzpicture}
	\begin{tikzpicture}[scale=0.55]
	\begin{axis}[scatter/classes = { a = {mark=o, draw=black} },
	font=\footnotesize,	grid = both,
	xmax = 0.75,	xmin = -0.75,	ymax = 0.75,	ymin = -0.75,
	title style={align=left},
	title = {$\Lambda_{\psi(\circ)}(\mathbf z, M)$ with\\ $\psi(x_j) = \frac{1}{2}\,\sin(\pi x_j)$ },
	unit vector ratio*=1 1 1
	]
	\addplot[scatter ,only marks, mark size=1.0, scatter src = explicit symbolic]coordinates{
		(0, 0)  (0.01047, 0.07304)  (0.02094, 0.1445)  (0.0314, 0.2129)  (0.04184, 0.2767)  (0.05226, 0.3346)  (0.06267, 0.3853)  (0.07304, 0.4277)  (0.08338, 0.4609)  (0.09369, 0.4843)  (0.104, 0.4973)  (0.1142, -0.4996)  (0.1243, -0.4911)  (0.1345, -0.4722)  (0.1445, -0.4431)  (0.1545, -0.4045)  (0.1644, -0.3572)  (0.1743, -0.3023)  (0.1841, -0.2409)  (0.1938, -0.1743)  (0.2034, -0.104)  (0.2129, -0.0314)  (0.2223, 0.04184)  (0.2316, 0.1142)  (0.2409, 0.1841)  (0.25, 0.25)  (0.259, 0.3106)  (0.2679, 0.3645)  (0.2767, 0.4106)  (0.2854, 0.4479)  (0.2939, 0.4755)  (0.3023, 0.493)  (0.3106, 0.4999)  (0.3187, -0.4961)  (0.3267, -0.4816)  (0.3346, -0.4568)  (0.3423, -0.4222)  (0.3498, -0.3785)  (0.3572, -0.3267)  (0.3645, -0.2679)  (0.3716, -0.2034)  (0.3785, -0.1345)  (0.3853, -0.06267)  (0.3918, 0.01047)  (0.3983, 0.08338)  (0.4045, 0.1545)  (0.4106, 0.2223)  (0.4165, 0.2854)  (0.4222, 0.3423)  (0.4277, 0.3918)  (0.433, 0.433)  (0.4382, 0.4649)  (0.4431, 0.4868)  (0.4479, 0.4982)  (0.4524, -0.499)  (0.4568, -0.4891)  (0.4609, -0.4686)  (0.4649, -0.4382)  (0.4686, -0.3983)  (0.4722, -0.3498)  (0.4755, -0.2939)  (0.4787, -0.2316)  (0.4816, -0.1644)  (0.4843, -0.09369)  (0.4868, -0.02094)  (0.4891, 0.05226)  (0.4911, 0.1243)  (0.493, 0.1938)  (0.4946, 0.259)  (0.4961, 0.3187)  (0.4973, 0.3716)  (0.4982, 0.4165)  (0.499, 0.4524)  (0.4996, 0.4787)  (0.4999, 0.4946)  (-0.5, -0.5)  (-0.4999, -0.4946)  (-0.4996, -0.4787)  (-0.499, -0.4524)  (-0.4982, -0.4165)  (-0.4973, -0.3716)  (-0.4961, -0.3187)  (-0.4946, -0.259)  (-0.493, -0.1938)  (-0.4911, -0.1243)  (-0.4891, -0.05226)  (-0.4868, 0.02094)  (-0.4843, 0.09369)  (-0.4816, 0.1644)  (-0.4787, 0.2316)  (-0.4755, 0.2939)  (-0.4722, 0.3498)  (-0.4686, 0.3983)  (-0.4649, 0.4382)  (-0.4609, 0.4686)  (-0.4568, 0.4891)  (-0.4524, 0.499)  (-0.4479, -0.4982)  (-0.4431, -0.4868)  (-0.4382, -0.4649)  (-0.433, -0.433)  (-0.4277, -0.3918)  (-0.4222, -0.3423)  (-0.4165, -0.2854)  (-0.4106, -0.2223)  (-0.4045, -0.1545)  (-0.3983, -0.08338)  (-0.3918, -0.01047)  (-0.3853, 0.06267)  (-0.3785, 0.1345)  (-0.3716, 0.2034)  (-0.3645, 0.2679)  (-0.3572, 0.3267)  (-0.3498, 0.3785)  (-0.3423, 0.4222)  (-0.3346, 0.4568)  (-0.3267, 0.4816)  (-0.3187, 0.4961)  (-0.3106, -0.4999)  (-0.3023, -0.493)  (-0.2939, -0.4755)  (-0.2854, -0.4479)  (-0.2767, -0.4106)  (-0.2679, -0.3645)  (-0.259, -0.3106)  (-0.25, -0.25)  (-0.2409, -0.1841)  (-0.2316, -0.1142)  (-0.2223, -0.04184)  (-0.2129, 0.0314)  (-0.2034, 0.104)  (-0.1938, 0.1743)  (-0.1841, 0.2409)  (-0.1743, 0.3023)  (-0.1644, 0.3572)  (-0.1545, 0.4045)  (-0.1445, 0.4431)  (-0.1345, 0.4722)  (-0.1243, 0.4911)  (-0.1142, 0.4996)  (-0.104, -0.4973)  (-0.09369, -0.4843)  (-0.08338, -0.4609)  (-0.07304, -0.4277)  (-0.06267, -0.3853)  (-0.05226, -0.3346)  (-0.04184, -0.2767)  (-0.0314, -0.2129)  (-0.02094, -0.1445)  (-0.01047, -0.07304)      
	};
	\end{axis}
	\end{tikzpicture}
	\begin{tikzpicture}[scale=0.55]
	\begin{axis}[scatter/classes = { a = {mark=o, draw=black} },
	font=\footnotesize,	grid = both,
	xmax = 0.75,	xmin = -0.75,	ymax = 0.75,	ymin = -0.75,
	title style={align=left},
	title = {$\Lambda_{\psi(\circ,\bm 3)}(\mathbf z, M)$ with\\ $\psi_j(x_j,3) = \frac{1}{2}\,\frac{(1+2x_j)^3 - (1-2x_j)^3}{(1+2x_j)^3 + (1-2x_j)^3}$ },
	unit vector ratio*=1 1 1
	]
	\addplot[scatter ,only marks, mark size=1.0, scatter src = explicit symbolic]coordinates{
		(0, 0)  (0.01999, 0.1368)  (0.03992, 0.2564)  (0.05975, 0.3489)  (0.0794, 0.4132)  (0.09883, 0.4541)  (0.118, 0.4781)  (0.1368, 0.4909)  (0.1553, 0.497)  (0.1734, 0.4993)  (0.191, 0.5)  (0.2081, -0.5)  (0.2248, -0.4997)  (0.2409, -0.4983)  (0.2564, -0.4941)  (0.2714, -0.4846)  (0.2858, -0.4661)  (0.2996, -0.4333)  (0.3129, -0.3797)  (0.3255, -0.2996)  (0.3375, -0.191)  (0.3489, -0.05975)  (0.3598, 0.0794)  (0.37, 0.2081)  (0.3797, 0.3129)  (0.3889, 0.3889)  (0.3975, 0.4391)  (0.4056, 0.4695)  (0.4132, 0.4864)  (0.4204, 0.4949)  (0.427, 0.4986)  (0.4333, 0.4998)  (0.4391, 0.5)  (0.4445, -0.4999)  (0.4495, -0.4991)  (0.4541, -0.4964)  (0.4584, -0.4895)  (0.4624, -0.4754)  (0.4661, -0.4495)  (0.4695, -0.4056)  (0.4726, -0.3375)  (0.4754, -0.2409)  (0.4781, -0.118)  (0.4804, 0.01999)  (0.4826, 0.1553)  (0.4846, 0.2714)  (0.4864, 0.3598)  (0.4881, 0.4204)  (0.4895, 0.4584)  (0.4909, 0.4804)  (0.4921, 0.4921)  (0.4931, 0.4975)  (0.4941, 0.4995)  (0.4949, 0.5)  (0.4957, -0.5)  (0.4964, -0.4996)  (0.497, -0.4979)  (0.4975, -0.4931)  (0.4979, -0.4826)  (0.4983, -0.4624)  (0.4986, -0.427)  (0.4989, -0.37)  (0.4991, -0.2858)  (0.4993, -0.1734)  (0.4995, -0.03992)  (0.4996, 0.09883)  (0.4997, 0.2248)  (0.4998, 0.3255)  (0.4999, 0.3975)  (0.4999, 0.4445)  (0.5, 0.4726)  (0.5, 0.4881)  (0.5, 0.4957)  (0.5, 0.4989)  (0.5, 0.4999)  (-0.5, -0.5)  (-0.5, -0.4999)  (-0.5, -0.4989)  (-0.5, -0.4957)  (-0.5, -0.4881)  (-0.5, -0.4726)  (-0.4999, -0.4445)  (-0.4999, -0.3975)  (-0.4998, -0.3255)  (-0.4997, -0.2248)  (-0.4996, -0.09883)  (-0.4995, 0.03992)  (-0.4993, 0.1734)  (-0.4991, 0.2858)  (-0.4989, 0.37)  (-0.4986, 0.427)  (-0.4983, 0.4624)  (-0.4979, 0.4826)  (-0.4975, 0.4931)  (-0.497, 0.4979)  (-0.4964, 0.4996)  (-0.4957, 0.5)  (-0.4949, -0.5)  (-0.4941, -0.4995)  (-0.4931, -0.4975)  (-0.4921, -0.4921)  (-0.4909, -0.4804)  (-0.4895, -0.4584)  (-0.4881, -0.4204)  (-0.4864, -0.3598)  (-0.4846, -0.2714)  (-0.4826, -0.1553)  (-0.4804, -0.01999)  (-0.4781, 0.118)  (-0.4754, 0.2409)  (-0.4726, 0.3375)  (-0.4695, 0.4056)  (-0.4661, 0.4495)  (-0.4624, 0.4754)  (-0.4584, 0.4895)  (-0.4541, 0.4964)  (-0.4495, 0.4991)  (-0.4445, 0.4999)  (-0.4391, -0.5)  (-0.4333, -0.4998)  (-0.427, -0.4986)  (-0.4204, -0.4949)  (-0.4132, -0.4864)  (-0.4056, -0.4695)  (-0.3975, -0.4391)  (-0.3889, -0.3889)  (-0.3797, -0.3129)  (-0.37, -0.2081)  (-0.3598, -0.0794)  (-0.3489, 0.05975)  (-0.3375, 0.191)  (-0.3255, 0.2996)  (-0.3129, 0.3797)  (-0.2996, 0.4333)  (-0.2858, 0.4661)  (-0.2714, 0.4846)  (-0.2564, 0.4941)  (-0.2409, 0.4983)  (-0.2248, 0.4997)  (-0.2081, 0.5)  (-0.191, -0.5)  (-0.1734, -0.4993)  (-0.1553, -0.497)  (-0.1368, -0.4909)  (-0.118, -0.4781)  (-0.09883, -0.4541)  (-0.0794, -0.4132)  (-0.05975, -0.3489)  (-0.03992, -0.2564)  (-0.01999, -0.1368)   
	};
	\end{axis}
	\end{tikzpicture}
\caption{A two-dimensional lattice $\Lambda(\mathbf z,M)$ with $\mathbf z = (1,7)^{\top}, M = 150$ in the top left, the transformed lattice $\Lambda_{\psi(\circ,\bm\eta)}(\mathbf z, M)$ for the sine transformation \eqref{eq:sine_trafo} in the bottom left,
	and resulting lattices from applying the logarithmic transformation \eqref{eq:logarithmic_trafo_comb} with $\bm\eta=\mathbf 3$.
	}
	\label{fig:Lattice_and_transformed_lattice}
\end{figure}
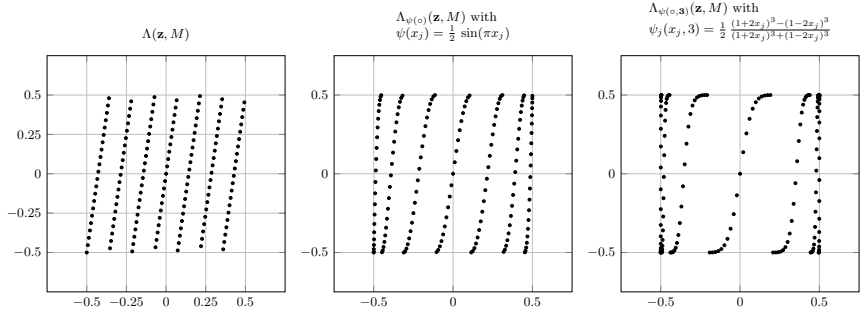

\subsection{Evaluation of transformed multivariate trigonometric polynomials}
\begin{algorithm}[tb]
\caption{Evaluation at {rank-$1$} lattice}
\label{alg:LFFT_eval}
\begin{tabular}{p{1.25cm}p{4.75cm}p{4.60cm}}
	Input: 	
	& $M\in\mathbb{N}$ & lattice size of $\Lambda_{\psi(\circ,\bm\eta)}(\mathbf z,M)$\\
	& $\mathbf z\in\mathbb{Z}^d$ & generating vector of $\Lambda_{\psi(\circ,\bm\eta)}(\mathbf z,M)$\\
	& $I\subset\mathbb{Z}^d$ & frequency set of finite cardinality\\
	& $\mathbf{\hat h} = \left(\hat{h}_{\mathbf k}\right)_{\mathbf k\in I}$ & Fourier coefficients of $h\in\Pi_{I,\psi(\circ,\bm\eta)}$
\end{tabular}
\begin{algorithmic}
	\STATE $\mathbf {\hat{g}}=\left(0\right)_{l=0}^{M-1}$
	\FOR{\textbf{each} $\mathbf k\in I$}
	\STATE  $\hat{g}_{\mathbf k\cdot\mathbf z\bmod{M}} = \hat{g}_{\mathbf k\cdot\mathbf z\bmod{M}} + \hat{h}_{\mathbf k}$
	\ENDFOR
	\STATE $\mathbf h=\mathrm{iFFT\_1D}(\mathbf{\hat{g}})$
	\STATE $\mathbf h=M\mathbf h$
\end{algorithmic}
\begin{tabular}{p{1.25cm}p{4.75cm}p{4.6cm}}
	Output: & $\mathbf h = \mathbf A\mathbf{\hat h} = \left(h(\mathbf y_j)\, \sqrt{\frac{\omega(\mathbf y_j, \bm\mu)}{\varrho(\mathbf y_j, \bm\eta)}}\right)_{j=0}^{M-1}$ & function values of $h\in\Pi_{I,\psi(\circ,\bm\eta)}$
\end{tabular}
\end{algorithm}
\noindent
Given a frequency set $I\subset\mathbb{Z}^d$ of finite cardinality $|I|<\infty$, we consider the multivariate trigonometric polynomial $h\in\Pi_{I,\psi(\circ,\bm\eta)}$ as in \eqref{def:trig_poly_trafo_mult}
with Fourier coefficients $\hat{h}_{\mathbf k} $. The evaluation of $h$ at transformed lattice points $\mathbf y_j \in \Lambda_{\psi(\circ,\bm\eta)}(\mathbf z, M)$ simplifies to
\begin{align*}
	h(\mathbf y_j)\, \sqrt{\frac{\omega(\mathbf y_j, \bm\mu)}{\varrho(\mathbf y_j, \bm\eta)}}
	&= \sum_{\mathbf k\in I} \hat{h}_{\mathbf k} \,\mathrm{e}^{2\pi\mathrm i \mathbf{k} \cdot \psi^{-1}(\mathbf y_j,\bm\eta)} \\
	&= \sum_{\ell = 0}^{M-1} 
		\left( \sum_{ \substack{\mathbf k\in I,\\ \mathbf k \cdot \mathbf z \equiv \ell \,(\bmod{M})} } \hat{h}_{\mathbf k}
		\right)	\,\mathrm{e}^{2\pi\mathrm i \ell \frac{j}{M}} 
	= \sum_{\ell = 0}^{M-1} \hat g_\ell \,\mathrm{e}^{2\pi\mathrm i \ell \frac{j}{M}},
\end{align*}
with
\begin{align*}
	\hat g_\ell
	= \sum_{ \substack{\mathbf k\in I,\\ \mathbf k\cdot\mathbf z \equiv \ell \,(\bmod{M})} } \hat{h}_{\mathbf k}. 
\end{align*}
In total, the evaluation of such a function is realized by simply pre-computing $(\hat g_\ell)_{\ell=0}^{M-1}$
and applying a one-dimensional inverse fast Fourier transform, see Algorithm~\ref{alg:LFFT_eval}.

\subsection{Reconstruction of transformed multivariate trigonometric polynomials}
\begin{algorithm}[tb]
\caption{Reconstruction from sampling values along a transformed reconstructing {rank-$1$} lattice} 
\label{alg:LFFT_recon}
\begin{tabular}{p{1.5cm}p{4.25cm}p{4.75cm}}
	Input: 	
	& $I\subset\mathbb{Z}^d$ & frequency set of finite cardinality\\
	& $M\in\mathbb{N}$ & lattice size of $\Lambda_{\psi(\circ,\bm\eta)}(\mathbf z,M,I)$\\
	& $\mathbf z\in\mathbb{Z}^d$ & generating vector of $\Lambda_{\psi(\circ,\bm\eta)}(\mathbf z,M,I)$\\
	& $\mathbf{h} = \left(h(\mathbf y_j)\, \sqrt{\frac{\omega(\mathbf y_j, \bm\mu)}{\varrho(\mathbf y_j, \bm\eta)}}\right)_{j=0}^{M-1}$ & function values of $h\in\Pi_{I,\psi(\circ,\bm\eta)}$
\end{tabular}
\begin{algorithmic}
	\STATE $\mathbf {\hat{g}}=\mathrm{FFT\_1D}(\mathbf h)$
	\FOR{\textbf{each} $\mathbf k\in I$}
	\STATE  $\hat{h}_{\mathbf k}=\frac{1}{M}\hat{g}_{\mathbf k\cdot\mathbf z\bmod{M}}$
	\ENDFOR
\end{algorithmic}
\begin{tabular}{p{1.5cm}p{4.25cm}p{4.75cm}}
	Output: & $\mathbf{\hat{h}} = M^{-1}\mathbf A^*\mathbf{h} = \left(\hat{h}_{\mathbf k}\right)_{\mathbf k\in I}$ & Fourier coefficients supported on $I$
\end{tabular}
\end{algorithm}
\noindent
For the reconstruction of a multivariate trigonometric polynomial $h\in\Pi_{I,\psi(\circ,\bm\eta)}$ as in \eqref{def:trig_poly_trafo_mult} from transformed lattice points $\mathbf y_j\in\Lambda_{\psi(\circ,\bm\eta)}(\mathbf z, M, I)$ we utilize the exact integration property \eqref{eq:exact_integration_prop_2} and the fact that we have
\begin{align}
	\label{eq:reconR1L_matrix_prop}
	\sum_{j=0}^{M-1} \left( \mathrm{e}^{2\pi\mathrm i\frac{(\mathbf k-\mathbf h)\cdot\mathbf z}{M} } \right)^{j} = 
	\begin{cases}
		M & \text{for } \mathbf{k}\cdot\mathbf{z} \equiv \mathbf{k}\cdot\mathbf{h} \,(\bmod{M}), \\
		0 & \text{otherwise},
	\end{cases}
\end{align}
and $\mathbf A^* \mathbf A = M \textbf{I}$ with $\textbf{I}\in\mathbb{C}^{|I|\times|I|}$ being the identity matrix.
For fixed parameters $\bm\eta,\bm\mu\in\mathbb{R}_{+}^d$ and $\mathbf x_j = (x_1^{j},\ldots, x_d^{j})^{\top}$ we have input sample points as in \eqref{eq:Input_sample_points} of the form 
\begin{align*}
	h(\mathbf y_j)\, \sqrt{\frac{\omega(\mathbf y_j, \bm\mu)}{\varrho(\mathbf y_j, \bm\eta)}}
	= h(\psi(\mathbf x_j,\bm\eta)) \, \sqrt{ \omega(\psi(\mathbf x_j,\bm\eta), \bm\mu) \, \prod_{k=1}^{d}\psi_{k}'(x_k^{j}, \eta_k) } 
	= f(\mathbf x_j,\bm\eta,\bm\mu).
\end{align*}
For the reconstruction of the Fourier coefficients $\hat{h}_{\mathbf k}$ we use a single one-dimensional fast Fourier transform. The entries of the resulting vector $\left( \hat g_\ell \right)_{\ell=0}^{M-1}$ are renumbered by means of the unique inverse mapping $\mathbf k \mapsto\mathbf k \cdot\mathbf z\bmod{M}$, see Algorithm~\ref{alg:LFFT_recon}.

\subsection{Discrete approximation error}
We use Algorithms~\ref{alg:LFFT_eval} and \ref{alg:LFFT_recon} to illustrate the error bounds of Theorems~\ref{thm:L_infty_approx_error_multivar} and \ref{thm:Hm_approx_error_decay_multivar}.
We discretize the approximation error 
	$\left\| h - S_{I}^{\Lambda}h \right\|_{L_{\infty}\left(\left[-\frac{1}{2},\frac{1}{2}\right]^d, \sqrt{\frac{\omega(\circ,\bm\mu)}{\varrho(\circ,\bm\eta)}}\right)}$
by sampling both the test function $h$ and the approximated Fourier partial sum $S_{I}^{\Lambda}h$.
At first we use the sample data vector 
$\mathbf{h}=\left(h(\mathbf y_j)\, \sqrt{\frac{\omega(\mathbf y_j, \bm\mu)}{\varrho(\mathbf y_j, \bm\eta)}}\right)_{j=0}^{M-1}$ with transformed lattice points ${\mathbf y_j\in\Lambda_{\psi(\circ,\bm\eta)}(\mathbf z, M, I)}$ and apply Algorithm~\ref{alg:LFFT_recon} yielding a vector of approximated Fourier coefficients via ${\mathbf{\hat{h}} = M^{-1}\mathbf A^*\mathbf{h}}$.
By putting this vector $\mathbf{\hat{h}}$ into Algorithm~\ref{alg:LFFT_eval} 
we have computed the vector 
$\mathbf{h}_{\mathrm{approx}} := M^{-1} \mathbf A \mathbf A^*\mathbf{h} = \left( \sqrt{\frac{\omega(\mathbf y_j, \bm\mu)}{\varrho(\mathbf y_j, \bm\eta)}} S_{I_N^d}^{\Lambda} h(\mathbf y_j) \right)_{j=0}^{M-1}$.
In \cite[Corollary~1]{Kae2013} and \cite[Theorem~2.1]{KaPoVo13} it was shown under mild assumptions that for each frequency set ${I\subset\mathbb{Z}^d}$ that induces a reconstructing {rank-$1$} lattice, there is an $M\in\mathbb{N}$ such that ${|I| \leq M \lesssim |I|^2}$.
The upper bound can be improved to $M\leq C |I|\log|I|$ with high probability by using multiple {rank-$1$} lattices \cite{Kae16,KaPoVo17}.
For a reconstructing {rank-$1$} lattice $\Lambda_{\psi(\circ,\bm\eta)}(\mathbf z,M, I)$ we have ${\mathbf A^* \mathbf A = M \textbf{I}}$ with $\textbf{I}\in\mathbb{C}^{|I|\times|I|}$ being the identity matrix, according to \eqref{eq:reconR1L_matrix_prop}.
However, $\mathbf A \mathbf A^* \in\mathbb{C}^{M\times M}$ is generally not an identity matrix.
Hence, there is a gap between the initially given values $\mathbf h$ and the resulting vector $\mathbf h_{\mathrm{approx}}$ that we measure with the \emph{relative discrete approximation error} 
\begin{align}
	\label{eq:Discrete_ellinfty_error}
	\varepsilon_{\infty}
	:= \frac{ \|\mathbf h - \mathbf h_{\mathrm{approx}}\|_{\ell_{\infty}} }{ \|\mathbf h \|_{\ell_{\infty}} }
	= \frac{ \max_{j=0,\ldots,M-1} \left| \sqrt{\frac{\omega(\mathbf y_j, \bm\mu)}{\varrho(\mathbf y_j, \bm\eta)}} \left( h(\mathbf y_j) - S_{I_N^d}^{\Lambda} h(\mathbf y_j) \right) \right| }{ \max_{j=0,\ldots,M-1} \left| \sqrt{\frac{\omega(\mathbf y_j, \bm\mu)}{\varrho(\mathbf y_j, \bm\eta)}} \, h(\mathbf y_j) \right| }.
\end{align}
This error is a discretization of the particular weighted $L_{\infty}$-norm appearing in Theorem~\ref{thm:L_infty_approx_error_multivar}.
For hyperbolic crosses $I_{N}^{d}$ and appropriately chosen parameters $\bm\eta,\bm\mu\in\mathbb{R}_{+}^d$ we have the upper bound
\begin{align}
	\label{eq:discrete_upper_bound}
	\|\mathbf h - \mathbf h_{\mathrm{approx}}\|_{\ell_{\infty}} 
	&\leq \left\|h - S_{I_{N}^{d}}^{\Lambda}h \right\|_{L_{\infty}\left(\left[-\frac{1}{2},\frac{1}{2}\right]^d,\sqrt{\frac{\omega(\circ,\bm\mu)}{\varrho(\circ,\bm\eta)}}\right)}\\
	&= \left\| f - S_{I_{N}^{d}}^{\Lambda}f \right\|_{L_{\infty}(\mathbb{T}^d)}
	\leq 2 N^{-m} \|h\|_{\mathcal{H}^{m}\left(\left[-\frac{1}{2},\frac{1}{2}\right]^d\right)}. \nonumber
\end{align} 
Hence, the theoretical results predict a certain decay rate of the discretized approximation error for increasing ${N\in\mathbb{N}}$ with fixed $m\in\mathbb{N}$ and suitably chosen parameters $\bm\eta$ and $\bm\mu$.

It's important to note that the particular discretization \eqref{eq:Discrete_ellinfty_error} was exclusively sampled at the {rank-$1$} lattice nodes, so that we don't measure the quality of the approximation at any point outside the {rank-$1$} lattice.
This limitation is overcome by oversampling.
For the $L_{2}$-approximation error we lack a similar discretization approach.
In Theorem~\ref{thm:Hm_approx_error_decay_multivar} we showed that for fixed $m\in\mathbb{N}$ and suitably chosen parameters $\bm\eta$ and $\bm\mu$ the approximation error 
$\left\| h - S_{I_{N}^{d}}^{\Lambda}h \right\|_{L_{2}\left(\left[-\frac{1}{2},\frac{1}{2}\right]^d,\omega\right)} = 
\left\|f - S_{I_{N}^{d}}^{\Lambda}f \right\|_{L_{2}(\mathbb{T}^d)}$
is estimated from above by a constant times ${N^{-m}(\log N)^{(d-1)/2}\|f\|_{\mathcal{H}^{m}(\mathbb{T}^d)}}$.
By Parseval's equation we have 
\begin{align*}
\left\| f - S_{I_{N}^{d}}^{\Lambda}f \right\|_{L_{2}(\mathbb{T}^d)}^2
	= \sum_{\mathbf k\in\mathbb{Z}^d} |\hat{f_{\mathbf k}} - \hat{f}_{\mathbf k}^{\Lambda}|^2 
	&= \sum_{\mathbf k\in\mathbb{Z}^d\setminus I_{N}^{d}} |\hat{f_{\mathbf k}}|^2 + \sum_{\mathbf k\in I_{N}^{d}} |\hat{f_{\mathbf k}} - \hat{f}_{\mathbf k}^{\Lambda}|^2 \nonumber\\
	&= \|f\|_{L_{2}(\mathbb{T}^d)}^2 + \sum_{\mathbf k\in I_{N}^{d}} \left( |\hat{f_{\mathbf k}} - \hat{f}_{\mathbf k}^{\Lambda}|^2 - |\hat{f_{\mathbf k}}|^2 \right).
\end{align*}
We could evaluate the $L_2$-approximation error if we use Algorithm~\ref{alg:LFFT_recon} to reconstruct the approximated Fourier coefficients $\hat{f}_{\mathbf k}^{\Lambda}$ and if it is possible to calculate the Fourier coefficients $\hat{f_{\mathbf k}}$ for all ${\mathbf k\in I_{N}^{d}}$.

\section{Examples}
Throughout this section we always consider a constant weight function $\omega \equiv 1$.
We consider specific test functions ${h\in \mathcal{C}_{\mathrm{mix}}^{m}\left(\left[-\frac{1}{2},\frac{1}{2}\right]^d\right)}$ in combination with 
the logarithmic transformation \eqref{eq:logarithmic_trafo_comb} and
the sine transformation \eqref{eq:sine_trafo} in dimensions $d\in\{1,2,5\}$.
For both transformations we check if the proposed smoothness conditions~\eqref{eq:Cm_composition_criteria_mult} in Theorem~\ref{thm:Cm_composition_criteria_mult} are fulfilled.
These smoothness conditions lead to ranges of the multivariate parameter $\bm\eta\in\mathbb{R}_{+}^{d}$ appearing in the logarithmic transformation \eqref{eq:logarithmic_trafo_comb} for which the transformed functions $f$ of the form $\eqref{eq:f_is_transformed_h_mult}$ have a guaranteed Sobolev smoothness degree $m\in\mathbb{N}$ and we have $f\in \mathcal H^{m}(\mathbb{T}^d)$.
For such functions we have proven $L_{\infty}$-approximation error bounds in Theorem~\ref{thm:L_infty_approx_error_multivar}.
We compare the corresponding relative discrete approximation error $\varepsilon_{\infty}$ given in \eqref{eq:Discrete_ellinfty_error} for both 
the logarithmic transformation \eqref{eq:logarithmic_trafo_comb} with various values of $\bm\eta\in\mathbb{R}_{+}^{d}$ 
and the sine transformation \eqref{eq:sine_trafo}.

In this section we repeatedly specify the parameter vectors $\bm\eta = (\eta, \ldots, \eta)^{\top}$ 
that have the same number in each entry, for which we recall the short notation of just using a single bold number that appeared earlier in the definition \eqref{def:rank_one_lattice} of rank-$1$ lattices $\Lambda(\mathbf z,M)$ in form of $\mathbf 1 = (1,\ldots,1)^{\top}$.

\subsection{Univariate approximation}
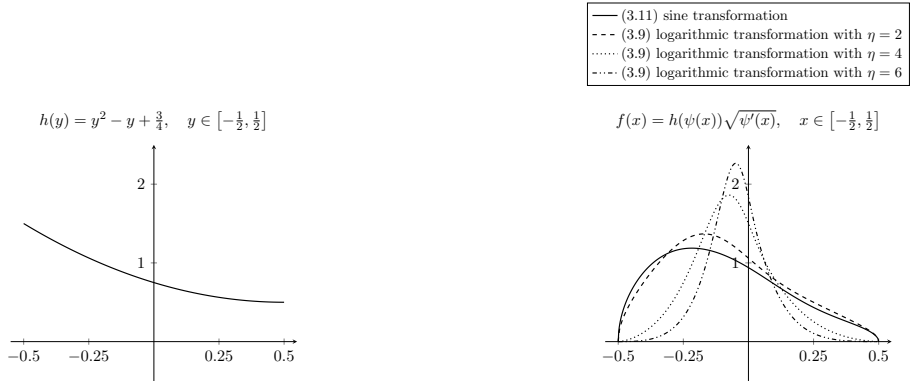
\begin{figure}[]
	\begin{minipage}[b]{.5\linewidth}
		\centering
		\begin{tikzpicture}[scale=0.55]
		\pgfmathsetmacro\ETA{2}
		\begin{axis}[
			samples=100,  
			xmin=-0.55, xmax=0.55, 
			ymin=-0.5, ymax=2.5,
			xtick ={-0.5,-0.25,0,0.25,0.5},
			title = {$h(y) = y^2-y+\frac{3}{4}, \quad y\in\textstyle\left[-\frac{1}{2},\frac{1}{2}\right]$},
			axis x line=center, axis y line=center,
			every axis plot/.append style={thick},
			legend style={at={(0.5,1.25)}, anchor=south,legend columns=4,legend cell align=left, font=\small}
		]
		\addplot[black, domain=-0.5:0.5] { x^2 - x +3/4 };
		\end{axis}
		\end{tikzpicture}
	\end{minipage}
	\begin{minipage}[b]{.5\linewidth}
		\centering
		\begin{tikzpicture}[
			scale=0.55,
			declare function = {logcomb(\ETA) = (1/2)*( (1+2*x)^(\ETA)-(1-2*x)^(\ETA) )/( (1+2*x)^(\ETA)+(1-2*x)^(\ETA) );}, 
			declare function = {logcombstrich(\ETA) = sqrt( 4*\ETA *( (1-4*x^2)^(\ETA-1) )/( ((1+2*x)^(\ETA)+(1-2*x)^(\ETA))^2 ) );}	
		]
		\pgfmathsetmacro\ETAa{2}
		\pgfmathsetmacro\ETAb{4}
		\pgfmathsetmacro\ETAc{6}
		\begin{axis}[
			samples=500,  
			xmin=-0.55, xmax=0.55, ymin=-0.5, ymax=2.5,
			xtick ={-0.5,-0.25,0,0.25,0.5},
			title = {$f(x) = h(\psi(x))\sqrt{\psi'(x)}, \quad x\in\textstyle\left[-\frac{1}{2},\frac{1}{2}\right]$},
			axis x line=center, axis y line=center,
			every axis plot/.append style={thick},
			legend style={at={(0.5,1.25)}, anchor=south,legend columns=1,legend cell align=left, font=\small}
		]
		\addplot[black,domain=-0.5:0.5] { (( (1/4)*sin(pi*deg(x))^2 )^2 - ( (1/2)*sin(pi*deg(x)) ) + 3/4)*sqrt(pi/2 * cos(pi*deg(x))) };
		\addlegendentry{\eqref{eq:sine_trafo} sine transformation};
		\addplot[domain=-1/2:1/2, samples = 200, samples y=0, dashed, thick, smooth] 
		 	( x, {((logcomb(\ETAa))^2 - (logcomb(\ETAa)) + 3/4)*(logcombstrich(\ETAa))} );
		\addlegendentry{\eqref{eq:logarithmic_trafo_comb} logarithmic transformation with $\eta=2$};
		\addplot[domain=-1/2:1/2, samples = 200, samples y=0, dotted, thick, smooth] 
		 	( x, {((logcomb(\ETAb))^2 - (logcomb(\ETAb)) + 3/4)*(logcombstrich(\ETAb))} );
		\addlegendentry{\eqref{eq:logarithmic_trafo_comb} logarithmic transformation with $\eta=4$};
		\addplot[domain=-1/2:1/2, samples = 200, samples y=0, dashdotdotted, thick, smooth] 
			( x, {((logcomb(\ETAc))^2 - (logcomb(\ETAc)) + 3/4)*(logcombstrich(\ETAc))} );
		\addlegendentry{\eqref{eq:logarithmic_trafo_comb} logarithmic transformation with $\eta=6$};
		\end{axis}
		\end{tikzpicture}
	\end{minipage}
	\caption{Plot of the univariate test function $h$ given in \eqref{eq:exemplary_h} on the left. A comparison of the transformed functions $f$ after applying the sine transformation as in \eqref{eq:trafo_function_sine} and after applying the logarithmic transformation \eqref{eq:logarithmic_trafo_comb2} as in \eqref{eq:trafo_function_comblog} with $\eta \in \{ 2, 4, 6 \}$ on the right.}
	\label{fig:univariate_test_function_plot}
\end{figure}
We consider the univariate test function
\begin{align}
	\label{eq:exemplary_h}
	h(y) = y^2-y+\frac{3}{4}, \quad y\in\textstyle\left[-\frac{1}{2},\frac{1}{2}\right],
\end{align}
shown on the left of Figure~\ref{fig:univariate_test_function_plot}.
We choose the weight function $\omega(y,\mu) \equiv 1$ for all ${\mu\in\mathbb{R}_{+}}$ so that the test function $h$ is in $L_{2}\left([-\frac{1}{2},\frac{1}{2}],\omega(\circ,\mu)\right) = L_{2}\left([-\frac{1}{2},\frac{1}{2}]\right)$.
Furthermore, we consider the logarithmic transformation $\psi(x,\eta)$ given in \eqref{eq:logarithmic_trafo_comb} with $x\in[-\frac{1}{2},\frac{1}{2}]$ and ${\eta \in\mathbb{R}_{+}}$.
Combining the test function $h$ in \eqref{eq:exemplary_h} with the logarithmic transformation of the form
\begin{align}
	\label{eq:logarithmic_trafo_comb2}
	\psi(x,\eta) &= \frac{1}{2}\,\frac{(1+2x)^\eta - (1-2x)^\eta}{(1+2x)^\eta + (1-2x)^\eta}
\end{align}
leads to transformed functions $f(\circ,\eta) := f(\circ,\eta,1)$ in the sense of \eqref{eq:f_is_transformed_h_mult} that are of the form
\begin{align}
	\label{eq:trafo_function_comblog}
	f( x,\eta)
	= h(\psi(x,\eta)) \sqrt{ \psi'(x,\eta) } 
	= \left( \psi(x,\eta)^2 - \psi(x,\eta) + \frac{3}{4}\right) \sqrt{ \psi'(x,\eta) }.
\end{align}
In Figure~\ref{fig:univariate_test_function_plot} we have a side-by-side comparison of the graphs of these transformed functions $f( x, \eta)$ with the parameter ${\eta \in\{ 1,  2,  4, 6\}}$.

We determine the values $\eta\in\mathbb{R}_{+}$ for which $f(\circ, \eta)$ as in \eqref{eq:trafo_function_comblog} is an element of $\mathcal H^{m}(\mathbb{T})$ by investigating the smoothness conditions~\eqref{eq:Cm_composition_criteria} in Theorem~\ref{thm:Cm_composition_criteria}.
First of all, for $\eta > 1$ the transformations $\psi(\circ,\eta)$ in \eqref{eq:logarithmic_trafo_comb2} are transformations of the form \eqref{def:Trafo_cube}.
The test function \eqref{eq:exemplary_h} is obviously in $\mathcal{C}^{m}\left(\left[-\frac{1}{2},\frac{1}{2}\right]\right)$ for any $m\in\mathbb{N}_{0}$.
We proceed to check conditions~\eqref{eq:Cm_composition_criteria} for a given $m\in\mathbb{N}_{0}$.
With the constant weight function $\omega\equiv 1$ these conditions simplify to the task of determining the values $\eta \in\mathbb{R}_{+}$ for which we have
\begin{align*}
	\left\|\left(\sqrt{\psi'(\circ, \eta)}\right)^{(j)}(\circ) \right\|_{L_{\infty}\left(\left[-\frac{1}{2},\frac{1}{2}\right]\right)}	< \infty,
\end{align*}
as well as
\begin{align*}
	\left(\sqrt{\psi'(x,\eta)}\right)^{(j)}(x) \to 0 \quad \text{for} \quad |x| \to \frac{1}{2}
\end{align*}
for all $j=0,\ldots,m$.
We obtain the following:
\begin{itemize}
\item 
	For $m=0$ we already mentioned in \eqref{eq:logarithmic_trafo_comb} that the functions $\psi'(\circ,\eta)$ are finite for $\eta\geq 1$ and converge to $0$ at the boundary points $\pm\frac{1}{2}$ for $\eta>1$. 
\item 
	For natural degrees of smoothness $m$ the $m$-th derivative of $\sqrt{\psi'(\circ, \eta)}$ is in $\mathcal{C}_{0}\left(\left[-\frac{1}{2},\frac{1}{2}\right]\right)$ if $\eta > 2m+1$.
\item 
	For values $2m+1 < \eta < 2m+3$ the $(m+1)$-th and all higher derivatives of $\sqrt{\psi'(\circ, \eta)}$ are unbounded
	and in case of $\eta = 2m+3$ they are bounded but not $\mathcal{C}_{0}\left(\left[-\frac{1}{2},\frac{1}{2}\right]\right)$.
\end{itemize}
In total, the transformed function $f(\circ,\eta)$ in \eqref{eq:trafo_function_comblog} is in $\mathcal H^{m}(\mathbb{T})$ for all $\eta > 2m+1$ according to the conditions in Theorem~\ref{thm:Cm_composition_criteria}.

Switching to the sine transformation \eqref{eq:sine_trafo} leads to a transformed function $f$ as given in \eqref{eq:f_is_transformed_h_mult} of the form
\begin{align}
	\label{eq:trafo_function_sine}
	f( x)
	= h(\psi(x)) \sqrt{ \psi'(x) } 
	= \left( \frac{1}{4}\,\sin^2(\pi x) - \frac{1}{2}\,\sin(\pi x) + \frac{3}{4} \right) \, \sqrt{ \frac{\pi}{2} \cos(\pi x) }
\end{align}
for $x\in\left[-\frac{1}{2},\frac{1}{2}\right]$.
The transformed function in \eqref{eq:trafo_function_sine} is only guaranteed to be in $L_{2}(\mathbb{T})$ according to Theorem~\ref{thm:Cm_composition_criteria}, because all derivatives of $\sqrt{\psi'}$ are unbounded.

Finally, we compare the relative discrete approximation errors $\varepsilon_{\infty}$ given in \eqref{eq:Discrete_ellinfty_error} of the sine transformed function in \eqref{eq:trafo_function_sine} and of the logarithmically transformed functions in \eqref{eq:trafo_function_comblog} with $\eta = 2$ and $\eta = 4$.
For this matter we consider the univariate hyperbolic cross $I_{N}^{1} = \{-N, \ldots, N\}$ and let $N \in \{4,5,\ldots, 80\}$.
In Figure~\ref{fig:numeric_test_univar} we showcase the approximation errors of both the sine transformed and the logarithmically transformed functions for $\eta = 2$ that behave similarly, because they are both $L_2(\mathbb{T})$-functions and are not guaranteed to have any upper bound as in \eqref{eq:discrete_upper_bound}.
By increasing the parameter to $\eta = 4$ it smoothed the logarithmically transformed function by one Sobolev smoothness degree, so that $f\in \mathcal H^{1}(\mathbb{T})$, causing the faster decaying upper bound \eqref{eq:discrete_upper_bound} and the faster decay of the relative approximation error $\varepsilon_{\infty}$. 
Another parameter increase to $\eta = 6$ increases the Sobolev smoothness by another degree so that $f\in \mathcal H^{2}(\mathbb{T})$ and for $\eta=8$ we have $f\in \mathcal H^{3}(\mathbb{T})$, resulting in even faster decays of the relative approximation errors $\varepsilon_{\infty}$ for large enough $N\in\mathbb{N}$.

\begin{remark}
	For the error function transformation \eqref{eq:error_function_trafo_comb} we obtained a very similar result regarding the parameter bound and the resulting approximation error decay.
	In fact, for non-negative integer degrees of smoothness $m \in \mathbb{N}_{0}$ the $m$-th derivative of $\sqrt{\psi'(\circ, \eta)}$ is in $\mathcal{C}_{0}\left(\left[-\frac{1}{2},\frac{1}{2}\right]\right)$ if $\eta > 2m+1$, too.
	Due to the now exponential density functions $\varrho(\circ,\eta)$ we obtain an overall faster decay of the discretized approximation error. 
	However, as with the logarithmic transformation \eqref{eq:logarithmic_trafo_comb} the rate of decay is at first only as fast as the decay obtained with the sine transformation \eqref{eq:sine_trafo} and increases rapidly once we increase the parameter values to $\eta> 3$.
\end{remark}

\begin{figure}[t]
	\centering
		\begin{tikzpicture}[baseline,scale=0.65]
		\begin{axis}[
		ymode = log,
		enlargelimits=false,
		xmin=4, xmax=80, ymin=1e-11, ymax=1e-0,
		ytick={1e-1,1e-2,1e-3,1e-4,1e-5,1e-6,1e-7,1e-8,1e-9,1e-10,1e-11},grid=both, 
		xlabel={$N$}, 
		ylabel={$\varepsilon_{\infty}$},
		legend style={at={(0.5,1.05)}, anchor=south,legend columns=1,legend cell align=left, font=\small,  
		},
		xminorticks=false,
		yminorticks=false
		]
		\addplot[mark options={solid},red!25!yellow,mark=triangle,mark size=1.5] coordinates {
			(4, 4.2248e-02)  (5, 3.1548e-02)  (6, 2.5202e-02)  (7, 2.0441e-02)  (8, 1.7202e-02)  (9, 1.4610e-02)  (10, 1.2700e-02)  (11, 1.1108e-02)  (12, 9.8707e-03)  (13, 8.8117e-03)  (14, 7.9569e-03)  (15, 7.2106e-03)  (16, 6.5911e-03)  (17, 6.0419e-03)  (18, 5.5760e-03)  (19, 5.1581e-03)  (20, 4.7973e-03)  (21, 4.4706e-03)  (22, 4.1845e-03)  (23, 3.9233e-03)  (24, 3.6919e-03)  (25, 3.4793e-03)  (26, 3.2890e-03)  (27, 3.1131e-03)  (28, 2.9544e-03)  (29, 2.8070e-03)  (30, 2.6729e-03)  (31, 2.5479e-03)  (32, 2.4335e-03)  (33, 2.3264e-03)  (34, 2.2278e-03)  (35, 2.1352e-03)  (36, 2.0495e-03)  (37, 1.9688e-03)  (38, 1.8938e-03)  (39, 1.8230e-03)  (40, 1.7569e-03)  (41, 1.6943e-03)  (42, 1.6357e-03)  (43, 1.5801e-03)  (44, 1.5279e-03)  (45, 1.4782e-03)  (46, 1.4314e-03)  (47, 1.3868e-03)  (48, 1.3446e-03)  (49, 1.3044e-03)  (50, 1.2663e-03)  (51, 1.2299e-03)  (52, 1.1953e-03)  (53, 1.1622e-03)  (54, 1.1306e-03)  (55, 1.1005e-03)  (56, 1.0717e-03)  (57, 1.0440e-03)  (58, 1.0176e-03)  (59, 9.9227e-04)  (60, 9.6798e-04)  (61, 9.4464e-04)  (62, 9.2224e-04)  (63, 9.0070e-04)  (64, 8.8000e-04)  (65, 8.6007e-04)  (66, 8.4089e-04)  (67, 8.2241e-04)  (68, 8.0459e-04)  (69, 7.8741e-04)  (70, 7.7083e-04)  (71, 7.5483e-04)  (72, 7.3937e-04)  (73, 7.2444e-04)  (74, 7.0999e-04)  (75, 6.9603e-04)  (76, 6.8251e-04)  (77, 6.6943e-04)  (78, 6.5675e-04)  (79, 6.4448e-04)  (80, 6.3258e-04)};
		\addlegendentry{\eqref{eq:sine_trafo} sine transformation};
		\addplot[mark options={solid}, blue, mark=o, mark size=1.0] coordinates {
			(4, 2.6252e-02)  (5, 1.9658e-02)  (6, 1.5625e-02)  (7, 1.2679e-02)  (8, 1.0561e-02)  (9, 9.0293e-03)  (10, 7.7887e-03)  (11, 6.8154e-03)  (12, 6.0506e-03)  (13, 5.3964e-03)  (14, 4.8578e-03)  (15, 4.4133e-03)  (16, 4.0199e-03)  (17, 3.6859e-03)  (18, 3.4011e-03)  (19, 3.1433e-03)  (20, 2.9196e-03)  (21, 2.7244e-03)  (22, 2.5447e-03)  (23, 2.3863e-03)  (24, 2.2455e-03)  (25, 2.1146e-03)  (26, 1.9977e-03)  (27, 1.8921e-03)  (28, 1.7934e-03)  (29, 1.7042e-03)  (30, 1.6226e-03)  (31, 1.5460e-03)  (32, 1.4761e-03)  (33, 1.4115e-03)  (34, 1.3507e-03)  (35, 1.2947e-03)  (36, 1.2426e-03)  (37, 1.1933e-03)  (38, 1.1477e-03)  (39, 1.1049e-03)  (40, 1.0643e-03)  (41, 1.0265e-03)  (42, 9.9082e-04)  (43, 9.5695e-04)  (44, 9.2524e-04)  (45, 8.9514e-04)  (46, 8.6651e-04)  (47, 8.3960e-04)  (48, 8.1392e-04)  (49, 7.8947e-04)  (50, 7.6640e-04)  (51, 7.4429e-04)  (52, 7.2321e-04)  (53, 7.0325e-04)  (54, 6.8405e-04)  (55, 6.6573e-04)  (56, 6.4832e-04)  (57, 6.3153e-04)  (58, 6.1547e-04)  (59, 6.0018e-04)  (60, 5.8539e-04)  (61, 5.7124e-04)  (62, 5.5772e-04)  (63, 5.4462e-04)  (64, 5.3206e-04)  (65, 5.2004e-04)  (66, 5.0836e-04)  (67, 4.9716e-04)  (68, 4.8643e-04)  (69, 4.7596e-04)  (70, 4.6592e-04)  (71, 4.5627e-04)  (72, 4.4686e-04)  (73, 4.3782e-04)  (74, 4.2911e-04)  (75, 4.2061e-04)  (76, 4.1243e-04)  (77, 4.0454e-04)  (78, 3.9683e-04)  (79, 3.8941e-04)  (80, 3.8222e-04)};
		\addlegendentry{\eqref{eq:logarithmic_trafo_comb} logarithmic transformation with $\eta=2$};
		\addplot[mark options={solid},red!75!yellow,mark=x,mark size=1.0] coordinates {
			(4, 3.6684e-03)  (5, 2.4061e-03)  (6, 5.0997e-04)  (7, 5.2479e-04)  (8, 5.6531e-05)  (9, 1.7015e-04)  (10, 9.0475e-05)  (11, 8.7376e-05)  (12, 6.7060e-05)  (13, 5.7632e-05)  (14, 4.8363e-05)  (15, 4.1645e-05)  (16, 3.5934e-05)  (17, 3.1138e-05)  (18, 2.7219e-05)  (19, 2.3993e-05)  (20, 2.1300e-05)  (21, 1.9028e-05)  (22, 1.7095e-05)  (23, 1.5390e-05)  (24, 1.3895e-05)  (25, 1.2605e-05)  (26, 1.1484e-05)  (27, 1.0505e-05)  (28, 9.6429e-06)  (29, 8.8809e-06)  (30, 8.1819e-06)  (31, 7.5580e-06)  (32, 7.0019e-06)  (33, 6.5041e-06)  (34, 6.0567e-06)  (35, 5.6531e-06)  (36, 5.2860e-06)  (37, 4.9440e-06)  (38, 4.6338e-06)  (39, 4.3514e-06)  (40, 4.0937e-06)  (41, 3.8578e-06)  (42, 3.6414e-06)  (43, 3.4399e-06)  (44, 3.2517e-06)  (45, 3.0784e-06)  (46, 2.9183e-06)  (47, 2.7702e-06)  (48, 2.6329e-06)  (49, 2.5053e-06)  (50, 2.3846e-06)  (51, 2.2716e-06)  (52, 2.1663e-06)  (53, 2.0681e-06)  (54, 1.9762e-06)  (55, 1.8903e-06)  (56, 1.8097e-06)  (57, 1.7325e-06)  (58, 1.6601e-06)  (59, 1.5920e-06)  (60, 1.5280e-06)  (61, 1.4677e-06)  (62, 1.4109e-06)  (63, 1.3568e-06)  (64, 1.3052e-06)  (65, 1.2564e-06)  (66, 1.2102e-06)  (67, 1.1665e-06)  (68, 1.1251e-06)  (69, 1.0858e-06)  (70, 1.0481e-06)  (71, 1.0121e-06)  (72, 9.7785e-07)  (73, 9.4529e-07)  (74, 9.1431e-07)  (75, 8.8480e-07)  (76, 8.5667e-07)  (77, 8.2945e-07)  (78, 8.0346e-07)  (79, 7.7866e-07)  (80, 7.5496e-07)};
		\addlegendentry{\eqref{eq:logarithmic_trafo_comb} logarithmic transformation with $\eta=4$};
		\addplot[mark options={solid},red!50!blue,mark=square,mark size=1.0] coordinates {
			(4, 1.2127e-03)  (5, 3.8788e-03)  (6, 3.3102e-03)  (7, 2.1827e-03)  (8, 1.3399e-03)  (9, 7.5543e-04)  (10, 4.2373e-04)  (11, 2.2129e-04)  (12, 1.1997e-04)  (13, 5.8627e-05)  (14, 3.2129e-05)  (15, 1.4045e-05)  (16, 8.3482e-06)  (17, 2.8596e-06)  (18, 2.3557e-06)  (19, 3.2622e-07)  (20, 8.2883e-07)  (21, 1.7313e-07)  (22, 4.0093e-07)  (23, 2.2489e-07)  (24, 2.5189e-07)  (25, 1.9239e-07)  (26, 1.8023e-07)  (27, 1.5217e-07)  (28, 1.3686e-07)  (29, 1.1997e-07)  (30, 1.0735e-07)  (31, 9.5656e-08)  (32, 8.5936e-08)  (33, 7.7313e-08)  (34, 6.9867e-08)  (35, 6.3301e-08)  (36, 5.7351e-08)  (37, 5.2114e-08)  (38, 4.7491e-08)  (39, 4.3395e-08)  (40, 3.9759e-08)  (41, 3.6506e-08)  (42, 3.3599e-08)  (43, 3.0994e-08)  (44, 2.8640e-08)  (45, 2.6520e-08)  (46, 2.4570e-08)  (47, 2.2788e-08)  (48, 2.1171e-08)  (49, 1.9706e-08)  (50, 1.8375e-08)  (51, 1.7152e-08)  (52, 1.6039e-08)  (53, 1.5018e-08)  (54, 1.4088e-08)  (55, 1.3224e-08)  (56, 1.2421e-08)  (57, 1.1676e-08)  (58, 1.0985e-08)  (59, 1.0364e-08)  (60, 9.7618e-09)  (61, 9.2172e-09)  (62, 8.7220e-09)  (63, 8.2531e-09)  (64, 7.8104e-09)  (65, 7.3995e-09)  (66, 7.0196e-09)  (67, 6.6916e-09)  (68, 6.3443e-09)  (69, 6.0348e-09)  (70, 5.7151e-09)  (71, 5.4710e-09)  (72, 5.1938e-09)  (73, 4.9553e-09)  (74, 4.7131e-09)  (75, 4.5364e-09)  (76, 4.3165e-09)  (77, 4.1092e-09)  (78, 3.9255e-09)  (79, 3.7542e-09)  (80, 3.6154e-09)};
		\addlegendentry{\eqref{eq:logarithmic_trafo_comb} logarithmic transformation with $\eta=6$};
		\addplot[mark options={solid},yellow!50!blue,mark=diamond,mark size=1.0] coordinates {
			(4, 1.7374e-02)  (5, 3.7763e-03)  (6, 1.4877e-03)  (7, 2.7295e-03)  (8, 2.5537e-03)  (9, 2.0041e-03)  (10, 1.4491e-03)  (11, 9.9757e-04)  (12, 6.6473e-04)  (13, 4.3220e-04)  (14, 2.7593e-04)  (15, 1.7343e-04)  (16, 1.0773e-04)  (17, 6.6159e-05)  (18, 4.0304e-05)  (19, 2.4327e-05)  (20, 1.4546e-05)  (21, 8.5588e-06)  (22, 5.0349e-06)  (23, 2.9378e-06)  (24, 1.7205e-06)  (25, 9.9460e-07)  (26, 5.8168e-07)  (27, 3.3220e-07)  (28, 1.9526e-07)  (29, 1.0936e-07)  (30, 6.5496e-08)  (31, 3.5207e-08)  (32, 2.2191e-08)  (33, 1.0946e-08)  (34, 7.7261e-09)  (35, 3.0108e-09)  (36, 2.9764e-09)  (37, 5.7568e-10)  (38, 1.2749e-09)  (39, 3.9020e-10)  (40, 9.2766e-10)  (41, 2.9772e-10)  (42, 4.4828e-10)  (43, 3.0344e-10)  (44, 6.3608e-10)  (45, 2.6997e-10)  (46, 1.1755e-09)  (47, 7.8001e-10)  (48, 4.4988e-10)  (49, 1.9397e-10)  (50, 1.7923e-10)  (51, 3.0917e-10)  (52, 4.9126e-10)  (53, 1.3824e-10)  (54, 1.2788e-10)  (55, 1.1848e-10)  (56, 1.0881e-10)  (57, 1.0609e-10)  (58, 9.9263e-11)  (59, 8.7010e-11)  (60, 8.3201e-11)  (61, 8.5883e-11)  (62, 7.5946e-11)  (63, 6.5438e-11)  (64, 7.4728e-11)  (65, 6.4721e-11)  (66, 5.3574e-11)  (67, 4.9805e-11)  (68, 4.6749e-11)  (69, 4.4747e-11)  (70, 4.6206e-11)  (71, 5.8899e-11)  (72, 3.7115e-11)  (73, 4.2753e-11)  (74, 3.3054e-11)  (75, 3.0913e-11)  (76, 3.4273e-11)  (77, 2.9090e-11)  (78, 3.0605e-11)  (79, 5.7212e-11)  (80, 2.3203e-11)};
		\addlegendentry{\eqref{eq:logarithmic_trafo_comb} logarithmic transformation with $\eta=8$};
		\end{axis}
		\end{tikzpicture}
	\caption{Comparison of discrete $\ell_\infty$-approximation error $\varepsilon_{\infty}$ as given in \eqref{eq:Discrete_ellinfty_error} of the one-dimensional test function \eqref{eq:exemplary_h} in combination with the sine transformation \eqref{eq:sine_trafo} and the logarithmic transformation \eqref{eq:logarithmic_trafo_comb} with $\eta \in \{ 2, 4, 6, 8 \}$.
	}
	\label{fig:numeric_test_univar}
\end{figure}	

\subsection{High-dimensional approximation}
Once again, we stress the fact that we have the fast Algorithm~\ref{alg:LFFT_eval} and Algorithm~\ref{alg:LFFT_recon} that are based on a single one-dimensional inverse FFT and an one-dimensional FFT, respectively.
We consider the test function
\begin{align}\label{eq:exemplary_h_mult}
	h(\mathbf y) = h(y_1,\ldots,y_d) = \sum_{j=1}^{d} y_j, \quad \mathbf y\in\textstyle\left[-\frac{1}{2},\frac{1}{2}\right]^d.
\end{align}
We choose the constant weight function $\omega(\circ,\bm \mu) \equiv 1$ for all ${\bm\mu\in\mathbb{R}_{+}^d}$
and the logarithmic transformation 
${\psi(\mathbf x,\bm\eta) = ((\psi_j(x_j,\eta_j))_{j=1}^{d})^{\top}}$  with $\mathbf x\in[-\frac{1}{2},\frac{1}{2}]^d$, the parameter ${\bm\eta = (\eta_1,\ldots,\eta_d)^{\top} \in\mathbb{R}_{+}^d}$ and its univariate components $\psi_j(x_j,\eta_j)$ in the form of \eqref{eq:logarithmic_trafo_comb2}.
Due to the constant weight function $\omega\equiv 1$ the test function $h$ is simply in $L_{2}\left([-\frac{1}{2},\frac{1}{2}]^d,\omega(\circ,\bm\mu)\right) = L_{2}\left([-\frac{1}{2},\frac{1}{2}]^d\right)$.
The test function $h$ in \eqref{eq:exemplary_h_mult} combined with the logarithmic transformation leads to the transformed functions $f(\circ,\bm\eta,1) =: f(\circ,\bm\eta)$ of the form \eqref{eq:f_is_transformed_h_mult}, reading as
\begin{align}
	\label{eq:trafo_function_comblog_mult}
	&f(\mathbf x,\bm\eta)
	= h(\psi_1(x_1,\eta_1),\ldots,\psi_d(x_d,\eta_d)) \, \prod_{j=1}^{d}\sqrt{ \psi_j'(x_j,\eta_j) } \nonumber\\
	&= \frac{1}{2} \sum_{j=1}^{d} \frac{(1+2x_j)^{\eta_j} - (1-2x_j)^{\eta_j}}{(1+2x_j)^{\eta_j} + (1-2x_j)^{\eta_j}} 
	\, \prod_{j=1}^{d}\sqrt{ 4\eta_j\,\frac{(1-4x_j^2)^{\eta_j-1}}{\left((1+2x_j)^{\eta_j}+(1-2x_j)^{\eta_j}\right)^2}}.
\end{align}
In Figure~\ref{fig:Plots_f_with_algTrafo2} we have a side-by-side comparison of the graphs of these transformed functions $f(\mathbf x, \bm\eta)$ for $d=2$ with the parameters ${\bm\eta \in\{\mathbf 1, \mathbf 2, \mathbf 4, \mathbf 6\}}$.

\begin{figure}[t]
\begin{minipage}{.245\linewidth}
\centering
	\includegraphics[width=\textwidth]{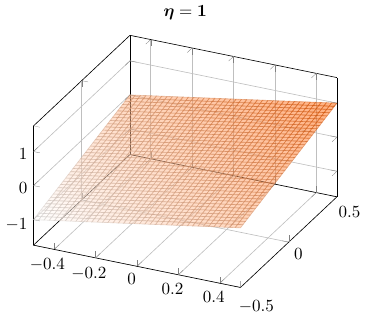}
\end{minipage}
\begin{minipage}{.245\linewidth}
\centering
	\includegraphics[width=\textwidth]{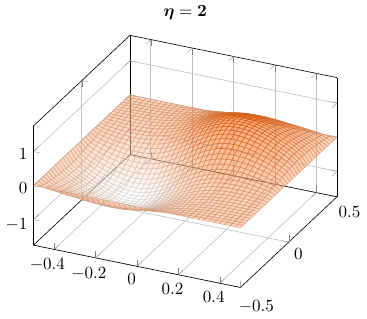}
\end{minipage}
\begin{minipage}{.245\linewidth}
\centering
	\includegraphics[width=\textwidth]{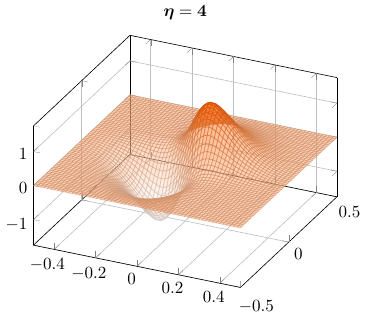}
\end{minipage}
\begin{minipage}{.245\linewidth}
\centering
	\includegraphics[width=\textwidth]{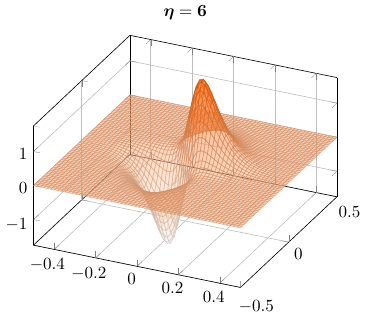}
\end{minipage}
\caption{Plots of the two-dimensional transformed function $f(\circ,\bm\eta)$ in \eqref{eq:trafo_function_comblog_mult} for ${\bm\eta \in\{\mathbf 1, \mathbf 2, \mathbf 4, \mathbf{6}\}}$ with the logarithmic transformation $\psi(\circ,\bm\eta)$ as in \eqref{eq:logarithmic_trafo_comb2}. }
\label{fig:Plots_f_with_algTrafo2}
\end{figure}

Next, we determine the values $\bm\eta\in\mathbb{R}_{+}^{d}$ for which $f(\circ, \bm\eta)$ as in \eqref{eq:trafo_function_comblog_mult} is element of $\mathcal H^{m}(\mathbb{T}^d)$ by investigating conditions~\eqref{eq:Cm_composition_criteria_mult} in Theorem~\ref{thm:Cm_composition_criteria_mult} for the derivatives of $\psi'$.
First of all, we observe that for $\eta_1,\ldots,\eta_d > 1$ the components $\psi_1,\ldots,\psi_d$ of the transformations $\psi(\circ,\bm\eta)$ in \eqref{eq:logarithmic_trafo_comb2} are transformations as defined in \eqref{def:Trafo_cube_mult}. As the test function \eqref{eq:exemplary_h_mult} is obviously in $\mathcal{C}_{\mathrm{mix}}^{m}\left(\left[-\frac{1}{2},\frac{1}{2}\right]^d\right)$ for any $m\in\mathbb{N}_{0}$, we proceed to check conditions~\eqref{eq:Cm_composition_criteria_mult} for a given $m\in\mathbb{N}_{0}$.
For the constant weight function these conditions simplify to the task of determining the values $\bm\eta = (\eta_1,\ldots,\eta_d)^{\top}\in\mathbb{R}_{+}^{d}$ for which we have
\begin{align*}
	\left\| \left(\sqrt{\psi'_{\ell}(\circ, \eta_{\ell})}\right)^{(j_{\ell})}(\circ) \right\|_{L_{\infty}\left(\left[-\frac{1}{2},\frac{1}{2}\right]\right)}	< \infty,
\end{align*}
as well as
\begin{align*}
	\left(\sqrt{\psi_{\ell}'}\right)^{(j_{\ell})}(x_{\ell}) \to 0 \quad \text{for} \quad |x_\ell|\to \frac{1}{2}
\end{align*}
for all $\ell=1,\ldots,d$ and $j_{\ell}=0,\ldots,m$.
For each dimension $\ell=1,\ldots,d$ we observe the following:
\begin{itemize}
\item 
	For $m=0$ we already mentioned in \eqref{eq:logarithmic_trafo_comb} that the functions $\psi_{\ell}'(\circ,\eta_{\ell})$ are finite for $\eta_{\ell}\geq 1$ but converge to $0$ at the boundary points $\pm\frac{1}{2}$ only for $\eta_{\ell}>1$. 
\item 
	For natural degrees of smoothness $m\geq 1$ the $m$-th derivative of $\sqrt{\psi_{\ell}'(\circ, \eta_{\ell})}$ is in $\mathcal{C}_{0}\left(\left[-\frac{1}{2},\frac{1}{2}\right]\right)$ if $\eta_{\ell} > 2m+1$.
\item 
	For values $2m+1 < \eta_{\ell} < 2m+3$ the $(m+1)$-th and all higher derivatives of $\sqrt{\psi_{\ell}'(\circ, \eta_{\ell})}$ are unbounded
	and in case of $\eta_{\ell} = 2m+3$ they are bounded but not $\mathcal{C}_{0}\left(\left[-\frac{1}{2},\frac{1}{2}\right]\right)$.
\end{itemize}
In total, according to the conditions in Theorem~\ref{thm:Cm_composition_criteria} the transformed function $f$ in \eqref{eq:trafo_function_comblog_mult} is in $\mathcal H^{m}(\mathbb{T}^d)$ if $\eta_{\ell} > 2m+1$ for all $\ell = 1,\ldots,d$.

Switching to the sine transformation \eqref{eq:sine_trafo} leads to a transformed function $f$ as given in \eqref{eq:f_is_transformed_h_mult} of the form
\begin{align*}
	f(\mathbf x)
	= h(\psi(\mathbf x)) \prod_{j=1}^{d} \sqrt{ \psi_j'(x_j) } 
	= \frac{1}{2^d} \sum_{j=1}^{d} \sin(\pi x_j) \, \prod_{j=1}^{d}\sqrt{ \frac{\pi}{2} \cos(\pi x_j) },
\end{align*}
that according to Theorem~\ref{thm:Cm_composition_criteria} is only in $\mathcal H^{0}(\mathbb{T}^d)$ because all derivatives of all $\sqrt{\psi_j'(\circ)}$ are unbounded.

Finally, for dimensions $d=2$ and $d=5$ we compare the relative discrete approximation errors $\varepsilon_{\infty}$ as in \eqref{eq:Discrete_ellinfty_error} of the sine transformed function in \eqref{eq:trafo_function_sine} and of the logarithmically transformed functions in \eqref{eq:trafo_function_comblog_mult} with $\bm\eta = \mathbf 2$, $\bm\eta = \mathbf 4$ and in case of $d=2$ additionally with $\bm\eta = \mathbf 6$.
For this matter we consider hyperbolic crosses $I_{N}^{d}$ as defined in \eqref{def:hyperbolic_cross} for $N \in \{8,9,\ldots, 200\}$ in $d=2$ and for $N \in \{8,9,\ldots,100\}$ in $d=5$.
We again emphasize the major advantage of Algorithms~\ref{alg:LFFT_eval} and \ref{alg:LFFT_recon} in having a complexity of just $\mathcal{O}(M \log M + d|I_{N}^{d}|)$ due to being based on a single univariate inverse FFT and univariate FFT, respectively. 
Thus, their computation time is rather quick considering the fact that we are dealing with up to $|I_{100}^{5}| = 665.145$ frequencies.
In Figure~\ref{fig:numeric_test_multivar} we showcase that the approximation errors of both the sine transformed and the logarithmically transformed functions for $\bm\eta = \mathbf 2$ behave similarly because they are both $L_2(\mathbb{T}^d)$-functions and are not guaranteed to have any upper bound as in \eqref{eq:discrete_upper_bound}.
By increasing the parameter to $\bm\eta = \mathbf 4$ it smoothed the logarithmically transformed function by one Sobolev smoothness degree, so that $f\in \mathcal H^{1}(\mathbb{T}^d)$, causing a faster decaying upper bound \eqref{eq:discrete_upper_bound} and thus a faster decay of the relative approximation error $\varepsilon_{\infty}$ as in \eqref{eq:Discrete_ellinfty_error}.
Another parameter increase to $\bm\eta = \mathbf 6$ for $d=2$ increases the Sobolev smoothness by another degree so that $f\in \mathcal H^{2}(\mathbb{T}^2)$ and the relative approximation error $\varepsilon_{\infty}$ decays even faster for large enough $N\in\mathbb{N}$.

\begin{figure}[t]
	\begin{minipage}[b]{.49\linewidth}
		\centering
		\begin{tikzpicture}[baseline,scale=0.65]
		\begin{axis}[
		ymode = log,
		enlargelimits=false,
		xmin=8, xmax=200, ymin=1e-07, ymax=1e-0,
		ytick={1e-1,1e-2,1e-3,1e-4,1e-5,1e-6,1e-7},grid=both, 
		xlabel={$N$}, 
		ylabel={$\varepsilon_{\infty}$},
		legend style={at={(0.5,1.05)}, anchor=south,legend columns=1,legend cell align=left, font=\small,  
		},
		xminorticks=false,
		yminorticks=false
		]
		\addplot[mark options={solid},red!25!yellow,mark=triangle,mark size=1.5] coordinates {
			(8, 1.2316e-01)  (9, 1.1577e-01)  (10, 1.0827e-01)  (11, 1.0137e-01)  (12, 1.0012e-01)  (13, 9.4045e-02)  (14, 8.9729e-02)  (15, 8.8029e-02)  (16, 8.5989e-02)  (17, 8.5058e-02)  (18, 8.0127e-02)  (19, 7.8404e-02)  (20, 7.7070e-02)  (21, 7.5974e-02)  (22, 7.3930e-02)  (23, 7.2652e-02)  (24, 6.9849e-02)  (25, 6.7803e-02)  (26, 6.4670e-02)  (27, 6.6895e-02)  (28, 6.5285e-02)  (29, 6.3289e-02)  (30, 6.2181e-02)  (31, 6.2189e-02)  (32, 6.1499e-02)  (33, 5.9811e-02)  (34, 5.8731e-02)  (35, 5.7691e-02)  (36, 5.7049e-02)  (37, 5.5540e-02)  (38, 5.5862e-02)  (39, 5.5570e-02)  (40, 5.4977e-02)  (41, 5.4095e-02)  (42, 5.2783e-02)  (43, 5.2845e-02)  (44, 5.2369e-02)  (45, 4.8569e-02)  (46, 4.9397e-02)  (47, 4.9930e-02)  (48, 5.0218e-02)  (49, 4.8608e-02)  (50, 4.7448e-02)  (51, 4.8860e-02)  (52, 4.7317e-02)  (53, 4.7624e-02)  (54, 4.7637e-02)  (55, 4.7182e-02)  (56, 4.6276e-02)  (57, 4.5102e-02)  (58, 4.5948e-02)  (59, 4.5475e-02)  (60, 4.4791e-02)  (61, 4.4789e-02)  (62, 4.4337e-02)  (63, 4.4259e-02)  (64, 4.3027e-02)  (65, 4.3409e-02)  (66, 4.2740e-02)  (67, 4.2152e-02)  (68, 4.2288e-02)  (69, 4.1401e-02)  (70, 4.1577e-02)  (71, 4.1162e-02)  (72, 4.0839e-02)  (73, 4.1065e-02)  (74, 4.0264e-02)  (75, 3.9601e-02)  (76, 4.0429e-02)  (77, 3.9672e-02)  (78, 3.9451e-02)  (79, 3.8978e-02)  (80, 3.8716e-02)  (81, 3.7727e-02)  (82, 3.8130e-02)  (83, 3.8342e-02)  (84, 3.7594e-02)  (85, 3.7688e-02)  (86, 3.7305e-02)  (87, 3.6698e-02)  (88, 3.6817e-02)  (89, 3.6590e-02)  (90, 3.6496e-02)  (91, 3.6231e-02)  (92, 3.6104e-02)  (93, 3.5871e-02)  (94, 3.5541e-02)  (95, 3.5506e-02)  (96, 3.5636e-02)  (97, 3.5302e-02)  (98, 3.5090e-02)  (99, 3.3213e-02)  (100, 3.4390e-02) (101, 3.4907e-02)  (102, 3.4687e-02)  (103, 3.3919e-02)  (104, 3.3779e-02)  (105, 3.3927e-02)  (106, 3.3527e-02)  (107, 3.3201e-02)  (108, 3.3165e-02)  (109, 3.3406e-02)  (110, 3.3290e-02)  (111, 3.3147e-02)  (112, 3.2404e-02)  (113, 3.2516e-02)  (114, 3.2387e-02)  (115, 3.2570e-02)  (116, 3.2457e-02)  (117, 3.2410e-02)  (118, 3.2205e-02)  (119, 3.1745e-02)  (120, 3.1469e-02)  (121, 3.1614e-02)  (122, 3.1306e-02)  (123, 3.0931e-02)  (124, 3.1075e-02)  (125, 3.1830e-02)  (126, 3.1069e-02)  (127, 3.1871e-02)  (128, 3.0672e-02)  (129, 3.1854e-02)  (130, 3.1753e-02)  (131, 3.0145e-02)  (132, 2.9251e-02)  (133, 3.0249e-02)  (134, 3.0633e-02)  (135, 2.9873e-02)  (136, 3.1396e-02)  (137, 2.9624e-02)  (138, 2.9438e-02)  (139, 3.1182e-02)  (140, 2.9921e-02)  (141, 2.9280e-02)  (142, 2.7794e-02)  (143, 3.1002e-02)  (144, 3.0540e-02)  (145, 2.9939e-02)  (146, 2.8597e-02)  (147, 2.9844e-02)  (148, 2.8581e-02)  (149, 2.9991e-02)  (150, 2.8693e-02)  (151, 2.8239e-02)  (152, 2.8032e-02)  (153, 2.8264e-02)  (154, 2.7929e-02)  (155, 2.8454e-02)  (156, 2.8651e-02)  (157, 2.7660e-02)  (158, 2.9510e-02)  (159, 2.7507e-02)  (160, 2.8041e-02)  (161, 2.7172e-02)  (162, 2.7129e-02)  (163, 2.6873e-02)  (164, 2.7213e-02)  (165, 2.8268e-02)  (166, 2.9214e-02)  (167, 2.6574e-02)  (168, 2.9380e-02)  (169, 2.9044e-02)  (170, 2.8518e-02)  (171, 2.7812e-02)  (172, 2.6308e-02)  (173, 2.8876e-02)  (174, 2.6916e-02)  (175, 2.6368e-02)  (176, 2.7505e-02)  (177, 2.8293e-02)  (178, 2.6092e-02)  (179, 2.7825e-02)  (180, 2.7030e-02)  (181, 2.6798e-02)  (182, 2.6171e-02)  (183, 2.6777e-02)  (184, 2.6412e-02)  (185, 2.8114e-02)  (186, 2.8725e-02)  (187, 2.7744e-02)  (188, 2.8759e-02)  (189, 2.6910e-02)  (190, 2.8346e-02)  (191, 2.8378e-02)  (192, 2.4801e-02)  (193, 2.6115e-02)  (194, 2.7617e-02)  (195, 2.7968e-02)  (196, 2.7975e-02)  (197, 2.7450e-02)  (198, 2.4645e-02)  (199, 2.6761e-02)  (200, 2.6791e-02) 
		};
		\addlegendentry{\eqref{eq:sine_trafo} sine transformation};
		\addplot[mark options={solid}, blue, mark=o, mark size=1.0] coordinates {
			 (8, 6.9814e-02)  (9, 6.8910e-02)  (10, 6.5724e-02)  (11, 6.3072e-02)  (12, 6.0151e-02)  (13, 5.6907e-02)  (14, 5.5408e-02)  (15, 5.2942e-02)  (16, 5.1820e-02)  (17, 4.9800e-02)  (18, 4.9805e-02)  (19, 4.6968e-02)  (20, 4.6050e-02)  (21, 4.5012e-02)  (22, 4.3847e-02)  (23, 4.3325e-02)  (24, 4.1597e-02)  (25, 4.1564e-02)  (26, 4.0077e-02)  (27, 4.0136e-02)  (28, 3.9158e-02)  (29, 3.8609e-02)  (30, 3.8086e-02)  (31, 3.7421e-02)  (32, 3.6255e-02)  (33, 3.6105e-02)  (34, 3.5669e-02)  (35, 3.5167e-02)  (36, 3.4324e-02)  (37, 3.3871e-02)  (38, 3.2342e-02)  (39, 3.3642e-02)  (40, 3.2554e-02)  (41, 3.2209e-02)  (42, 3.2027e-02)  (43, 3.0664e-02)  (44, 3.1041e-02)  (45, 3.0729e-02)  (46, 3.0364e-02)  (47, 3.0351e-02)  (48, 3.0150e-02)  (49, 2.9717e-02)  (50, 2.9163e-02)  (51, 2.9077e-02)  (52, 2.8304e-02)  (53, 2.8484e-02)  (54, 2.8556e-02)  (55, 2.8284e-02)  (56, 2.7374e-02)  (57, 2.7367e-02)  (58, 2.7235e-02)  (59, 2.6827e-02)  (60, 2.6816e-02)  (61, 2.6596e-02)  (62, 2.6366e-02)  (63, 2.6384e-02)  (64, 2.6155e-02)  (65, 2.5926e-02)  (66, 2.5567e-02)  (67, 2.5477e-02)  (68, 2.5241e-02)  (69, 2.5149e-02)  (70, 2.4869e-02)  (71, 2.4680e-02)  (72, 2.4597e-02)  (73, 2.4260e-02)  (74, 2.4212e-02)  (75, 2.3923e-02)  (76, 2.3472e-02)  (77, 2.3623e-02)  (78, 2.3585e-02)  (79, 2.3131e-02)  (80, 2.3352e-02)  (81, 2.2975e-02)  (82, 2.3077e-02)  (83, 2.2844e-02)  (84, 2.2670e-02)  (85, 2.2485e-02)  (86, 2.2389e-02)  (87, 2.2171e-02)  (88, 2.2193e-02)  (89, 2.2169e-02)  (90, 2.1947e-02)  (91, 2.1821e-02)  (92, 2.1470e-02)  (93, 2.1599e-02)  (94, 2.1279e-02)  (95, 2.1227e-02)  (96, 2.0931e-02)  (97, 2.1201e-02)  (98, 2.0926e-02)  (99, 2.0924e-02)  (100, 2.0694e-02) (101, 2.0826e-02)  (102, 2.0593e-02)  (103, 2.0540e-02)  (104, 2.0341e-02)  (105, 2.0348e-02)  (106, 2.0106e-02)  (107, 2.0114e-02)  (108, 2.0059e-02)  (109, 1.9973e-02)  (110, 1.9928e-02)  (111, 1.9681e-02)  (112, 1.9492e-02)  (113, 1.9608e-02)  (114, 1.9500e-02)  (115, 1.9384e-02)  (116, 1.9340e-02)  (117, 1.9151e-02)  (118, 1.9120e-02)  (119, 1.8916e-02)  (120, 1.9015e-02)  (121, 1.8987e-02)  (122, 1.8425e-02)  (123, 1.8705e-02)  (124, 1.8731e-02)  (125, 1.8614e-02)  (126, 1.8599e-02)  (127, 1.8523e-02)  (128, 1.8461e-02)  (129, 1.8108e-02)  (130, 1.8271e-02)  (131, 1.8230e-02)  (132, 1.8098e-02)  (133, 1.7982e-02)  (134, 1.7975e-02)  (135, 1.7956e-02)  (136, 1.7937e-02)  (137, 1.7859e-02)  (138, 1.7661e-02)  (139, 1.7649e-02)  (140, 1.7581e-02)  (141, 1.7484e-02)  (142, 1.7218e-02)  (143, 1.7467e-02)  (144, 1.7283e-02)  (145, 1.7389e-02)  (146, 1.7204e-02)  (147, 1.7258e-02)  (148, 1.7066e-02)  (149, 1.6857e-02)  (150, 1.7035e-02)  (151, 1.6942e-02)  (152, 1.6747e-02)  (153, 1.6873e-02)  (154, 1.6652e-02)  (155, 1.6436e-02)  (156, 1.6623e-02)  (157, 1.6541e-02)  (158, 1.6605e-02)  (159, 1.6455e-02)  (160, 1.6316e-02)  (161, 1.6352e-02)  (162, 1.6408e-02)  (163, 1.6304e-02)  (164, 1.6266e-02)  (165, 1.6211e-02)  (166, 1.6205e-02)  (167, 1.6081e-02)  (168, 1.6138e-02)  (169, 1.6022e-02)  (170, 1.5979e-02)  (171, 1.5917e-02)  (172, 1.5813e-02)  (173, 1.5897e-02)  (174, 1.5823e-02)  (175, 1.5716e-02)  (176, 1.5701e-02)  (177, 1.5573e-02)  (178, 1.5737e-02)  (179, 1.5529e-02)  (180, 1.5466e-02)  (181, 1.5416e-02)  (182, 1.5492e-02)  (183, 1.5350e-02)  (184, 1.5330e-02)  (185, 1.5306e-02)  (186, 1.5313e-02)  (187, 1.5279e-02)  (188, 1.4993e-02)  (189, 1.5042e-02)  (190, 1.4236e-02)  (191, 1.5022e-02)  (192, 1.4940e-02)  (193, 1.5063e-02)  (194, 1.4891e-02)  (195, 1.4903e-02)  (196, 1.4836e-02)  (197, 1.4825e-02)  (198, 1.4609e-02)  (199, 1.4347e-02)  (200, 1.4755e-02)  
		};
		\addlegendentry{\eqref{eq:logarithmic_trafo_comb} logarithmic transformation with $\bm\eta=\mathbf 2$};
		\addplot[mark options={solid},red!75!yellow,mark=x,mark size=1.0] coordinates {
			(8, 1.3064e-02)  (9, 9.4548e-03)  (10, 6.1442e-03)  (11, 6.3643e-03)  (12, 2.4421e-03)  (13, 2.5671e-03)  (14, 2.2509e-03)  (15, 1.4287e-03)  (16, 8.7295e-04)  (17, 8.8559e-04)  (18, 6.6800e-04)  (19, 6.5847e-04)  (20, 5.7679e-04)  (21, 5.0223e-04)  (22, 4.7692e-04)  (23, 4.6967e-04)  (24, 4.7291e-04)  (25, 4.2624e-04)  (26, 4.1256e-04)  (27, 3.8961e-04)  (28, 3.5063e-04)  (29, 3.4812e-04)  (30, 3.2602e-04)  (31, 3.2115e-04)  (32, 3.0530e-04)  (33, 2.8772e-04)  (34, 2.7370e-04)  (35, 2.6575e-04)  (36, 2.4763e-04)  (37, 2.4523e-04)  (38, 2.3516e-04)  (39, 2.2645e-04)  (40, 2.1003e-04)  (41, 2.1404e-04)  (42, 2.0109e-04)  (43, 1.9750e-04)  (44, 1.8345e-04)  (45, 1.8394e-04)  (46, 1.7877e-04)  (47, 1.7468e-04)  (48, 1.6407e-04)  (49, 1.6106e-04)  (50, 1.5580e-04)  (51, 1.5224e-04)  (52, 1.4810e-04)  (53, 1.4628e-04)  (54, 1.4003e-04)  (55, 1.3839e-04)  (56, 1.3249e-04)  (57, 1.2990e-04)  (58, 1.2828e-04)  (59, 1.2692e-04)  (60, 1.1697e-04)  (61, 1.1793e-04)  (62, 1.1503e-04)  (63, 1.1301e-04)  (64, 1.1018e-04)  (65, 1.0758e-04)  (66, 1.0388e-04)  (67, 1.0325e-04)  (68, 1.0080e-04)  (69, 9.6817e-05)  (70, 9.5354e-05)  (71, 9.5078e-05)  (72, 9.1490e-05)  (73, 9.1017e-05)  (74, 8.9734e-05)  (75, 8.6911e-05)  (76, 8.4727e-05)  (77, 8.4124e-05)  (78, 8.0172e-05)  (79, 8.1073e-05)  (80, 7.8172e-05)  (81, 7.7536e-05)  (82, 7.5402e-05)  (83, 7.5817e-05)  (84, 7.2949e-05)  (85, 7.2225e-05)  (86, 7.1695e-05)  (87, 6.9867e-05)  (88, 6.8223e-05)  (89, 6.7618e-05)  (90, 6.6552e-05)  (91, 6.6013e-05)  (92, 6.3721e-05)  (93, 6.3391e-05)  (94, 6.2176e-05)  (95, 6.1794e-05)  (96, 5.9745e-05)  (97, 5.9760e-05)  (98, 5.8976e-05)  (99, 5.6242e-05)  (100, 5.6668e-05)  (101, 5.6184e-05)  (102, 5.4527e-05)  (103, 5.4360e-05)  (104, 5.3111e-05)  (105, 5.2982e-05)  (106, 5.1540e-05)  (107, 5.1447e-05)  (108, 5.0474e-05)  (109, 5.0387e-05)  (110, 4.7690e-05)  (111, 4.8834e-05)  (112, 4.7928e-05)  (113, 4.7308e-05)  (114, 4.6656e-05)  (115, 4.6605e-05)  (116, 4.5382e-05)  (117, 4.4884e-05)  (118, 4.4409e-05)  (119, 4.3939e-05)  (120, 4.3246e-05)  (121, 4.2582e-05)  (122, 4.2649e-05)  (123, 4.1961e-05)  (124, 4.0791e-05)  (125, 4.0691e-05)  (126, 4.0057e-05)  (127, 4.0191e-05)  (128, 3.9552e-05)  (129, 3.8977e-05)  (130, 3.8484e-05)  (131, 3.8355e-05)  (132, 3.7126e-05)  (133, 3.7243e-05)  (134, 3.6806e-05)  (135, 3.6431e-05)  (136, 3.5840e-05)  (137, 3.5535e-05)  (138, 3.5192e-05)  (139, 3.5053e-05)  (140, 3.4516e-05)  (141, 3.4128e-05)  (142, 3.3561e-05)  (143, 3.3596e-05)  (144, 3.2818e-05)  (145, 3.2613e-05)  (146, 3.2124e-05)  (147, 3.2071e-05)  (148, 3.1349e-05)  (149, 3.0969e-05)  (150, 3.1129e-05)  (151, 3.0698e-05)  (152, 3.0569e-05)  (153, 3.0378e-05)  (154, 2.9946e-05)  (155, 2.9840e-05)  (156, 2.9305e-05)  (157, 2.9048e-05)  (158, 2.8945e-05)  (159, 2.8232e-05)  (160, 2.8031e-05)  (161, 2.8240e-05)  (162, 2.7229e-05)  (163, 2.7634e-05)  (164, 2.7269e-05)  (165, 2.7096e-05)  (166, 2.6623e-05)  (167, 2.6061e-05)  (168, 2.6338e-05)  (169, 2.5795e-05)  (170, 2.5907e-05)  (171, 2.5577e-05)  (172, 2.5308e-05)  (173, 2.5042e-05)  (174, 2.4946e-05)  (175, 2.4946e-05)  (176, 2.4615e-05)  (177, 2.4365e-05)  (178, 2.3981e-05)  (179, 2.3890e-05)  (180, 2.3444e-05)  (181, 2.3619e-05)  (182, 2.3556e-05)  (183, 2.3191e-05)  (184, 2.2713e-05)  (185, 2.2912e-05)  (186, 2.2604e-05)  (187, 2.2560e-05)  (188, 2.2292e-05)  (189, 2.2188e-05)  (190, 2.1733e-05)  (191, 2.1725e-05)  (192, 2.1475e-05)  (193, 2.1183e-05)  (194, 2.1236e-05)  (195, 2.1022e-05)  (196, 2.0645e-05)  (197, 2.0388e-05)  (198, 2.0566e-05)  (199, 2.0556e-05)  (200, 2.0275e-05)  };
		\addlegendentry{\eqref{eq:logarithmic_trafo_comb} logarithmic transformation with $\bm\eta=\mathbf 4$};
		\addplot[mark options={solid},red!50!blue,mark=square,mark size=1.0] coordinates {
			(8, 1.0903e-01)  (9, 8.5285e-02)  (10, 6.9509e-02)  (11, 6.8490e-02)  (12, 4.2589e-02)  (13, 4.2518e-02)  (14, 3.7084e-02)  (15, 2.8734e-02)  (16, 2.2158e-02)  (17, 2.1968e-02)  (18, 1.6464e-02)  (19, 1.6517e-02)  (20, 1.2022e-02)  (21, 9.8255e-03)  (22, 9.4716e-03)  (23, 9.1362e-03)  (24, 6.3548e-03)  (25, 5.3711e-03)  (26, 5.3896e-03)  (27, 4.8351e-03)  (28, 3.8586e-03)  (29, 3.7909e-03)  (30, 2.7118e-03)  (31, 2.6845e-03)  (32, 2.2263e-03)  (33, 2.1072e-03)  (34, 2.0971e-03)  (35, 1.6801e-03)  (36, 1.1629e-03)  (37, 1.1777e-03)  (38, 1.1814e-03)  (39, 1.1560e-03)  (40, 8.0846e-04)  (41, 8.3159e-04)  (42, 6.2865e-04)  (43, 6.3117e-04)  (44, 5.7255e-04)  (45, 4.7174e-04)  (46, 4.7109e-04)  (47, 4.7186e-04)  (48, 3.5188e-04)  (49, 3.0800e-04)  (50, 2.6006e-04)  (51, 2.5809e-04)  (52, 2.4599e-04)  (53, 2.4549e-04)  (54, 2.0289e-04)  (55, 1.8052e-04)  (56, 1.3333e-04)  (57, 1.3245e-04)  (58, 1.3288e-04)  (59, 1.3231e-04)  (60, 9.8355e-05)  (61, 9.6220e-05)  (62, 9.8180e-05)  (63, 7.9267e-05)  (64, 6.8601e-05)  (65, 6.3768e-05)  (66, 5.4110e-05)  (67, 5.4071e-05)  (68, 5.3410e-05)  (69, 5.2664e-05)  (70, 4.1928e-05)  (71, 4.2062e-05)  (72, 2.8397e-05)  (73, 2.8258e-05)  (74, 2.8299e-05)  (75, 2.7147e-05)  (76, 2.7167e-05)  (77, 2.3017e-05)  (78, 2.0832e-05)  (79, 2.0920e-05)  (80, 1.6159e-05)  (81, 1.4093e-05)  (82, 1.4283e-05)  (83, 1.4203e-05)  (84, 1.1181e-05)  (85, 1.0994e-05)  (86, 1.0980e-05)  (87, 1.1002e-05)  (88, 9.0981e-06)  (89, 9.1090e-06)  (90, 6.5598e-06)  (91, 5.8130e-06)  (92, 5.8199e-06)  (93, 5.8529e-06)  (94, 5.8240e-06)  (95, 5.8021e-06)  (96, 4.6604e-06)  (97, 4.6565e-06)  (98, 4.1908e-06)  (99, 3.4086e-06)  (100, 2.9619e-06)  (101, 2.9654e-06)  (102, 2.8656e-06)  (103, 2.8774e-06)  (104, 2.4692e-06)  (105, 2.2940e-06)  (106, 2.2891e-06)  (107, 2.2808e-06)  (108, 1.8568e-06)  (109, 1.8615e-06)  (110, 1.4719e-06)  (111, 1.4689e-06)  (112, 1.1652e-06)  (113, 1.1719e-06)  (114, 1.1541e-06)  (115, 1.1621e-06)  (116, 1.1693e-06)  (117, 1.0057e-06)  (118, 1.0006e-06)  (119, 9.6801e-07)  (120, 6.7956e-07)  (121, 6.5234e-07)  (122, 6.5671e-07)  (123, 6.4960e-07)  (124, 6.4161e-07)  (125, 6.3611e-07)  (126, 5.6220e-07)  (127, 5.6227e-07)  (128, 5.4063e-07)  (129, 5.3734e-07)  (130, 5.4170e-07)  (131, 5.4147e-07)  (132, 5.0882e-07)  (133, 5.0451e-07)  (134, 5.0154e-07)  (135, 4.9215e-07)  (136, 4.8591e-07)  (137, 4.8762e-07)  (138, 4.7111e-07)  (139, 4.7391e-07)  (140, 4.4184e-07)  (141, 4.4080e-07)  (142, 4.3611e-07)  (143, 4.4730e-07)  (144, 4.0751e-07)  (145, 3.9711e-07)  (146, 3.9793e-07)  (147, 3.9247e-07)  (148, 3.8868e-07)  (149, 3.8853e-07)  (150, 3.7702e-07)  (151, 3.7550e-07)  (152, 3.6974e-07)  (153, 3.6487e-07)  (154, 3.5097e-07)  (155, 3.4218e-07)  (156, 3.3377e-07)  (157, 3.2976e-07)  (158, 3.3385e-07)  (159, 3.3286e-07)  (160, 3.2023e-07)  (161, 3.1807e-07)  (162, 3.0701e-07)  (163, 3.0350e-07)  (164, 3.0027e-07)  (165, 2.9263e-07)  (166, 2.9425e-07)  (167, 2.9168e-07)  (168, 2.7781e-07)  (169, 2.7428e-07)  (170, 2.7315e-07)  (171, 2.6966e-07)  (172, 2.6797e-07)  (173, 2.6685e-07)  (174, 2.6083e-07)  (175, 2.5388e-07)  (176, 2.4941e-07)  (177, 2.4702e-07)  (178, 2.4638e-07)  (179, 2.4388e-07)  (180, 2.3432e-07)  (181, 2.3270e-07)  (182, 2.3073e-07)  (183, 2.2915e-07)  (184, 2.2267e-07)  (185, 2.2044e-07)  (186, 2.1787e-07)  (187, 2.1764e-07)  (188, 2.1466e-07)  (189, 2.0953e-07)  (190, 2.0555e-07)  (191, 2.0552e-07)  (192, 1.9959e-07)  (193, 1.9913e-07)  (194, 1.9983e-07)  (195, 1.9517e-07)  (196, 1.9116e-07)  (197, 1.9040e-07)  (198, 1.8613e-07)  (199, 1.8558e-07)  (200, 1.8133e-07)  };
		\addlegendentry{\eqref{eq:logarithmic_trafo_comb} logarithmic transformation with $\bm\eta=\mathbf 6$};
		\end{axis}
		\end{tikzpicture}
	\end{minipage}
\begin{minipage}[b]{.49\linewidth}
		\centering
		\begin{tikzpicture}[baseline,scale=0.65]
		\begin{axis}[
		ymode = log,
		enlargelimits=false,
		xmin=8, xmax=100, ymin=1e-04, ymax=1e-0,
		ytick={1e-1,1e-2,1e-3,1e-4,1e-5,1e-6},grid=both, 
		xlabel={$N$}, 
		ylabel={$\varepsilon_{\infty}$},
		legend style={at={(0.5,1.05)}, anchor=south,legend columns=1,legend cell align=left, font=\small,  
		},
		xminorticks=false,
		yminorticks=false
		]
		\addplot[mark options={solid},red!25!yellow,mark=triangle,mark size=1.5] coordinates {
			(8, 1.4082e-01)  (9, 1.3944e-01)  (10, 1.2593e-01)  (11, 1.3195e-01)  (12, 1.2559e-01)  (13, 1.2120e-01)  (14, 1.1700e-01)  (15, 1.1041e-01)  (16, 1.0625e-01)  (17, 1.0537e-01)  (18, 1.0588e-01)  (19, 1.0243e-01)  (20, 1.0192e-01)  (21, 9.4173e-02)  (22, 9.3574e-02)  (23, 9.3856e-02)  (24, 8.5224e-02)  (25, 9.1155e-02)  (26, 8.1913e-02)  (27, 8.4325e-02)  (28, 7.8682e-02)  (29, 8.4704e-02)  (30, 8.1016e-02)  (31, 7.7720e-02)  (32, 7.9508e-02)  (33, 7.7787e-02)  (34, 7.6247e-02)  (35, 6.7897e-02)  (36, 7.5090e-02)  (37, 6.9546e-02)  (38, 6.9953e-02)  (39, 6.9874e-02)  (40, 6.8250e-02)  (41, 6.7790e-02)  (42, 6.6948e-02)  (43, 6.5529e-02)  (44, 6.4932e-02)  (45, 6.6525e-02)  (46, 6.2358e-02)  (47, 6.1502e-02)  (48, 6.3003e-02)  (49, 5.9546e-02)  (50, 6.2895e-02)  (51, 6.2127e-02)  (52, 6.1232e-02)  (53, 6.2054e-02)  (54, 5.6843e-02)  (55, 6.0694e-02)  (56, 5.8763e-02)  (57, 5.8529e-02)  (58, 5.9155e-02)  (59, 5.7849e-02)  (60, 5.7613e-02)  (61, 5.7514e-02)  (62, 5.5141e-02)  (63, 5.8047e-02)  (64, 5.5279e-02)  (65, 5.3323e-02)  (66, 5.5033e-02)  (67, 5.5775e-02)  (68, 5.3109e-02)  (69, 5.5608e-02)  (70, 5.2199e-02)  (71, 4.9543e-02)  (72, 5.2716e-02)  (73, 5.1184e-02)  (74, 5.1227e-02)  (75, 4.8539e-02)  (76, 5.2330e-02)  (77, 4.9999e-02)  (78, 5.0299e-02)  (79, 5.1780e-02)  (80, 4.9209e-02)  (81, 5.0037e-02)  (82, 5.0456e-02)  (83, 4.6924e-02)  (84, 4.8883e-02)  (85, 4.9041e-02)  (86, 4.6809e-02)  (87, 4.7288e-02)  (88, 4.8307e-02)  (89, 4.7553e-02)  (90, 4.6827e-02)  (91, 4.7652e-02)  (92, 4.7325e-02)  (93, 4.7024e-02)  (94, 4.5347e-02)  (95, 4.3777e-02)  (96, 4.4126e-02)  (97, 4.4912e-02)  (98, 4.6110e-02)  (99, 4.6472e-02)  (100, 4.4287e-02)
		};\addlegendentry{\eqref{eq:sine_trafo} sine transformation};
		\addplot[mark options={solid}, blue, mark=o, mark size=1.0] coordinates {
			(8, 8.8754e-02)  (9, 8.7861e-02)  (10, 8.3030e-02)  (11, 8.1347e-02)  (12, 7.7534e-02)  (13, 7.2512e-02)  (14, 7.1221e-02)  (15, 6.9284e-02)  (16, 6.5260e-02)  (17, 6.0823e-02)  (18, 6.2172e-02)  (19, 5.9757e-02)  (20, 6.0699e-02)  (21, 5.8489e-02)  (22, 5.5764e-02)  (23, 5.4150e-02)  (24, 5.4396e-02)  (25, 5.2711e-02)  (26, 5.1309e-02)  (27, 5.1636e-02)  (28, 5.1364e-02)  (29, 4.8420e-02)  (30, 4.9890e-02)  (31, 4.7346e-02)  (32, 4.6684e-02)  (33, 4.8587e-02)  (34, 4.5908e-02)  (35, 4.4563e-02)  (36, 4.3068e-02)  (37, 4.3762e-02)  (38, 4.2703e-02)  (39, 4.1731e-02)  (40, 4.3847e-02)  (41, 4.1571e-02)  (42, 4.0274e-02)  (43, 4.1197e-02)  (44, 3.9672e-02)  (45, 3.9584e-02)  (46, 4.0278e-02)  (47, 3.7582e-02)  (48, 3.9561e-02)  (49, 3.7921e-02)  (50, 3.7942e-02)  (51, 3.7950e-02)  (52, 3.6909e-02)  (53, 3.5960e-02)  (54, 3.6807e-02)  (55, 3.6656e-02)  (56, 3.4564e-02)  (57, 3.5216e-02)  (58, 3.4080e-02)  (59, 3.5163e-02)  (60, 3.5329e-02)  (61, 3.4647e-02)  (62, 3.4684e-02)  (63, 3.3555e-02)  (64, 3.3465e-02)  (65, 3.3999e-02)  (66, 3.2565e-02)  (67, 3.2931e-02)  (68, 3.2373e-02)  (69, 3.2698e-02)  (70, 3.1628e-02)  (71, 3.3171e-02)  (72, 3.1470e-02)  (73, 3.0987e-02)  (74, 3.0337e-02)  (75, 3.0001e-02)  (76, 3.0693e-02)  (77, 2.9781e-02)  (78, 3.0462e-02)  (79, 2.9817e-02)  (80, 2.9275e-02)  (81, 3.0519e-02)  (82, 2.9682e-02)  (83, 2.9482e-02)  (84, 2.8347e-02)  (85, 2.9201e-02)  (86, 2.9260e-02)  (87, 2.8955e-02)  (88, 2.9283e-02)  (89, 2.8130e-02)  (90, 2.7958e-02)  (91, 2.8064e-02)  (92, 2.7833e-02)  (93, 2.7437e-02)  (94, 2.7932e-02)  (95, 2.7302e-02)  (96, 2.7451e-02)  (97, 2.7557e-02)  (98, 2.6918e-02)  (99, 2.7053e-02)	(100, 2.6876e-02)
		};\addlegendentry{\eqref{eq:logarithmic_trafo_comb} logarithmic transformation with $\bm\eta=\mathbf 2$};
		\addplot[mark options={solid},red!75!yellow,mark=x,mark size=1.0] coordinates {
			(8, 8.3373e-02)  (9, 8.4999e-02)  (10, 7.3285e-02)  (11, 6.9481e-02)  (12, 3.8419e-02)  (13, 3.8579e-02)  (14, 3.8679e-02)  (15, 3.7163e-02)  (16, 2.7124e-02)  (17, 2.7093e-02)  (18, 1.7991e-02)  (19, 2.0329e-02)  (20, 1.6914e-02)  (21, 1.7308e-02)  (22, 1.6583e-02)  (23, 1.6861e-02)  (24, 9.5070e-03)  (25, 9.1242e-03)  (26, 9.6162e-03)  (27, 8.8766e-03)  (28, 8.3628e-03)  (29, 8.6707e-03)  (30, 6.7219e-03)  (31, 7.3097e-03)  (32, 5.7095e-03)  (33, 5.8397e-03)  (34, 5.9909e-03)  (35, 5.9131e-03)  (36, 3.8582e-03)  (37, 3.9243e-03)  (38, 3.8957e-03)  (39, 3.9192e-03)  (40, 3.2632e-03)  (41, 3.2210e-03)  (42, 3.1133e-03)  (43, 2.9698e-03)  (44, 3.1110e-03)  (45, 2.8537e-03)  (46, 2.9982e-03)  (47, 3.0039e-03)  (48, 1.8203e-03)  (49, 1.8384e-03)  (50, 1.7646e-03)  (51, 1.8240e-03)  (52, 1.7563e-03)  (53, 1.8020e-03)  (54, 1.4993e-03)  (55, 1.5144e-03)  (56, 1.4163e-03)  (57, 1.4180e-03)  (58, 1.4459e-03)  (59, 1.4459e-03)  (60, 1.1168e-03)  (61, 1.1089e-03)  (62, 1.0896e-03)  (63, 1.1075e-03)  (64, 9.4483e-04)  (65, 9.1116e-04)  (66, 9.4141e-04)  (67, 9.7008e-04)  (68, 9.4664e-04)  (69, 9.9667e-04)  (70, 9.5531e-04)  (71, 9.6105e-04)  (72, 5.9995e-04)  (73, 6.0738e-04)  (74, 5.9555e-04)  (75, 5.9066e-04)  (76, 5.8474e-04)  (77, 5.9962e-04)  (78, 6.1136e-04)  (79, 6.0512e-04)  (80, 5.1099e-04)  (81, 4.8710e-04)  (82, 4.9630e-04)  (83, 4.8112e-04)  (84, 4.5894e-04)  (85, 4.6841e-04)  (86, 4.6325e-04)  (87, 4.4650e-04)  (88, 4.5351e-04)  (89, 4.3343e-04)  (90, 3.9565e-04)  (91, 3.9774e-04)  (92, 3.8874e-04)  (93, 3.9003e-04)  (94, 3.9614e-04)  (95, 3.8185e-04)  (96, 2.6792e-04)  (97, 2.5136e-04)  (98, 2.7313e-04)  (99, 2.6518e-04)   (100, 2.5822e-04) 
		};\addlegendentry{\eqref{eq:logarithmic_trafo_comb} logarithmic transformation with $\bm\eta=\mathbf 4$};
		\end{axis}
	\end{tikzpicture}
	\end{minipage} 
	\caption{Comparison of discrete $\ell_\infty$-approximation error $\varepsilon_{\infty}$ as given in \eqref{eq:Discrete_ellinfty_error} of the $d$-dimensional test function \eqref{eq:exemplary_h_mult}.
	On the left, in dimension $d=2$ we compare the sine transformation \eqref{eq:sine_trafo} and the logarithmic transformation \eqref{eq:logarithmic_trafo_comb} with $\bm\eta \in \{ \mathbf 2, \mathbf 4, \mathbf 6 \}$.
	On the right, in dimension $d=5$ we compare the sine transformation \eqref{eq:sine_trafo} and the logarithmic transformation \eqref{eq:logarithmic_trafo_comb} with $\bm\eta \in \{ \mathbf 2, \mathbf 4 \}$}
	\label{fig:numeric_test_multivar}
\end{figure}
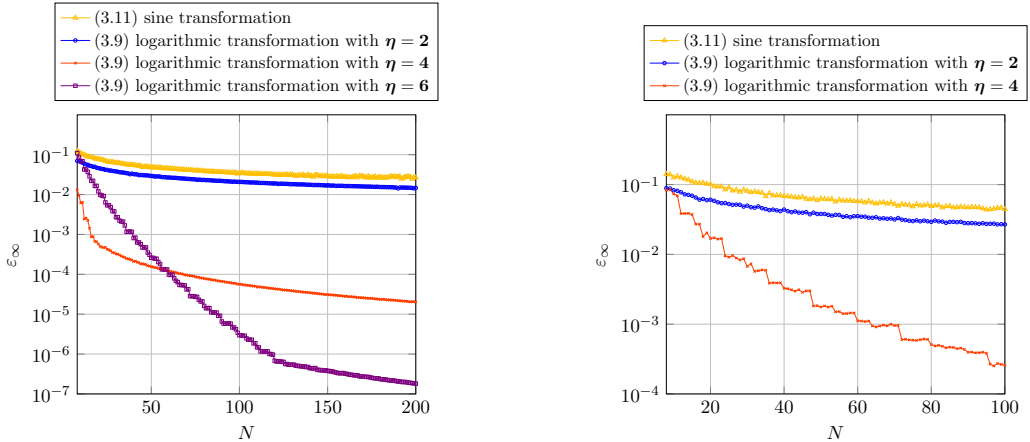

\section{Conclusion}
We consider functions $h\in L_{2}\left(\left[-\frac{1}{2},\frac{1}{2}\right]^d,\omega(\circ,\bm\mu)\right) \cap \mathcal{C}^{m}\left(\left[-\frac{1}{2},\frac{1}{2}\right]^d\right)$ with a parameterized weight function $\omega(\circ,\bm\mu):\left[-\frac{1}{2},\frac{1}{2}\right]^d\to[0,\infty), \bm\mu\in\mathbb{R}_{+}^d$ and discussed a particular periodization strategy that transforms them into functions $f$ that are continuously extendable on the torus $\mathbb T^d$.
The applied multivariate torus-to-cube transformations ${\psi: \left[-\frac{1}{2},\frac{1}{2}\right]^d\to\left[-\frac{1}{2},\frac{1}{2}\right]^d}$, in particular parameterized torus-to-cube transformations $\psi(\circ,\bm\eta)$ with $\bm\eta\in\mathbb{R}_{+}^d$,
let us control in which Sobolev space $\mathcal H^{m}(\mathbb{T}^d),m\in\mathbb{N}$ a function $h$ defined on the cube $\left[-\frac{1}{2},\frac{1}{2}\right]^d$ is transformed into.
Due to the embedding of the Sobolev spaces $\mathcal H^{m}(\mathbb{T}^d)$ into the Wiener algebra $\mathcal{A}(\mathbb{T}^d)$ of functions with absolutely summable Fourier coefficients, we have information on the rate of decay of the Fourier coefficients $\hat f_{\mathbf k}$ and $\hat h_{\mathbf k}$ without having to calculate them -- which in a lot of cases is not possible in the first place.
The $L_2$- and $L_{\infty}$-approximation error bounds for smooth functions on the torus $\mathbb{T}^d$ proposed in \cite[Theorem~2.30]{volkmerdiss} and \cite[Theorem~3.3]{KaPoVo13} can be adjusted for classes of non-periodic functions defined on $\left[-\frac{1}{2},\frac{1}{2}\right]^d$ by means of the inverse transformation 
${\psi^{-1}(\circ,\bm\eta)}$.
Furthermore, only slight modifications are necessary to incorporate such transformations into the efficient algorithms based on single reconstructing {rank-$1$} lattices for the evaluation and the reconstruction of transformed multivariate trigonometric polynomials presented in \cite[Algorithm~3.1 and 3.2]{kaemmererdiss}.
Algorithms based on multiple reconstructing {rank-$1$} lattices \cite{Kae16} and sparse fast Fourier transformations \cite{PoVo14} can be adjusted, too, but weren't discussed in depth in this work.

Our numerical tests in up to dimension $d=5$ show that these algorithms are still working within the proposed upper bounds for the approximation error. 
These tests also highlight the limited smoothness effect of static torus-to-cube transformations $\psi$ as opposed to specific parameterized torus-to-cube transformations $\left\{\psi(\circ,\bm\eta)\right\}_{\bm\eta\in\mathbb{R}_{+}^{d}}$ with which we can control the eventual smoothening effect.

\section*{Acknowledgements}
The authors thank the referees for their valuable suggestions and remarks.
The first named author gratefully acknowledges the support by the funding of the European Union and the
Free State of Saxony (ESF).

\bibliographystyle{spbasic}      


\end{document}